\documentclass[12pt,svgnames]{article}
\usepackage{amsmath}
\usepackage{amssymb}
\usepackage{amsthm}
\usepackage{enumerate}
\usepackage{wasysym}
\usepackage{dsfont}
\usepackage{comment}
\usepackage{tikz}
\usepackage{tikz-cd}
\usepackage{extarrows}
\usepackage{mathrsfs}
\usepackage{hyperref}
\usepackage{fullpage}
\usepackage{enumitem}
\usepackage{titlesec}
\usepackage{stmaryrd}
\usepackage[all]{xy}
\hypersetup{
    linkbordercolor = {PaleTurquoise},
    citebordercolor = {PaleGreen},
}

\titleformat{\section}{\normalsize\bfseries}{\thesection}{1em}{}
\titleformat{\subsection}{\normalsize\bfseries}{\thesubsection}{1em}{}

\numberwithin{equation}{subsection}

\theoremstyle{plain}

\theoremstyle{definition}

\theoremstyle{definition}

\newtheorem{defn}[subsection]{Definition}
\newtheorem{para}[subsection]{}
\newtheorem{egg}[subsection]{Example}
\newtheorem{rmk}[subsection]{Remark}

\newtheorem*{assumption*}{Assumption}

\newtheorem{defn_sub}[subsubsection]{Definition}
\newtheorem{para_sub}[subsubsection]{}
\newtheorem{egg_sub}[subsubsection]{Example}
\newtheorem{rmk_sub}[subsubsection]{Remark}
\newtheorem{assumption_sub}[subsubsection]{Assumption}

\theoremstyle{plain}

\newtheorem{prop}[subsection]{Proposition}
\newtheorem{theo}[subsection]{Theorem}
\newtheorem{lem}[subsection]{Lemma}
\newtheorem{cor}[subsection]{Corollary}

\newtheorem*{claim*}{Claim}
\newtheorem*{just*}{Justification}
\newtheorem*{lem*}{Lemma}
\newtheorem*{prop*}{Proposition}

\newtheorem{prop_sub}[subsubsection]{Proposition}
\newtheorem{theo_sub}[subsubsection]{Theorem}
\newtheorem{lem_sub}[subsubsection]{Lemma}
\newtheorem{cor_sub}[subsubsection]{Corollary}

\newcommand{\X}{\mathscr{X}}
\newcommand{\Y}{\mathscr{Y}}
\newcommand{\Z}{\mathscr{Z}}

\newcommand{\J}{\mathscr{J}}
\newcommand{\V}{\mathscr{V}}

\newcommand{\JSig}{\Sig_{\underJ}}
\newcommand{\ob}{\ensuremath{\operatorname{\textnormal{\textsf{ob}}}}}
\newcommand{\End}{\mathsf{End}}
\newcommand{\di}{\diamondsuit}
\newcommand{\tensor}{\otimes}
\newcommand{\Mnd}{\mathsf{Mnd}}

\newcommand{\Lan}{\mathsf{Lan}}
\newcommand{\Cat}{\text{-}\mathsf{Cat}}

\newcommand{\Cocts}{\text{-}\mathsf{Cocts}}
\newcommand{\underJ}{\kern -0.5ex \mathscr{J}}
\newcommand{\Set}{\mathsf{Set}}
\newcommand{\ev}{\mathsf{ev}}
\newcommand{\op}{\mathsf{op}}
\newcommand{\Th}{\mathsf{Th}}
\newcommand{\T}{\mathbb{T}}
\newcommand{\C}{\mathscr{C}}

\newcommand{\inj}{\mathsf{in}}
\newcommand{\B}{\mathscr{B}}
\newcommand{\A}{\mathscr{A}}
\newcommand{\U}{\mathcal{U}}
\newcommand{\Mod}{\text{-}\mathsf{Mod}}

\newcommand{\y}{\mathsf{y}}

\newcommand{\Alg}{\text{-}\mathsf{Alg}}

\newcommand{\F}{\mathcal{F}}

\newcommand{\Lbar}{\bar{L}}
\newcommand{\lambdabar}{\bar{\lambda}}
\newcommand{\K}{\mathscr{K}}
\newcommand{\M}{\mathscr{M}}
\newcommand{\limit}{\operatorname{\mathsf{lim}}}

\newcommand{\terml}{t}
\newcommand{\termr}{u}
\newcommand{\termlbar}{\overline{\terml}}
\newcommand{\termrbar}{\overline{\termr}}

\newcommand{\II}{\mathbb{I}}

\newcommand{\la}{\langle}
\newcommand{\ra}{\rangle}
\newcommand{\Ran}{\mathsf{Ran}}

\newcommand{\FinCard}{\mathsf{FinCard}}
\newcommand{\N}{\mathbb{N}}

\newcommand{\Monadic}{\mathsf{Monadic}}

\newcommand{\Mon}{\mathsf{Mon}}
\newcommand{\colim}{\operatorname{\mathsf{colim}}}
\newcommand{\G}{\mathscr{G}}
\newcommand{\Iso}{\mathsf{Iso}}

\newcommand{\W}{\mathcal{W}}
\newcommand{\Sig}{\mathsf{Sig}}
\newcommand{\D}{\mathscr{D}}
\newcommand{\bbD}{\mathbb{D}}

\newcommand{\scrT}{\mathscr{T}}

\newcommand{\E}{\mathscr{E}}
\newcommand{\MM}{\mathbb{M}}
\newcommand{\All}{\mathsf{All}}
\newcommand{\Ord}{\mathsf{Ord}}
\newcommand{\CAT}{\text{-}\mathsf{CAT}}
\newcommand{\ALG}{\mathsf{Alg}}

\newcommand{\Hom}{\mathsf{Hom}}

\newcommand{\SF}{\mathsf{SF}}
\newcommand{\VProf}{\V\textnormal{-}\mathsf{Prof}}
\newcommand{\Mat}{\mathsf{Mat}}
\newcommand{\aff}{\mathsf{aff}}
\newcommand{\Graph}{\text{-}\mathsf{Graph}}
\newcommand{\VMat}{\V\text{-}\mathsf{Mat}}
\newcommand{\llb}{\llbracket}
\newcommand{\rrb}{\rrbracket}
\def\profto{\mathop{\rlap{\hspace{1.4ex}$|$}{\longrightarrow}}\nolimits}

\begin{document}

\title{\Large \textbf{Presentations and algebraic colimits of\\ enriched monads for a subcategory of arities}}
\author{Rory B. B. Lucyshyn-Wright\let\thefootnote\relax\thanks{We acknowledge the support of the Natural Sciences and Engineering Research Council of Canada (NSERC), [funding reference numbers RGPIN-2019-05274, RGPAS-2019-00087, DGECR-2019-00273].  Cette recherche a été financée par le Conseil de recherches en sciences naturelles et en génie du Canada (CRSNG), [numéros de référence RGPIN-2019-05274, RGPAS-2019-00087, DGECR-2019-00273].} \medskip \\ Jason Parker \medskip
\\
\small Brandon University, Brandon, Manitoba, Canada}
\date{}

\maketitle

\begin{abstract}
We develop a general framework for studying signatures, presentations, and algebraic colimits of enriched monads for a subcategory of arities, even when the base of enrichment $\V$ is not locally presentable. When $\V$ satisfies the weaker requirement of local boundedness, the resulting framework is sufficiently general to apply to the $\Phi$-accessible monads of Lack and Rosick\'y and the $\J$-ary monads of the first author, while even without local boundedness our framework captures in full generality the presentations of strongly finitary monads of Lack and Kelly as well as Wolff's presentations of $\V$-categories by generators and relations.  Given any small subcategory of arities $j : \J \hookrightarrow \C$ in an enriched category $\C$, satisfying certain assumptions, we prove results on the existence of free $\J$-ary monads, the monadicity of $\J$-ary monads over $\J$-signatures, and the existence of algebraic colimits of $\J$-ary monads.  We study a notion of presentation for $\J$-ary monads and show that every such presentation presents a $\J$-ary monad. Certain of our results generalize earlier results of Kelly, Power, and Lack for finitary enriched monads in the locally finitely presentable setting, as well as analogous results of Kelly and Lack for strongly finitary monads on cartesian closed categories. Our main results hold for a wide class of subcategories of arities in \emph{locally bounded} enriched categories.
\end{abstract}

\section{Introduction}
When the basic concepts of universal algebra are defined in the usual way via \textit{signatures} or \textit{similarity types}, i.e. sets of operation symbols with specified arities, there arises the well-known complication that a given pair of varieties of algebras over two different signatures may turn out to be isomorphic as concrete categories over $\Set$, as exemplified by the fact that the familiar notion of group admits multiple distinct presentations.  An elegant solution to this problem was given by Lawvere \cite{Law:PhD}, whose \textit{algebraic theories} (or \textit{Lawvere theories}) classify varieties up to isomorphism (\textit{qua} concrete categories over $\Set$).  Lawvere was therefore able to regard systems of equations over signatures as \textit{presentations} of Lawvere theories \cite[II.2]{Law:PhD}, the latter providing a presentation-independent notion of algebraic theory.  By way of the equivalence between Lawvere theories and finitary monads on $\Set$, Lawvere's work provides also an elegant approach to presentations of finitary monads that has since been generalized to the setting of enriched category theory by Kelly and Power \cite{KellyPower}, and the present paper is a contribution to this line of generalization.

In \cite{KellyPower}, Kelly and Power showed that if $\C$ is a locally finitely presentable $\V$-category over a locally finitely presentable symmetric monoidal closed category $\V$, then the forgetful functor $\W : \Mnd_f(\C) \to \End_f(\C)$ from finitary $\V$-monads on $\C$ to finitary $\V$-endofunctors on $\C$ is monadic, and the forgetful functor $\U : \Mnd_f(\C) \to \Sig_f(\C)$ from finitary $\V$-monads on $\C$ to finitary signatures in $\C$ is of descent type (see \cite[5.1]{KellyPower}), which entails in particular that $\U$ has a left adjoint and that every finitary $\V$-monad on $\C$ has a \emph{presentation} given by abstract operations and equations. In \cite{KeLa}, Kelly and Lack established analogous results for \textit{strongly finitary $\V$-monads} on cartesian closed categories $\V$, while Borceux and Day \cite[2.6.2]{BorceuxDay} had earlier exposited and credited to Kelly an analogue of one of these results, formulated in terms of enriched theories in the slightly more general setting of their \textit{$\pi$-categories}.  In \cite{Lackmonadicity}, Lack provided a further substantial advance in this area by showing that $\U$ is in fact \emph{monadic} (see \cite[Corollary 3]{Lackmonadicity}).  A different approach to enriched equational presentations was introduced by Fiore and Hur \cite{termequationalsystems}, with existence results under hypotheses involving the preservation of colimits of chains indexed by a given limit ordinal.

An area of recent interest in enriched category theory has been the study of classes of enriched monads and theories defined relative to a given class of weights or a given \textit{subcategory of arities}, i.e., a full subcategory that is dense in the enriched sense.  In the setting of ordinary $\Set$-enriched category theory, the latter theme goes back to Linton \cite{Lin:OutlFuncSem} and includes the theories and monads with arities of Berger, Melli\`es, and Weber \cite{BMW}.  In \cite{LR}, Lack and Rosick\'y defined the concept of \textit{$\Phi$-accessible $\V$-monad} for a class of weights $\Phi$ satisfying their Axiom A, along with the notion of \textit{Lawvere $\Phi$-theory}.  In the 2016 paper \cite{EAT}, the first author studied enriched \textit{$\J$-theories} and \textit{$\J$-ary} monads for a \textit{system of arities} $\J \hookrightarrow \V$, which is a (possibly large) subcategory of arities closed under the monoidal structure.  The 2019 paper \cite{BourkeGarner} of Bourke and Garner studied monads and theories relative to a small subcategory of arities $\A \hookrightarrow \C$ in a locally presentable $\V$-category $\C$ over a locally presentable closed category $\V$.  While the latter work is not primarily concerned with presentations, Bourke and Garner show that their $\A$\emph{-nervous monads} are monadic over $\A$-signatures, and that the category of $\A$-nervous monads is cocomplete (even locally presentable), with small colimits therein being \emph{algebraic} (see \cite[Proposition 31, Theorem 38]{BourkeGarner}).  However, since the latter work is situated in the locally presentable setting, it does not subsume any of the following frameworks in full generality: (1) the $\Phi$-accessible monads of Lack and Rosick\'y \cite{LR}; (2) the $\J$-ary monads of \cite{EAT}; (3) the enriched theories of Borceux and Day \cite{BorceuxDay}, which are encompassed by 1 and by 2; (4) the strongly finitary $\V$-monads of Lack and Kelly \cite{KeLa}, which are encompassed by 1, 2, and 3.  Indeed, in settings (1)-(4), $\V_0$ is not required to be locally presentable.  Moreover, it is clear from the work of Kelly \cite[Chapter 6]{Kelly}, the paper of Lack and Rosick\'y \cite{LR}, and the present paper that various methods of enriched-categorical algebra are applicable at least in the \textit{locally bounded closed categories} $\V$ of Kelly \cite[6.1]{Kelly} and, more generally, the \textit{locally bounded $\V$-categories} \cite{locbd} over such $\V$.  This ensures that these methods may be applied to the various locally bounded closed categories that provide backgrounds for mathematics and computer science (see \cite[\S 5.3]{locbd}), and which need not be locally presentable, along with the various locally bounded $\V$-categories of structures in such $\V$ (\cite[Chapter 6]{Kelly}, \cite[\S 11]{locbd}).

The first purpose of this paper is to develop a theory of presentations and colimits of enriched monads for subcategories of arities with sufficient generality to accommodate the following general classes of examples when $\V$ is locally bounded: (1) the $\Phi$-accessible $\V$-monads of Lack and Rosick\'y \cite{LR}, and (2) the $\J$-ary $\V$-monads for a small and \textit{eleutheric} system of arities $\J \hookrightarrow \V$ \cite{EAT}.  The second purpose of this paper is to ensure that the resulting theory of presentations and algebraic colimits covers in full generality the following specific settings, even when $\V$ is not locally bounded: (a) the strongly finitary $\V$-monads of Lack and Kelly \cite{KeLa} when $\V$ is a complete and cocomplete cartesian closed category $\V$ or, more generally, a $\pi$-category in the sense of Borceux and Day \cite{BorceuxDay}, and (b) Wolff's presentations of $\V$-categories by generators and relations, for an arbitrary complete and cocomplete $\V$ \cite{WolffVcat}, which (as we show in \S \ref{representablesexample}) are recovered by taking the subcategory of arities to be the Yoneda embedding for a discrete $\V$-category.

We accomplish these objectives by working with enriched monads for a suitable subcategory of arities $j : \J \hookrightarrow \C$ in a $\V$-category $\C$, where $\V$ is a complete and cocomplete symmetric monoidal closed category that need not be locally presentable. Our results apply when $\C$ is a locally bounded $\V$-category over a locally bounded closed category $\V$, and in some cases even without these assumptions, e.g. when $\V$ is a $\pi$-category.  Certain of our results generalize the results of Kelly, Power, and Lack mentioned above.

To obtain our results, we make some modest completeness and cocompleteness assumptions on the $\V$-category $\C$, and two main assumptions on the subcategory of arities $j : \J \hookrightarrow \C$. First, we generally assume that $j : \J \hookrightarrow \C$ is small and \emph{eleutheric} (cf. \cite{EAT}), which is a certain `exactness' condition that guarantees that the $\V$-endofunctors on $\C$ that are left Kan extensions along $j$ are precisely those that \emph{preserve} left Kan extensions along $j$; we call these the $\J$\emph{-ary} $\V$-endofunctors. We also assume that $j : \J \hookrightarrow \C$ satisfies a mild \emph{boundedness} condition, which we define in terms of certain notions from Kelly's classic paper \cite{Kellytrans} on transfinite constructions in category theory, where the reader can find a list of references regarding the rich history of the theme of existence of free constructions. Our main results on free $\J$-ary monads, algebraic colimits of $\J$-ary monads, and presentations of $\J$-ary monads then hold for any bounded and eleutheric subcategory of arities $j : \J \hookrightarrow \C$ in a $\V$-category $\C$ satisfying mild assumptions.

In order to further clarify the relation of this paper to the existing literature, we now further contrast this work with aspects of the recent paper of Bourke and Garner \cite{BourkeGarner}.  In the latter paper it is assumed that both $\V$ and $\C$ are locally presentable, while the present paper is applicable to a much broader class of enriched categories, including locally bounded $\V$-categories $\C$ over a locally bounded $\V$.  On the other hand, here we require that the subcategory of arities $\J \hookrightarrow \C$ be bounded and eleutheric, while Bourke and Garner allow an arbitrary small subcategory of arities $\A \hookrightarrow \C$ with $\C$ and $\V$ locally presentable.  However, in the latter locally presentable setting, every small subcategory of arities $\A$ is necessarily bounded and is actually contained in a \textit{larger} subcategory of arities $\J = \C_\alpha \hookrightarrow \C$ that is not only bounded but also \textit{eleutheric}, consisting of the (enriched) $\alpha$-presentable objects for a suitable cardinal $\alpha$ (\ref{eleuthericexamples}, \ref{boundedexamples}).  Hence, in the Bourke-Garner setting, presentations relative to $\A$ may be viewed also as presentations relative to $\J = \C_\alpha$, and the $\A$-nervous monads that they present are, in particular, $\alpha$-ary monads and, for some purposes, may be studied as such, by way of the earlier methods of Kelly-Power \cite{KellyPower} and Lack \cite{Lackmonadicity}.  In the present paper, we escape the locally presentable setting by asking for a \textit{given} bounded and eleutheric subcategory of arities $\J \hookrightarrow \C$ to play a role analogous to that of $\C_\alpha$.  It is important to note that in this general setting we can once again consider arbitrary small subcategories of arities $\A \hookrightarrow \C$ with $\A \subseteq \J$, and again any $\A$-presentation may be viewed as a $\J$-presentation and so, by the results in this paper, presents a \textit{$\J$-ary} monad.

In a subsequent paper \cite{Pres2}, we shall further explore the implications of the theory developed in this paper, and in particular we shall provide additional `user-friendly' techniques for constructing presentations of $\J$-ary monads, which will allow us to easily define many further examples of such presentations in a manner that closely resembles ordinary mathematical practice.  Furthermore, these techniques will enable us to generalize the adjunction between $\A$-\textit{pretheories} and $\V$-monads of Bourke and Garner \cite{BourkeGarner} beyond the locally presentable setting, by letting $\A \hookrightarrow \C$ be an arbitrary subcategory of arities that is contained in a given bounded and eleutheric subcategory of arities $\J \hookrightarrow \C$, as described above.

\medskip

We now provide a detailed overview of the paper. After reviewing some notation and background on enriched category theory in \S \ref{background}, in \S \ref{firstsection} we begin by defining the notion of an eleutheric subcategory of arities $j : \J \hookrightarrow \C$ in a $\V$-category $\C$ (originally defined in \cite{EAT} for $\C = \V$), and in \S \ref{Jarysection} we define the notions of $\J$\emph{-ary} $\V$-endofunctor and $\J$\emph{-ary} $\V$-monad on $\C$. In \S \ref{algfreemonads} we develop the theory of \emph{algebraically free monads} in the enriched context, generalizing aspects of Kelly's work on this topic in the ordinary $\Set$-enriched context in \cite{Kellytrans}. We begin \S \ref{freemonadsection} by defining the notion of a \emph{bounded} subcategory of arities, and we show that our running examples have this property, as does any small subcategory of arities in a locally bounded $\V$-category over a locally bounded closed category $\V$. We then prove our first main results in \ref{mainalgfreetheorem} and \ref{Whasleftadjoint}, which show that if $j : \J \hookrightarrow \C$ is a bounded subcategory of arities in a cocomplete and cotensored $\V$-category $\C$, then the forgetful functor $\W : \Mnd_{\underJ}(\C) \to \End_{\underJ}(\C)$ from $\J$-ary $\V$-monads on $\C$ to $\J$-ary $\V$-endofunctors on $\C$ is monadic, and the free $\J$-ary $\V$-monad on a $\J$-ary $\V$-endofunctor is \emph{algebraically} free.  

We commence \S \ref{freemonadsonsignatures} by defining the notion of a \emph{$\Sigma$-algebra} for a $\J$-signature $\Sigma$ in $\C$, relative to a subcategory of arities $j : \J \hookrightarrow \C$, and we then show in \ref{Vismonadic} under certain hypotheses that the forgetful functor $\End_{\underJ}(\C) \to \Sig_{\underJ}(\C)$ from $\J$-ary $\V$-endofunctors on $\C$ to $\J$-signatures in $\C$ is monadic. We then deduce in \ref{freemonadalgebrasaresignaturealgebras} that $\U : \Mnd_{\underJ}(\C) \to \Sig_{\underJ}(\C)$ has a left adjoint, and that the $\V$-category of algebras for the free $\J$-ary $\V$-monad on a $\J$-signature $\Sigma$ is isomorphic to the  $\V$-category $\Sigma\Alg$ of $\Sigma$-algebras, which is defined more directly in terms of the signature $\Sigma$. In \S \ref{monadicity} we use a theorem of Lack \cite{Lackmonadicity} to prove in \ref{Uismonadic} that the forgetful functor $\U : \Mnd_{\underJ}(\C) \to \Sig_{\underJ}(\C)$ is actually \emph{monadic}. 

\S \ref{algebraiccolimits} is concerned with algebraic colimits of $\J$-ary $\V$-monads, and is divided into three subsections. In \S \ref{limitVcategories} we first review some necessary background material on limits in $\V\CAT$ and $\V\CAT/\C$, and we prove some results about limits and colimits in limit $\V$-categories. We then use this material in \S \ref{algebraiccolimitssubsection} to define and study the notion of an algebraic colimit of $\V$-monads, thereby enriching the corresponding notion studied by Kelly in \cite{Kellytrans}. The final subsection \ref{algebraiccolimitsJary} defines the notion of an algebraic colimit of $\J$-ary $\V$-monads, and we then prove in \ref{colimitcategoryofalgebras} that if $j : \J \hookrightarrow \C$ is bounded, then the category $\Mnd_{\underJ}(\C)$ of $\J$-ary $\V$-monads on $\C$ has small algebraic colimits. 

We begin \S \ref{presentationssection} by defining the notion of a $\J$-\emph{presentation} $P = (\Sigma,E)$ for a subcategory of arities $j : \J \hookrightarrow \C$, consisting of a $\J$-signature $\Sigma$ and a \textit{system of $\J$-ary equations} $E = (\Gamma \rightrightarrows \U\left(\T_\Sigma\right))$, i.e. a pair of $\J$-signature morphisms from a $\J$-signature $\Gamma$ (the \textit{signature of equations}) to the underlying $\J$-signature of the free $\J$-ary $\V$-monad $\T_\Sigma$ on $\Sigma$. From results in \S \ref{algebraiccolimits} we then deduce in \ref{everypresentationpresentsJarymonad} that every $\J$-presentation $P$ \emph{presents} a $\J$-ary $\V$-monad $\T_P$, whose $\V$-category of algebras we show in \ref{Palgebrasisomorphismcor} is isomorphic to the $\V$-category $P\Alg$ of $P$\emph{-algebras} for the $\J$-presentation $P = (\Sigma,E)$, i.e. the full sub-$\V$-category of $\Sigma\Alg$ consisting of those $\Sigma$-algebras that \textit{satisfy the equations in $E$}, in a suitable sense. In \ref{presentationcor}, we also deduce that every $\J$-ary $\V$-monad has a $\J$-presentation, using our results on algebraic colimits and monadicity of $\J$-ary monads. In \S \ref{firstexamples} we discuss some examples of $\J$-presentations; firstly, we show that presentations of $\V$-categories by generators and relations are recovered as examples when $\J$ consists of the representables in a power of $\V$, and secondly we discuss presentations of strongly finitary $\V$-monads in cartesian closed topological categories over $\Set$, treating examples including internal modules and affine spaces over internal rigs (i.e. semirings).

In \S \ref{summary} we summarize the main results of the paper, noting that our running examples satisfy the hypotheses of these results, as does any small and eleutheric subcategory of arities in a locally bounded $\V$-category over a locally bounded closed category $\V$.

\section{Notation and background}  
\label{background}

We make substantial use of the methods of enriched category theory throughout this paper; for more details, one can consult (e.g.) the classic texts \cite{Dubucbook, Kelly}.  For the most part, we use the notation of Kelly's text \cite{Kelly}.  Throughout, we let $\V = (\V_0, \tensor, I)$ be a symmetric monoidal closed category with $\V_0$ locally small, complete, and cocomplete.

A \emph{weight} is a $\V$-functor $W : \B \to \V$ with $\B$ a (not necessarily small) $\V$-category, while we say that the weight $W$ is \emph{small} if $\B$ is small.  A \emph{weighted diagram} in a $\V$-category $\C$ is a pair $(W, D)$ consisting of a weight $W : \B \to \V$ and a $\V$-functor $D : \B \to \C$. A \emph{cylinder} for the weighted diagram $(W, D)$ is a pair $(C, \lambda)$ consisting of an object $C \in \ob\C$ and a $\V$-natural transformation $\lambda:W \to \C(C, D-)$.  A \textit{(weighted) limit} of $(W,D)$ is an object $\{W,D\}$ of $\C$ equipped with the structure of a \textit{limit cylinder} $(\{W,D\},\lambda)$ for $(W,D)$, i.e. a limit of $D$ indexed by $W$ in the sense of \cite[\S 3.1]{Kelly}.

It is convenient to define a \emph{dually weighted diagram} in $\C$ to be a pair $(W, D)$ consisting of a weight $W : \B^\op \to \V$ and a $\V$-functor $D : \B \to \C$, noting that $(W,D^\op)$ is then a weighted diagram in $\C^\op$.  A \emph{cylinder} for the dually weighted diagram $(W, D)$ is, by definition, a cylinder for the weighted diagram $(W,D^\op)$, and a \emph{(weighted) colimit} $W * D$ of $(W, D)$ is a limit $\{W,D^\op\}$ of $(W,D^\op)$.

Given a class of weighted diagrams $\Lambda$, a $\V$-functor $F : \C \to \D$ \textit{creates $\Lambda$-limits} provided that for every weighted diagram $(W,D) \in \Lambda$ in $\C$ and every limit cylinder $(C, \lambda)$ for $(W, FD)$ that exists in $\D$, there is a unique cylinder $\left(\bar{C}, \bar{\lambda}\right)$ for $(W, D)$ with $\left(F\bar{C}, F\bar{\lambda}\right) = (C, \lambda)$, and moreover $\left(\bar{C}, \bar{\lambda}\right)$ is a limit cylinder for $(W, D)$. Dually, we have the notion of \textit{creation of $\Lambda$-colimits} for a class of dually weighted diagrams $\Lambda$.  Given instead a class of weights $\Phi$, $F$ \textit{creates $\Phi$-limits} if $F$ creates $\Lambda$-limits for the class $\Lambda$ of all weighted diagrams with weights in $\Phi$; dually, we have the notion of \textit{creation of $\Phi$-colimits}.

Given a $\V$-functor $F : \C \to \D$ and a class of (possibly large) weights $\Phi$, we say that $F$ \emph{conditionally preserves} $\Phi$-limits \cite[2.3]{EAT} provided that for every limit $\{W, D\}$ that exists in $\C$ with $W \in \Phi$, if $\{W, FD\}$ exists in $\D$ then $F$ preserves the limit $\{W, D\}$. It is then easy to see that $F$ conditionally preserves $\Phi$-limits if $F$ creates $\Phi$-limits, and that $F$ preserves $\Phi$-limits if $\D$ has $\Phi$-limits and $F$ conditionally preserves $\Phi$-limits. We also have the dual notion of \emph{conditional preservation of $\Phi$-colimits}.

Given a class of weights $\Phi$, a $\V$-category $\C$ is $\Phi$\emph{-(co)complete} if $\C$ admits all $\Phi$-(co)limits, and a $\V$-functor is $\Phi$\emph{-(co)continuous} if it preserves all $\Phi$-(co)limits. In particular, a $\V$-category $\C$ is\emph{(co)complete} if it is $\Phi$-(co)complete for the class $\Phi$ of all small weights, and a $\V$-functor is \emph{(co)continuous} if it is $\Phi$-(co)continuous for this same class $\Phi$. Given objects $V \in \ob\V$ and $C \in \ob\C$, we denote the \emph{cotensor} of $C$ by $V$ in $\C$ (if it exists) by $[V, C]$, and the \emph{tensor} of $C$ by $V$ in $\C$ (if it exists) by $V \tensor C$.

We let $\V\CAT$ be the category of (possibly large) $\V$-categories. Given a $\V$-category $\C$, a \emph{$\V$-category over $\C$} is an object $(\A,U)$ of the slice category $\V\CAT \slash \C$, i.e., a $\V$-category $\A$ equipped with a specified $\V$-functor $U:\A \rightarrow \C$.  We often write $\A$ to denote $(\A,U)$.   We say that $\A$ is a \textbf{strictly monadic $\V$-category over $\C$} if $U$ is strictly monadic, i.e., if $U$ has a left adjoint and the comparison $\V$-functor of \cite[II]{Dubucbook} is an isomorphism, equivalently, if $\A \cong \T\Alg$ in $\V\CAT/\C$ for some $\V$-monad $\T$ on $\C$, where we equip the $\V$-category $\T\Alg$ of $\T$-algebras with its forgetful $\V$-functor.

We write $\VProf$ to denote the bicategory in which an object is a small $\V$-category, a 1-cell $F:\A \profto \B$ is a $\V$-profunctor, i.e. a $\V$-functor $F:\B^\op \otimes \A \rightarrow \V$, and a 2-cell is a $\V$-natural transformation.  Writing $\VProf^\op$ for the bicategory obtained by reversing only the 1-cells in $\VProf$, there is an isomorphism $(-)^\circ:\VProf^\op \rightarrow \VProf$ that is given on objects by $\A \mapsto \A^\op$ and sends each 1-cell $F:\A \profto \B$ to the 1-cell $F^\circ:\B^\op \profto \A^\op$ obtained as the composite $\A \otimes \B^\op \xrightarrow{\sim} \B^\op \otimes \A \xrightarrow{F} \V$.  Given an object $\A$ of a bicategory $\K$, we write $\Mnd_\K(\A)$ to denote the category of monads on $\A$ in $\K$, i.e. monoids in $\K(\A,\A)$, noting that $\Mnd_\K(\A) = \Mnd_{\K^\op}(\A)$.

\section{Eleutheric subcategories of arities}
\label{firstsection}

We begin by defining the fundamental notion of a subcategory of arities in an enriched category:

\begin{defn}
\label{subcategoryofarities}
A \textbf{subcategory of arities} in a $\V$-category $\C$ is a (not necessarily small) $\V$-category $\J$ equipped with a dense, fully faithful $\V$-functor $j:\J \to \C$.  For most purposes we may assume that $j:\J \hookrightarrow \C$ is a full and dense sub-$\V$-category $j : \J \hookrightarrow \C$. \qed
\end{defn}

\begin{rmk}\label{arities_remark}
In \cite{EAT} the first author defined the notion of a \emph{system of arities} $j : \J \hookrightarrow \V$ in the symmetric monoidal closed category $\V$, which is (equivalently, see \cite[3.8]{EAT}) a full sub-$\V$-category that is closed under $\tensor$ and contains the unit object $I$ (and hence is automatically dense by \cite[5.17]{Kelly}). In this paper, we are generalizing from systems of arities in $\V$ to subcategories of arities in arbitrary $\V$-categories. Our notion of subcategory of arities essentially agrees with that of \cite{BourkeGarner}, except that their subcategories of arities are always \emph{small}, and are only defined relative to locally presentable $\V$-categories over locally presentable closed categories $\V$. Nevertheless, most of the subcategories of arities that we consider in this paper will indeed be small.  \qed
\end{rmk}

\noindent If $j : \J \hookrightarrow \C$ is a subcategory of arities in a $\V$-category $\C$, then we let $\Phi_{\underJ}$ be the class of (not necessarily small) weights $\C(j-, C) : \J^\op \to \V$ with $C \in \ob\C$. We now generalize \cite[7.1]{EAT} from systems of arities in $\V$ to subcategories of arities in arbitrary $\V$-categories: 

\begin{defn}
\label{eleutheric}
A subcategory of arities $j : \J \hookrightarrow \C$ is \textbf{eleutheric} if $\C$ is $\Phi_{\underJ}$-cocomplete and $\C(J, -) : \C \to \V$ preserves $\Phi_{\underJ}$-colimits for each $J \in \ob\J$. \qed
\end{defn}
 
\noindent Equivalently (see \cite[7.3]{EAT}), the subcategory of arities $j : \J \hookrightarrow \C$ is eleutheric if every $\V$-functor $H : \J \to \C$ has a left Kan extension along $j$ that is preserved by each $\C(J, -) : \C \to \V$ ($J \in \ob\J$).

The notion of eleutheric subcategory of arities is related to the notion of \emph{saturated} subcategory of arities defined in \cite[Definition 39]{BourkeGarner}, where (in our notation) a subcategory of arities $j : \J \hookrightarrow \C$ is \emph{saturated} if the composition of any two $\V$-endofunctors on $\C$ that are left Kan extensions along $j$ is itself a left Kan extension along $j$. Using \cite[7.9]{EAT}, it is easy to see that any eleutheric subcategory of arities is saturated, but \emph{a priori} the notion of saturatedness is (slightly) weaker; however, all of the examples of saturated subcategories of arities provided in \cite[Examples 41 and 44]{BourkeGarner} are actually eleutheric by \ref{cocompletioneleutheric} below.

Before providing examples in \ref{eleuthericexamples}, we first prove some useful properties and characterizations of eleutheric subcategories of arities. Recall from \cite[Page 402]{KS} that the \emph{saturation} $\Phi^*$ of a class of small weights $\Phi$ is defined as follows: a small weight $W$ belongs to $\Phi^*$ iff every $\Phi$-cocomplete $\V$-category is $\{W\}$-cocomplete and every $\Phi$-cocontinuous $\V$-functor between $\Phi$-cocomplete $\V$-categories is $\{W\}$-cocontinuous. Recall also (see \cite[3.7]{KS} and \cite[5.35]{Kelly}) that if $F : \mathscr{C} \to \mathscr{D}$ is a $\V$-functor and $\Phi$ is a class of small weights, then $F$ \emph{presents} $\mathscr{D}$ \emph{as a free} $\Phi$\emph{-cocompletion of} $\mathscr{C}$ if $\mathscr{D}$ is $\Phi$-cocomplete and for every $\Phi$-cocomplete $\V$-category $\mathscr{E}$, precomposition with $F$ induces an equivalence of categories $\Phi\Cocts(\mathscr{D}, \mathscr{E}) \xrightarrow{\sim} \V\CAT(\mathscr{C}, \mathscr{E})$, where $\Phi\Cocts(\mathscr{D}, \mathscr{E})$ is the full subcategory of $\V\CAT(\mathscr{D}, \mathscr{E})$ consisting of the $\Phi$-cocontinuous $\V$-functors $\mathscr{D} \to \mathscr{E}$. 

For a small subcategory of arities $j : \J \hookrightarrow \C$, we let $\Psi_{\underJ}$ be the class of all small $\J$-flat weights, where a weight $W$ is $\J$\textbf{-flat} if each $\C(J, -) : \C \to \V$ ($J \in \ob\J$) preserves $W$-colimits. The following lemma is now immediate from the definitions:

\begin{lem}
\label{firsteleuthericlem}
Let $j : \J \hookrightarrow \C$ be a small subcategory of arities in a $\Phi_{\underJ}$-cocomplete $\V$-category $\C$. Then $\J$ is eleutheric iff $\Phi_{\underJ} \subseteq \Psi_{\underJ}$. \qed
\end{lem}

\begin{lem}
\label{secondeleuthericlem}
Let $j : \J \hookrightarrow \C$ be a small subcategory of arities, and let $\Psi$ be a class of small weights such that $\Phi_{\underJ} \subseteq \Psi$ and $\C$ is $\Psi$-cocomplete. Then $j$ presents $\C$ as a free $\Psi$-cocompletion of $\J$ iff $\Psi \subseteq \Psi_{\underJ}$.
\end{lem}

\begin{proof}
The forward implication follows immediately from \cite[4.2]{KS}. Assuming that $\Psi \subseteq \Psi_{\underJ}$, i.e. that each $\C(J, -) : \C \to \V$ ($J \in \ob\J$) preserves $\Psi$-colimits, it then remains by \cite[4.3]{KS} to show that each $C \in \ob\C$ is a $\Psi^*$-colimit of a diagram in $\J$. But the density of $j : \J \hookrightarrow \C$ entails that each $C \in \ob\C$ is a $\Phi_{\underJ}$-colimit of a diagram in $\J$ (see \cite[5.1]{Kelly}), and so the desired claim follows from the inclusions $\Phi_{\underJ} \subseteq \Psi \subseteq \Psi^*$.   
\end{proof}

\noindent The following result now generalizes \cite[7.8]{EAT} from systems of arities $\J$ in $\V$ to subcategories of arities $\J$ in arbitrary $\V$-categories (though here we assume $\J$ is small):

\begin{prop}
\label{eleuthericgeneralization}
Let $j : \J \hookrightarrow \C$ be a small subcategory of arities. Then $\J$ is eleutheric iff $j$ presents $\C$ as a free $\Phi_{\underJ}$-cocompletion of $\J$. 
\end{prop}

\begin{proof}
This follows immediately from \ref{firsteleuthericlem} and \ref{secondeleuthericlem}.  
\end{proof}

\begin{prop}
\label{hypereleutheric}
If $j : \J \hookrightarrow \C$ is a small and eleutheric subcategory of arities such that $\C$ is $\Psi_{\underJ}$-cocomplete, then $j$ presents $\C$ as a free $\Psi_{\underJ}$-cocompletion of $\J$.
\end{prop}

\begin{proof}
This follows immediately from \ref{firsteleuthericlem} and \ref{secondeleuthericlem}.  
\end{proof}

\begin{prop}
\label{cocompletioneleutheric}
Let $j : \J \hookrightarrow \C$ be a small full sub-$\V$-category. Then $j$ is an eleutheric subcategory of arities iff there is a class of small weights $\Psi$ such that $j$ presents $\C$ as a free $\Psi$-cocompletion of $\J$.
\end{prop}

\begin{proof}
The forward implication follows immediately from \ref{eleuthericgeneralization}. For the converse, suppose $j$ is a free $\Psi$-cocompletion. By \cite[3.11, 4.2]{KS}, $j$ is dense, $\C$ is $\Psi$-cocomplete, $\Phi_{\underJ} \subseteq \Psi^*$, and each $\C(J, -) : \C \to \V$ ($J \in \ob\J$) preserves $\Psi$-colimits, so that $\Psi \subseteq \Psi_{\underJ}$.  But $\Psi_{\underJ}$ is clearly saturated, so $\Phi_{\underJ} \subseteq \Psi^* \subseteq \Psi_{\underJ}$, while $\C$ is $\Psi^*$-cocomplete and hence $\Phi_{\underJ}$-cocomplete, so $j$ is eleutheric by \ref{firsteleuthericlem}.    
\end{proof}

\begin{egg}
\label{eleuthericexamples}
We now provide the following examples of eleutheric subcategories of arities:
\begin{enumerate}[leftmargin=0pt,labelindent=0pt,itemindent=*,label=(\arabic*)]
\item Let $\V$ be locally $\alpha$-presentable as a (symmetric monoidal) closed category in the sense of \cite[7.4]{Kellystr}. If $\C$ is a locally $\alpha$-presentable $\V$-category and $\C_\alpha$ is a skeleton of the full sub-$\V$-category consisting of the (enriched) $\alpha$-presentable objects, then $j : \C_\alpha \hookrightarrow \C$ is a small and eleutheric subcategory of arities.  Indeed, by \cite[7.2, 7.4]{Kellystr} we deduce firstly that $\C_\alpha$ is small and secondly that $j : \C_\alpha \hookrightarrow \C$ presents $\C$ as a free cocompletion of $\C_\alpha$ under small conical $\alpha$-filtered colimits, so this follows from \ref{cocompletioneleutheric}.

In fact, \emph{any} small subcategory of arities $j : \J \hookrightarrow \C$ in a locally $\alpha$-presentable $\V$-category $\C$ (over a locally $\alpha$-presentable closed category $\V$) is contained in a small and eleutheric subcategory of arities. Because if $\J$ is small, then there is a regular cardinal $\beta \geq \alpha$ such that every $J \in \ob\J$ is $\beta$-presentable in the enriched sense by \cite[7.4]{Kellystr}, so that $\J$ is contained in the small subcategory of arities $\C_\beta \hookrightarrow \C$, which is eleutheric by the above because $\C$ is locally $\beta$-presentable.    

\item In particular, if $\V = \Set$, then the subcategory of arities $j : \FinCard \hookrightarrow \Set$ consisting of the finite cardinals, which may be regarded as the classical subcategory of arities from universal algebra, is eleutheric, by \cite[7.5.2]{EAT}.

\item If $\V$ is cartesian closed, or more generally if $\V$ is a $\pi$\emph{-category} in the sense of \cite{BorceuxDay}, then the subcategory of arities $j : \SF(\V) = \left\{ n \cdot I \mid n \in \N\right\} \hookrightarrow \V$ on the finite copowers of the unit object of $\V$ is eleutheric, by \cite[7.5.5]{EAT}.  If $\V$ is cartesian closed, then $\SF(\V)$ is isomorphic to the free $\V$-category on $\FinCard$, by \cite[\S 3]{KeLa}.

\item The inclusion $\{I\} \hookrightarrow \V$ of the unit object is an (obviously small) eleutheric subcategory of arities \cite[7.5.4]{EAT}.

\item Given an arbitrary $\V$-category $\C$, it is readily verified that the identity $\V$-functor $1_\C : \C \to \C$ is a (not generally small) eleutheric subcategory of arities (generalizing the case where $\C = \V$ in \cite[7.5.3]{EAT}).

\item Let $\A$ be a small $\V$-category. By \ref{cocompletioneleutheric}, the Yoneda embedding $\y_\A : \A^\op \hookrightarrow [\A, \V]$ is an eleutheric subcategory of arities, since $\y_\A$ presents $[\A,\V]$ as a free cocompletion of $\A^\op$ under all small colimits by \cite[4.51]{Kelly}.

\item Let $\Phi$ be a class of small weights that satisfies Axiom A of Lack-Rosick\'y \cite{LR} and is \textit{locally small} in the sense of \cite[8.10]{KS} (as in \cite[p. 370]{LR}).  Let $\scrT$ be a $\Phi$\emph{-theory}, i.e. a small $\V$-category $\scrT$ with $\Phi$-limits, and let $\C = \Phi\Mod(\scrT)$ be the $\V$-category of \emph{models} of $\scrT$ in $\V$, i.e. the full sub-$\V$-category of $[\scrT, \V]$ consisting of the $\Phi$-continuous $\V$-functors.  Following \cite{LR}, we call such $\V$-categories $\C$ locally $\Phi$-presentable---noting, however, that $\C_0$ and $\V_0$ need not be locally presentable (see \ref{mainresultsexamples}(5) below).  Recall that a small weight $W : \B^\op \to \V$ is $\Phi$\emph{-flat} \cite{KS} if $W$-colimits commute in $\V$ with $\Phi$-limits, equivalently, if $W \ast (-) : [\B, \V] \to \V$ is $\Phi$-continuous.  By \cite[\S 6.4]{LR}, $\C$ is cocomplete and the (corestricted) Yoneda embedding $\y_\Phi:\scrT^\op \rightarrow \C$ presents $\C$ as a free cocompletion of $\scrT^\op$ under small $\Phi$-flat colimits.  Hence, by \ref{cocompletioneleutheric} we find that $\y_\Phi : \scrT^\op \hookrightarrow \C$ is an eleutheric subcategory of arities. 

\item As a special case of (7), if $\mathbb{D}$ is any small class of small categories that is a \emph{sound doctrine} in the sense of \cite{ABLR}, and $\V$ is locally $\bbD$-presentable as a $\tensor$-category in the sense of \cite[5.4]{LR}, then we can take $\Phi := \Phi_{\mathbb{D}}$ to be the saturation of the class of (small) weights for conical $\mathbb{D}$-limits and cotensors by $\bbD$-presentable objects of $\V$, which satisfies Axiom A by \cite[5.22]{LR} and is locally small by the remarks in \cite[p. 421]{KS}, as it is the saturation of a small class of weights (in view of \cite[5.20]{LR}).   \qed  
\end{enumerate}
\end{egg}

\noindent We emphasize that in examples (3)-(7), $\V_0$ and $\C_0$ need not be locally presentable.

\section{\texorpdfstring{$\J$}{J}-ary \texorpdfstring{$\V$}{V}-endofunctors and \texorpdfstring{$\V$}{V}-monads}
\label{Jarysection}

We now define the notion of a $\J$-ary $\V$-endofunctor or $\V$-monad (cf. \cite[11.1]{EAT} for the original definition in the context of systems of arities in $\V$):

\begin{defn}
\label{Jary}
Let $j : \J \hookrightarrow \C$ be a subcategory of arities in a $\V$-category $\C$. A \mbox{$\V$-functor} $H : \C \to \C$ is $\J$\textbf{-ary} (or \textbf{$j$-ary}) if it preserves $\Phi_{\underJ}$-colimits, or equivalently if it preserves left Kan extensions along $j$. A $\V$-monad $\T$ on $\C$ is $\J$\textbf{-ary} if its underlying $\V$-endofunctor is $\J$-ary.\qed
\end{defn}

\noindent When $\J$ is eleutheric, $\J$-ary $\V$-endofunctors can also be characterized as follows:

\begin{prop}\label{charns_jary}
Let $j:\J \hookrightarrow \C$ be a small and eleutheric subcategory of arities, and let $H:\C \rightarrow \C$ be a $\V$-functor.  The following are equivalent:
\begin{enumerate}
\item[1.] $H$ is $\J$-ary (i.e. $H$ is $\Phi_{\underJ}$-cocontinuous);
\item[2.] $H$ is a left Kan extension along $j$ (equivalently, $H \cong \Lan_j(Hj)$).
\end{enumerate}
If $\C$ is $\Psi_{\underJ}$-cocomplete, then (1) and (2) are also equivalent to the following:
\begin{enumerate}
\item[3.] $H$ preserves small $\J$-flat colimits (i.e. $H$ is $\Psi_{\underJ}$-cocontinuous).
\end{enumerate}
Moreover, if $\Psi$ is any class of small weights such that $j$ is a free $\Psi$-cocompletion, then (1) and (2) are equivalent to the following:
\begin{enumerate}
\item[4.] $H$ is $\Psi$-cocontinuous.
\end{enumerate}
\end{prop}
\begin{proof}
Firstly, if $\Psi$ is any class of small weights such that $j$ is a free $\Psi$-cocompletion, then (2) is equivalent to (4) by \cite[3.6]{KS}.  In particular, since $j$ is a free $\Phi_{\underJ}$-cocompletion by \ref{eleuthericgeneralization}, this entails that (2) is equivalent to (1).  If $\C$ is $\Psi_{\underJ}$-cocomplete, then $j$ is also a free $\Psi_{\underJ}$-cocompletion by \ref{hypereleutheric}, so (2) is equivalent to (3).
\end{proof}

\begin{rmk}\label{non_eleu_jary_is_lan}
When $j:\J \hookrightarrow \C$ is not assumed eleutheric, the implication (1)$\Rightarrow$(2) in \ref{charns_jary} still holds, because if $H$ is $\Phi_{\underJ}$-cocontinuous then $HC \cong H(\C(j-,C) * j) \cong \C(j-,C) * Hj$ $\V$-naturally in $C \in \C$, by the density of $j$, so $H \cong \Lan_j(Hj)$. \qed
\end{rmk}

\begin{para}
Let $j : \J \hookrightarrow \C$ be a subcategory of arities in a $\V$-category $\C$. We let $\End(\C)$ be the ordinary category of $\V$-endofunctors on $\C$ and $\V$-natural transformations, and we let $\End_{\underJ}(\C)$ be its full subcategory consisting of the $\J$-ary $\V$-endofunctors on $\C$. We let $\Mnd(\C)$ be the ordinary category of $\V$-monads on $\C$ and $\V$-monad morphisms, and we let $\Mnd_{\underJ}(\C)$ be its full subcategory consisting of the $\J$-ary $\V$-monads on $\C$. Regarding $\End_{\underJ}(\C)$ as a strict monoidal category with composition as monoidal product and the identity $\V$-functor as unit, we have that $\Mnd_{\underJ}(\C) = \Mon\left(\End_{\underJ}(\C)\right)$, the category of monoids in $\End_{\underJ}(\C)$. \qed 
\end{para}

\begin{para}
\label{eleuthericequivalence}
Let $j : \J \hookrightarrow \C$ be a small eleutheric subcategory of arities.  Then, by \ref{eleuthericgeneralization}, $j$ presents $\C$ as a free $\Phi_\J$-cocompletion of $\J$, so we have an equivalence $\V\CAT\left(\J, \C\right) \simeq \End_{\underJ}(\C)$, given by restriction and left Kan extension along $j$ (in view of \cite[3.6]{KS}). \qed
\end{para} 

\begin{para}
\label{Jtheory}
It is shown in \cite[11.8]{EAT} that if $j : \J \hookrightarrow \V$ is an eleutheric \emph{system} of arities in $\V$ (\ref{arities_remark}), then the category $\Mnd_{\underJ}(\V)$ of $\J$-ary $\V$-monads on $\V$ is equivalent to the category $\Th_{\underJ}$ of $\J$\emph{-theories}, where a $\J$-theory is a $\V$-category $\scrT$ equipped with an identity-on-objects $\V$-functor $\tau : \J^\op \to \scrT$ that preserves $\J$-cotensors \cite[4.1]{EAT}. \qed
\end{para} 

\begin{egg}
\label{Jaryendofunctors}
We now characterize the $\J$-ary $\V$-endofunctors and $\V$-monads for the following eleutheric subcategories of arities $j : \J \hookrightarrow \C$ (see \ref{eleuthericexamples}):
\begin{enumerate}[leftmargin=0pt,labelindent=0pt,itemindent=*,label=(\arabic*)]

\item If $\V$ is locally $\alpha$-presentable as a closed category and $\C$ is a locally $\alpha$-presentable $\V$-category with the associated small and eleutheric subcategory of arities $j : \C_\alpha \hookrightarrow \C$, then because $j : \C_\alpha \hookrightarrow \C$ presents $\C$ as a free cocompletion of $\C_\alpha$ under small conical $\alpha$-filtered colimits (\ref{eleuthericexamples}), it follows by \ref{charns_jary} that the $\C_\alpha$-ary $\V$-endofunctors on $\C$ are precisely the $\alpha$-ary $\V$-endofunctors, i.e. the $\V$-endofunctors that preserve small conical $\alpha$-filtered colimits. When $\C = \V$, the $\alpha$-ary $\V$-monads on $\V$ correspond to an $\alpha$-ary generalization of the enriched Lawvere theories of \cite{NishizawaPower}. 

\item In particular, if $\V = \Set$ and $j : \FinCard \hookrightarrow \Set$ is the classical system of arities in universal algebra, then the $\FinCard$-ary endofunctors on $\Set$ are precisely the usual finitary endofunctors, i.e. the endofunctors that preserve small filtered colimits, and the finitary monads on $\Set$ correspond to Lawvere theories in the usual sense \cite{Law:PhD}. 

\item If $j : \SF(\V) \hookrightarrow \V$ is the eleutheric subcategory of arities on the finite copowers of the unit object in a symmetric monoidal closed $\pi$-category $\V$, then a $\V$-functor $H : \V \to \V$ is $\SF(\V)$-ary iff $H$ preserves small $\SF(\V)$-flat colimits (i.e. small colimits that commute in $\V$ with finite powers) by \ref{charns_jary}. The $\SF(\V)$-ary $\V$-monads on $\V$ correspond to the enriched algebraic theories of Borceux and Day \cite{BorceuxDay}, by \cite[4.2]{EAT}.  In the case where $\V$ is cartesian closed, $\SF(\V)$ is the free $\V$-category on $\FinCard$ (\ref{eleuthericexamples}), so $\SF(\V)$-ary $\V$-endofunctors are precisely the \textit{strongly finitary} endofunctors of Kelly and Lack \cite[\S 3]{KeLa}.

\item For the eleutheric subcategory of arities $j : \{I\} \hookrightarrow \V$, the $\{I\}$-ary endofunctors on $\V$ are precisely those $\V$-endofunctors that are isomorphic to $X \tensor (-) : \V \to \V$ for some $X \in \ob\V$, in view of \cite[7.5(4)]{EAT} and \ref{charns_jary}.  The $\{I\}$-ary $\V$-monads on $\V$ correspond to monoids in $\V$ by \cite[4.2(5), 11.8]{EAT}.

\item For the eleutheric subcategory of arities $1_\C : \C \to \C$ in a $\V$-category $\C$, the $\C$-ary endofunctors on $\C$ are arbitrary $\V$-endofunctors by \cite[11.3.2]{EAT}. So when $\C = \V$, the $\V$-ary $\V$-monads on $\V$ are arbitrary $\V$-monads on $\V$, which correspond to the $\V$-theories of Dubuc \cite{Dubucsemantics}.

\item Consider the small and eleutheric subcategory of arities $\y_\A : \A^\op \hookrightarrow \C = [\A, \V]$ for a small $\V$-category $\A$. Since $\y_\A$ presents $\C$ as a free cocompletion of $\A^\op$ under all small colimits, we deduce from \ref{charns_jary} that a $\V$-functor $H : \C \to \C$ is $\y_\A$-ary iff $H$ preserves small colimits. The equivalences $\End_{\y_\A}(\C) \simeq \V\CAT(\A^\op,\C) \cong \V\CAT(\A\otimes\A^\op,\V) \cong \V\CAT(\A^\op \otimes \A,\V)$ underlie equivalences of monoidal categories $\End_{\y_\A}(\C)  \simeq \VProf(\A^\op,\A^\op) \cong \VProf^\op(\A,\A)$ by \S \ref{background}.  Hence the category of $\y_\A$-ary $\V$-monads $\Mnd_{\y_\A}(\C)$ is equivalent to the category $\Mnd_{\VProf^\op}(\A) = \Mnd_{\VProf}(\A)$ of \emph{$\V$-profunctor monads} on $\A$ (by \S \ref{background}).  But $\Mnd_{\VProf}(\A)$ is equivalent to the coslice category $\A\slash\V\Cat(\ob\A)$, where $\V\Cat(\ob\A)$ is the category whose objects are $\V$-categories with object set $\ob\A$, and whose morphisms are identity-on-objects $\V$-functors (by, e.g., \cite[10.4]{EAT}).  Hence a $\y_\A$-ary $\V$-monad on $\C = [\A,\V]$ is equivalently given by a $\V$-category $\scrT$ equipped with an identity-on-objects $\V$-functor $\A \rightarrow \scrT$.

\item Let $\Phi$ be a locally small class of small weights satisfying Axiom A of \cite{LR}, and let $\C = \Phi\Mod(\scrT)$ be the $\V$-category of models of a $\Phi$-theory $\scrT$, with the small and eleutheric subcategory of arities $\y_\Phi : \scrT^\op \hookrightarrow \C$ consisting of the representables. By \ref{charns_jary}, a $\V$-functor $H : \C \to \C$ is $\y_\Phi$-ary iff $H$ preserves small $\Phi$-flat colimits (since $\y_\Phi$ presents $\C$ as a free cocompletion of $\scrT^\op$ under small $\Phi$-flat colimits). Hence $\y_\Phi$-ary $\V$-monads on $\C$ are the $\Phi$\emph{-accessible} $\V$-monads of \cite{LR}; \cite[7.7]{LR} provides a correspondence between these and Lawvere $\Phi$-theories in $\C$.  

\item In particular, given a sound doctrine $\bbD$, if $\V$ is locally $\bbD$-presentable as a $\tensor$-category and we let $\C =  \Phi_\bbD\Mod(\scrT)$ for a $\Phi_\bbD$-theory $\scrT$, with subcategory of arities $\y_{\Phi_\bbD}:\scrT^\op \hookrightarrow \C$, then a $\V$-endofunctor $H : \C \to \C$ is $\y_{\Phi_\bbD}$-ary iff $H$ preserves small $\Phi_\bbD$-flat colimits, which is equivalent to $H$ preserving small conical $\bbD$-filtered colimits by the following lemma:
\end{enumerate}

\begin{lem}
\label{accessiblelem}
Let $\bbD$ be a sound doctrine, let $\V$ be locally $\bbD$-presentable as a $\tensor$-category, and let $\C, \C'$ be $\Phi_\bbD$-cocomplete $\V$-categories. Then a $\V$-functor $F : \C \to \C'$ preserves small $\Phi_\bbD$-flat colimits iff $F$ preserves small conical $\bbD$-filtered colimits.   
\end{lem}

\begin{proof}
The forward implication holds because small conical $\bbD$-filtered colimits are $\Phi_\bbD$-flat (as remarked before \cite[5.21]{LR}). Conversely, suppose that $F$ preserves small conical $\bbD$-filtered colimits. If $W : \B^\op \to \V$ is a small \textit{$\Phi_\bbD$-continuous weight}, meaning that $\B$ is $\Phi_\bbD$-cocomplete and $W$ preserves $\Phi_\bbD$-limits, then $W$ is a conical $\bbD$-filtered colimit of representables by \cite[5.22]{LR}, and thus $W$ belongs to the saturation of the class of weights for small conical $\bbD$-filtered colimits by \cite[3.12]{KS}. Our hypothesis then entails that $F$ also preserves small $\Phi_\bbD$-continuous weighted colimits. Now let $W : \B^\op \to \V$ be an arbitrary small $\Phi_\bbD$-flat weight and $D : \B \to \C$ a $\V$-functor, and let us show that $F$ preserves the colimit $W \ast D$. Consider the free $\Phi_\bbD$-cocompletion $i : \B \to \Phi_\bbD(\B)$ of $\B$, which is small because $\B$ is small and $\Phi_\bbD$ is locally small (\ref{eleuthericexamples}). Then $\Phi_\bbD(\B)^\op$ is a small and $\Phi_\bbD$-complete $\V$-category equipped with a fully faithful $\V$-functor $i^\op : \B^\op \to \Phi_\bbD(\B)^\op$. Since $W$ is $\Phi_\bbD$-flat, it follows by \cite[2.4]{LR} that $\Lan_{i^\op}W : \Phi_\bbD(\B)^\op \to \V$ is $\Phi_\bbD$-flat and hence $\Phi_\bbD$-continuous by \cite[5.22]{LR}. Since $\C$ is $\Phi_\bbD$-cocomplete and $\Phi_\bbD(\B)$ is the free $\Phi_\bbD$-cocompletion of $\B$, there is a ($\Phi_\bbD$-cocontinuous) $\V$-functor $D' : \Phi_\bbD(\B) \to \C$ with $D' \circ i \cong D$. By \cite[2.1]{LR} we now have canonical isomorphisms 
\[ F(W \ast D) \cong F\left(W \ast D'i\right) \cong F\left(\Lan_{i^\op}W \ast D'\right) \cong \Lan_{i^\op}W \ast FD' \cong W \ast FD'i \cong W \ast FD, \]
the third isomorphism existing because the small weight $\Lan_{i^\op}W$ is $\Phi_\bbD$-continuous and $F$ preserves small $\Phi_\bbD$-continuous weighted colimits.  
\end{proof}
\end{egg}

\begin{egg}[\textbf{$\V$-categories as $\V$-matrix monads}]\label{presheaf_as_em}
Let $X$ be a set, and write also $X$ to denote the discrete $\V$-category on $X$.  Specializing Example \ref{Jaryendofunctors}(6) to the case where $\A = X$, consider the eleutheric subcategory of arities $\y_X:X \rightarrow [X,\V] = \V^X$ given by $\y_X(x) = X(x,-)$.  Let us write $\VMat$ to denote the bicategory of \textit{$\V$-matrices}, i.e. the full sub-bicategory of $\VProf$ consisting of the small, discrete $\V$-categories (which we regard also as sets).  As a special case of \ref{Jaryendofunctors}(6), the category $\Mnd_{\y_X}(\V^X)$ of $\y_X$-ary $\V$-monads on $\V^X$ is equivalent to the category $\Mnd_{\VMat}(X) = \Mnd_{\VMat^\op}(X)$ of \textit{$\V$-matrix monads} on $X$, i.e. monads on $X$ in $\VMat$, which in turn is equivalent to the category $\V\Cat(X)$ of $\V$-categories with object set $X$ (with identity-on-objects $\V$-functors).  Given a $\V$-category $\scrT$ with $\ob\scrT = X$, let us write $\Hom_{\scrT}$ to denote the $\V$-matrix monad on $X$ determined by $\scrT$.  Writing composition of $1$-cells in $\VMat$ as $\otimes$, the unit $\V$-category $\II$ determines a homomorphism of bicategories $\VMat(-,\II):\VMat^\op \rightarrow \mathsf{Cat}$ that sends $\Hom_{\scrT}$ to a monad $(-)\otimes \Hom_{\scrT}$ on $\VMat(X,\II) = \V\CAT(X,\V)$, which underlies a $\V$-monad $(-) \otimes \Hom_{\scrT}$ on $[X,\V] = \V^X$.  Under the equivalence $\Mnd_{\VMat}(X) \simeq \Mnd_{\y_X}(\V^X)$ of \ref{Jaryendofunctors}(6), the $\y_X$-ary $\V$-monad $\T$ corresponding to $\scrT$ is precisely $(-) \otimes \Hom_{\scrT}:\V^X \rightarrow \V^X$.  Hence $\T\Alg = ((-) \otimes \Hom_{\scrT})\Alg$ is the $\V$-category of right $\Hom_{\scrT}$-modules in $\V^X$ for the monoid $\Hom_{\scrT}$ in $\VMat(X,X)$ (relative to the right action of the latter monoidal category on $\V^X$ by $\V$-matrix composition).  But right $\Hom_\scrT$-modules in $\V^X$ are equivalently described as (covariant!) $\V$-functors from $\scrT$ to $\V$, and moreover
\[\T\Alg \cong [\scrT,\V]\;.\tag*{\qed}\]
 \end{egg}  

\begin{defn}\label{str_jmonadic}
Let $j:\J \hookrightarrow \C$ be a subcategory of arities, and let $\A$ be a $\V$-category equipped with a $\V$-functor $U : \A \to \C$.  We say that is $U$ is \textbf{strictly $\J$-monadic} if there is a $\J$-ary $\V$-monad $\T$ on $\C$ with $\A \cong \T\Alg$ in $\V\CAT/\C$, in which case we say that $\A$ is a \textbf{strictly $\J$-monadic $\V$-category over $\C$}. \qed
\end{defn}

\noindent In general, if $\Psi$ is a class of weights and $\T$ is a $\V$-monad on a $\Psi$-cocomplete $\V$-category $\C$, then $T:\C \rightarrow \C$ preserves $\Psi$-colimits if and only if $U^\T:\T\Alg \rightarrow \C$ creates $\Psi$-colimits, from which we obtain the following:

\begin{prop}\label{str_jmon_prop}
In the situation of \ref{str_jmonadic}, suppose that $\C$ is $\Phi_{\underJ}$-cocomplete.  Then $U$ is strictly $\J$-monadic if and only if $U$ is strictly monadic and $U$ creates $\Phi_{\underJ}$-colimits. \qed
\end{prop}

\section{Algebraically free \texorpdfstring{$\V$}{V}-monads in general}
\label{algfreemonads}

To achieve our first main objectives in \S \ref{freemonadsection}, we first need to extend certain techniques and results on \emph{algebraically free monads} \cite{Kellytrans} from the ordinary to the enriched context.  

\begin{para}\label{h_algebras}
Given an endofunctor $H$, we write $H\Alg$ for the category of \textit{$H$-algebras} \cite[5.37]{AHS}.  Given instead a $\V$-endofunctor $H : \C \to \C$, we call the objects of $H_0\Alg$ simply $H$-algebras, and $H_0\Alg$ underlies a $\V$-category $H\Alg$ that is defined as in \cite[II.1]{Dubucbook} and is equipped with a ($\V$-)faithful $\V$-functor $U^H : H\Alg \to \C$ given on objects by $(A, a) \mapsto A$.  If $\T = (T,\eta,\mu)$ is a $\V$-monad on $\C$, then Eilenberg-Moore $\T$-algebras constitute a full sub-$\V$-category $\T\Alg \hookrightarrow T\Alg$, and we write $U^\T : \T\Alg \to \C$ to denote the restriction of $U^T$.  \qed
\end{para}

\begin{para}
\label{pointedendofunctor}
A \textbf{pointed} $\V$\textbf{-endofunctor} on a $\V$-category $\C$ is a pair $\mathsf{P} = (P, \pi)$ consisting of a $\V$-endofunctor $P : \C \to \C$ and a $\V$-natural transformation $\pi : 1_\C \to P$.  A \mbox{$\mathsf{P}$\textbf{-algebra}} is then a $P$-algebra $(A, a)$ with $a \circ \pi_A = 1_A$. We write $\mathsf{P}\Alg$ to denote the full sub-$\V$-category of $P\Alg$ consisting of the $\mathsf{P}$-algebras, and we write $U^{\mathsf{P}} : \mathsf{P}\Alg \to \C$ to denote the restriction of $U^P$.

Pointed $\V$-endofunctors on $\C$ are the objects of a category $\End_*(\C)$, in which a morphism $\alpha : \mathsf{P} \to \mathsf{Q}$ is a $\V$-natural transformation $\alpha : P \to Q$ that commutes with the associated transformations $1_\C \to P$ and $1_\C \to Q$. \qed 
\end{para} 

\begin{para}\label{semantics_early}
The \textbf{semantics} functor $\ALG : \Mnd(\C)^\op \to \V\CAT/\C$ \cite[II]{Dubucbook} sends each $\V$-monad $\T$ to $\T\Alg$, equipped with its associated $\V$-functor $U^\T : \T\Alg \to \C$, and sends each morphism of $\V$-monads $\alpha : \T \to \T'$ to the unique $\V$-functor $\alpha^* : \T'\Alg \to \T\Alg$ that is given on objects by $(A,a) \mapsto (A, a \circ \alpha_A)$ and commutes with the faithful $\V$-functors  $U^{\T'}$ and $U^{\T}$.  The semantics functor $\ALG$ is fully faithful by \cite[Pages 74--75]{Dubucbook}.

More basically, one can similarly show that there is a functor $\End_*(\C)^\op \rightarrow \V\CAT \slash \C$ that is given on objects by $\mathsf{P} \mapsto \mathsf{P}\Alg$ and sends each morphism $\alpha : \mathsf{P} \to \mathsf{Q}$ in $\End_*(\C)$ to the unique $\V$-functor $\alpha^* : \mathsf{Q}\Alg \to \mathsf{P}\Alg$ that is given on objects by $(A,a) \mapsto (A,a \circ \alpha_A)$ and commutes with the faithful $\V$-functors $U^\mathsf{Q}$ and $U^\mathsf{P}$.  \qed 
\end{para}

\noindent The following serves as an enrichment of the corresponding definition for ordinary categories in \cite[\S 22]{Kellytrans}, but we provide a different formulation of this definition that we find conceptually clarifying and technically convenient:

\begin{defn}
\label{algebraicallyfreemonad2}
Let $\mathsf{P} = (P, \pi)$ be a pointed $\V$-endofunctor on a $\V$-category $\C$ and $\T$ a $\V$-monad on $\C$. Then $\T$ is an \textbf{algebraically free} $\V$\textbf{-monad on} $\mathsf{P}$ if $\T\Alg \cong \mathsf{P}\Alg$ in $\V\CAT/\C$. Since the semantics functor $\ALG:\Mnd(\C)^\op \rightarrow \V\CAT\slash\C$ is fully faithful (\ref{semantics_early}), it follows that an algebraically free $\V$-monad on $\mathsf{P}$ is unique up to isomorphism if it exists, in which case we denote it by $\T_\mathsf{P}$.  Note that the algebraically free $\V$-monad on $\mathsf{P}$ exists iff $\mathsf{P}\Alg$ is a strictly monadic $\V$-category over $\C$. \qed
\end{defn}

\noindent Writing $\W_*:\Mnd(\C) \rightarrow \End_*(\C)$ for the forgetful functor, our aim is now to enrich \cite[22.2]{Kellytrans} to show that an algebraically free $\V$-monad on a pointed $\V$-endofunctor $\mathsf{P}$ is also a \emph{free} $\V$-monad on $\mathsf{P}$, with respect to $\W_*$. To do this, we first require the following enrichment and reformulation of \cite[22.1]{Kellytrans}:
 
\begin{lem}
\label{enrichedalgfreelemma}
Let $\T = (T, \eta, \mu)$ be a $\V$-monad on $\C$.  Then
$$(\V\CAT\slash \C)(\T\Alg,\mathsf{P}\Alg) \;\cong\; \End_*(\C)(\mathsf{P},\W_*(\T))$$
naturally in $\mathsf{P} \in \End_*(\C)$.     
\end{lem}
\begin{proof}
Through a straightforward enrichment of the proof of \cite[22.1]{Kellytrans}, we can show that the inclusion $\T\Alg \hookrightarrow \W_*(\T)\Alg$ is the counit of a representation of the needed form.
\end{proof}

\noindent We now have the following enrichment of \cite[22.2]{Kellytrans}:  

\begin{prop}
\label{algfreemonadisfreemonad}
Let $\mathsf{P}$ be a pointed $\V$-endofunctor on a $\V$-category $\C$, and suppose that the algebraically free $\V$-monad $\T_\mathsf{P}$ on $\mathsf{P}$ exists.  Then $\T_\mathsf{P}$ is also a free $\V$-monad on $\mathsf{P}$, with respect to the forgetful functor $\W_*:\Mnd(\C) \rightarrow \End_*(\C)$. 
\end{prop}

\begin{proof}
By definition $\T_\mathsf{P}\Alg \cong \mathsf{P}\Alg$ in $\V\CAT \slash \C$, so since the semantics functor is fully faithful (\ref{semantics_early}) we compute that
$$
\begin{array}{rcl}
\Mnd(\C)(\T_\mathsf{P},\mathbb{U}) & \cong & (\V\CAT\slash \C)(\mathbb{U}\Alg,\T_\mathsf{P}\Alg)\\
& \cong & (\V\CAT\slash\C)(\mathbb{U}\Alg,\mathsf{P}\Alg)\\
& \cong & \End_*(\C)(\mathsf{P},\W_*(\mathbb{U}))
\end{array} 
$$
naturally in $\mathbb{U} \in \Mnd(\C)$.
\end{proof}

\noindent We now require the following lemma on creation of limits and colimits, whose proof is a straightforward variation and enrichment of the corresponding well-known result for (ordinary) monads: 

\begin{lem}
\label{endofunctorcreationcolimits}
Let $H$ (resp. $\mathsf{P}$) be a $\V$-endofunctor (resp. a pointed $\V$-endofunctor) on a $\V$-category $\C$, let $\A$ be the $\V$-category of $H$-algebras (resp. $\mathsf{P}$-algebras), and let $U:\A \rightarrow \C$ be the forgetful $\V$-functor.  Then $U$ creates all limits.  Also, if $\Lambda$ is a class of dually weighted diagrams and $H$ (resp. $P$) preserves $(W,UD)$-colimits for every dually weighted diagram $(W,D) \in \Lambda$ in $\A$, then $U$ creates $\Lambda$-colimits. \qed
\end{lem}

\noindent We can now prove the following enrichment of \cite[22.3]{Kellytrans}:

\begin{theo}
\label{algfreeleftadjoint}
Let $\mathsf{P} = (P, \pi)$ be a pointed $\V$-endofunctor on a $\V$-category $\C$. Then the algebraically free $\V$-monad $\T_{\mathsf{P}}$ on $\mathsf{P}$ exists iff the forgetful $\V$-functor $U^{\mathsf{P}} : \mathsf{P}\Alg \to \C$ has a left adjoint $F^{\mathsf{P}}$; and then $\T_{\mathsf{P}}$ is the $\V$-monad arising from the adjunction $F^{\mathsf{P}} \dashv U^{\mathsf{P}}$.   
\end{theo}

\begin{proof}
By \ref{endofunctorcreationcolimits}, $U^{\mathsf{P}}$ creates conical coequalizers of $U^{\mathsf{P}}$-split pairs, so by \cite[II.2.1]{Dubucbook} we deduce that $U^{\mathsf{P}}$ is strictly monadic iff $U^{\mathsf{P}}$ has a left adjoint, and the result follows, in view of \ref{algebraicallyfreemonad2}.        
\end{proof}

\noindent Thus, the algebraically free $\V$-monad on a pointed $\V$-endofunctor $\mathsf{P}$ is the free $\mathsf{P}$-algebra $\V$-monad (if it exists). 

If $H : \C \to \C$ is a $\V$-endofunctor and $\C$ has (conical) binary coproducts, so that $\End(\C)$ does as well, then it is easy to verify that $\mathsf{P}_H := (1_\C + H, \inj_1)$ is the free \emph{pointed} $\V$-endofunctor on $H$, where $\inj_1 : 1_\C \to 1_\C + H$ is the first coproduct insertion. 

\begin{defn}
\label{algfreepointedendofunctor}
Let $H : \C \to \C$ be a $\V$-endofunctor on a $\V$-category $\C$, and let $\mathsf{P} = (P, \pi)$ be a pointed $\V$-endofunctor on $\C$. Then $\mathsf{P}$ is an \textbf{algebraically free pointed} $\V$\textbf{-endofunctor on} $H$ if $H\Alg \cong \mathsf{P}\Alg$ in $\V\CAT/\C$. \qed 
\end{defn}

\noindent We now have the following lemma, whose proof is a straightforward enrichment of the considerations in \cite[\S 18]{Kellytrans}: 

\begin{lem}
\label{algebrasforfreepointedendofunctor}
Let $H : \C \to \C$ be a $\V$-endofunctor on a $\V$-category $\C$ with conical binary coproducts. Then the free pointed $\V$-endofunctor $\mathsf{P}_H$ on $H$ is an algebraically free pointed $\V$-endofunctor on $H$. \qed
\end{lem}

\begin{defn}
\label{algfreemonadonendofunctor}
Let $H : \C \to \C$ be a $\V$-endofunctor on a $\V$-category $\C$, and let $\T$ be a $\V$-monad on $\C$. Then $\T$ is an \textbf{algebraically free} $\V$\textbf{-monad on} $H$ if $H\Alg \cong \T\Alg$ in $\V\CAT/\C$. Since the semantics functor is fully faithful (\ref{semantics_early}), an algebraically free $\V$-monad on $H$ is unique up to isomorphism if it exists, in which case we denote it by $\T_H$.\qed 
\end{defn}

\begin{cor}
\label{algfreemonadonendofunctorisfree}
Let $H$ be a $\V$-endofunctor on a $\V$-category $\C$ with conical binary coproducts, and suppose that the algebraically free $\V$-monad $\T_H$ on $H$ exists.  Then $\T_H$ is a free $\V$-monad on $H$.
\end{cor}

\begin{proof}
Let $\mathsf{P}_H := \left(1_\C + H, \inj_1\right)$ be the free pointed $\V$-endofunctor on $H$.
By hypothesis and \ref{algebrasforfreepointedendofunctor} we have isomorphisms $\T_H\Alg \cong H\Alg \cong \mathsf{P}_H\Alg$ in $\V\CAT/\C$, so $\T_H$ is algebraically free on $\mathsf{P}_H$ and hence free on $\mathsf{P}_H$ by \ref{algfreemonadisfreemonad}, so $\T_H$ is free on $H$.
\end{proof}

\begin{theo}
\label{algfreemonadonendofunctorcor}
Let $H$ be a $\V$-endofunctor on a $\V$-category $\C$ with conical binary coproducts. Then the algebraically free $\V$-monad $\T_{H}$ on $H$ exists iff the forgetful $\V$-functor $U^H : H\Alg \to \C$ has a left adjoint $F^{H} : \C \to H\Alg$; and then $\T_{H}$ is the $\V$-monad arising from the adjunction $F^{H} \dashv U^{H}$. 
\end{theo}

\begin{proof}
Since $H\Alg \cong \mathsf{P}_H\Alg$ in $\V\CAT\slash\C$ by \ref{algebrasforfreepointedendofunctor}, this follows from \ref{algfreeleftadjoint}.    
\end{proof}

\section{The free \texorpdfstring{$\J$}{J}-ary \texorpdfstring{$\V$}{V}-monad on a \texorpdfstring{$\J$}{J}-ary \texorpdfstring{$\V$}{V}-endofunctor}
\label{freemonadsection}

The objective of this section is to show that if $j : \J \hookrightarrow \C$ is a subcategory of arities satisfying certain hypotheses, then every $\J$-ary $\V$-endofunctor on $\C$ has an algebraically free $\J$-ary $\V$-monad on $\C$, so that in particular the forgetful functor $\W : \Mnd_{\underJ}(\C) \to \End_{\underJ}(\C)$ has a left adjoint. In the special case of the canonical subcategory of arities $j : \C_{\aleph_0} \hookrightarrow \C$ in a locally finitely presentable $\V$-category $\C$ over a locally finitely presentable $\V$, a result of the latter form was proved by Kelly and Power in \cite[\S 4]{KellyPower}, and an analogous result for strongly finitary $\V$-monads on a cartesian closed category $\V$ was proved by Kelly and Lack in \cite[\S 3]{KeLa}.

\subsection{Bounded subcategories of arities}
\label{boundedsection}

In this section, we identify hypotheses on $\J \hookrightarrow \C$ that will enable us to prove in \S \ref{existenceJarysection} that every $\J$-ary $\V$-endofunctor $H : \C \to \C$ has an algebraically free $\J$-ary $\V$-monad $\T_H$, by showing that the forgetful $\V$-functor $U^H : H\Alg \to \C$ has a left adjoint, then invoking \ref{algfreemonadonendofunctorcor} and proving that the induced $\V$-monad is $\J$-ary.  The hypotheses that we impose on $\J$ to enable this proof strategy involve enriched factorization systems.  A standard reference on factorization systems for ordinary categories is e.g.\;\cite{FreydKelly}, while we refer the reader to \cite{enrichedfact} for the definition of an \emph{enriched} factorization system on a $\V$-category. We recall that an enriched factorization system $(\E, \M)$ is \emph{proper} if the left class $\E$ is contained in the $\V$-epimorphisms and the right class $\M$ in the $\V$-monomorphisms (defined as in \cite[0.1]{Dubucbook}). If the $\V$-category $\C$ is tensored and cotensored, then the $\V$-epimorphisms ($\V$-monomorphisms) in $\C$ coincide with the epimorphisms (monomorphisms) in $\C_0$ by \cite[2.4]{enrichedfact}, so that an enriched factorization system $(\E, \M)$ on $\C$ is proper iff the underlying ordinary factorization system on $\C_0$ is proper.  

\begin{defn_sub}
\label{factegory}
A \textbf{factegory} is a category $\C$ equipped with a factorization system $(\E, \M)$, while a factegory $\C$ is \textbf{proper} if $(\E,\M)$ is proper.  The factegory $\C$ is \textbf{cocomplete} if the category $\C$ is cocomplete and has arbitrary cointersections of $\E$-morphisms (i.e., wide pushouts of arbitrary families of $\E$-morphisms with the same domain). A factegory $\C$ is \textbf{$\E$-cowellpowered} if each $C \in \ob\C$ has just a (small) set of $\E$-quotients, i.e. isomorphism classes of $\E$-morphisms with domain $C$.
\qed 
\end{defn_sub}

\noindent Note that if $\C$ is a cocomplete factegory, then $\E$ consists of epimorphisms by \cite[1.3]{Kellytrans}.  In particular, this is true when $\C$ is an $\E$-cowellpowered factegory with small colimits.  Every category $\C$ can be equipped with the structure of an $\E$-cowellpowered factegory by taking $(\E, \M) = (\Iso, \All)$, where $\Iso$ is the class of all isomorphisms in $\C$ and $\All$ is the class of all morphisms in $\C$.

The term \textit{factegory} was recently introduced by the authors in \cite{locbd} in the case where $(\E,\M)$ is proper, along with the following terminology:

\begin{defn_sub}
\label{closedfactegory}
A \textbf{(symmetric monoidal) closed factegory} is a a symmetric monoidal closed category $\V$ equipped with a $\V$-enriched factorization system $(\E, \M)$. \qed
\end{defn_sub}

\begin{assumption_sub}\label{blanket_assumption}
For the remainder of the paper, we assume (along with our blanket assumptions in \S \ref{background}) that $\V$ is a closed factegory with enriched factorization system $(\E, \M)$. \qed
\end{assumption_sub}

\begin{defn_sub}
\label{compatible}
Let $\C$ be a $\V$-category equipped with an enriched factorization system $(\E_\C, \M_\C)$. Then $(\E_\C, \M_\C)$ \textbf{is compatible with} $(\E, \M)$ if each $\C(C, -) : \C \to \V$ ($C \in \ob\C$) preserves the right class (i.e. $m \in \M_\C$ implies $\C(C, m) \in \M$). \qed
\end{defn_sub}

\noindent This notion of compatibility for enriched factorization systems was first introduced by the authors in \cite{locbd}, in the case where the factorization systems are proper, and it leads to the following concept, which also was introduced in \cite{locbd} in the proper case:

\begin{defn_sub}
\label{Vfactegory}
A $\V$\textbf{-factegory} is a $\V$-category $\C$ equipped with an enriched factorization system $(\E_\C, \M_\C)$ compatible with $(\E, \M)$, while we say that $\C$ is \textbf{proper} if $(\E_\C,\M_\C)$ is proper. The $\V$-factegory $\C$ is \textbf{cocomplete} if the $\V$-category $\C$ is cocomplete and $\C$ has arbitrary (conical) cointersections of $\E_\C$-morphisms. \qed
\end{defn_sub}

\noindent Where this will not cause confusion, we shall (also) write $(\E, \M)$ for the enriched factorization system of a $\V$-factegory $\C$. We can now formulate the important notion of \emph{boundedness}, which is based on Kelly's taxonomy of the \textit{preservation of $\E$-tightness} of various classes of cocones by ordinary functors in \cite[2.3]{Kellytrans}:

\begin{defn_sub}
\label{boundedVfunctor}
Let $\C$ and $\D$ be $\V$-factegories with small conical colimits, and let $F : \C \to \D$ be a $\V$-functor.  Given a regular cardinal $\alpha$, we say that $F$ is $\alpha$\textbf{-bounded} if for every small $\alpha$-filtered diagram $D : \A \to \C_0$ and every $\M$\emph{-cocone} $m = \left(m_A : DA \to C\right)_{A \in \A}$ on $D$ (i.e. $m_A \in \M$ for all $A \in \ob\A$), if the induced morphism $\overline{m} : \colim \ D \to C$ lies in $\E$, then the induced morphism $\overline{Fm} : \colim \ FD \to FC$ lies in $\E$. To say that $F$ is $\alpha$-bounded in this sense is to say that $F$ \emph{preserves the} $\E$\emph{-tightness of small} $\alpha$\emph{-filtered} $\M$\emph{-cocones}, following the terminology of \cite[2.3]{Kellytrans}.  We say that $F : \C \to \D$ is \textbf{bounded} if $F$ is $\alpha$-bounded for some regular cardinal $\alpha$.

Given an object $C$ of a $\V$-factegory $\C$ with small conical colimits, we say that $C$ is ($\alpha$-)\textbf{bounded} if the $\V$-functor $\C(C, -) : \C \to \V$ is ($\alpha$-)bounded. \qed  
\end{defn_sub}

\noindent Note that if each of $(\E, \M), (\E_\C, \M_\C), (\E_\D, \M_\D)$ is the trivial factorization system $(\Iso, \All)$, then a $\V$-functor $F : \C \to \D$ is $\alpha$-bounded iff $F$ preserves small conical $\alpha$-filtered colimits. So the notion of $\alpha$-boundedness can be regarded as a factorization-system-theoretic variant of the notion of having rank $\leq \alpha$.

We mention in passing that the slightly weaker notion of \textit{preservation of the $\E$-tightness of $(\M,\alpha)$-cocones} \cite[2.3(i)]{Kellytrans} is defined analogously to $\alpha$-boundedness but employs $\alpha$-chains rather than $\alpha$-filtered diagrams in general.  This weaker notion is used in various of the existence results in \cite{Kellytrans} and, consequently, still suffices in enabling the central existence results in this paper (namely \ref{mainalgfreetheorem}, \ref{freemonadalgebrasaresignaturealgebras}, \ref{colimitcategoryofalgebras}) when substituted in place of ($\alpha$-)boundedness.

\begin{rmk_sub}
\label{boundedvsunions}
We now discuss how the boundedness of certain $\V$-functors between \textit{proper} $\V$-factegories admits an intuitive equivalent characterization in terms of \textit{$\M$-unions}.  This special case forms part of the foundation of the theory of \textit{locally bounded $\V$-categories}, recently introduced by the authors in \cite{locbd}, generalizing the locally bounded categories of Freyd and Kelly (\cite{FreydKelly}, \cite[\S 6.1]{Kelly}).

Let $\C$ be a cocomplete \textit{proper} $\V$-factegory.  A small sink $(m_i : C_i \to C)_{i \in I}$ (i.e. a small family of morphisms with shared codomain) is $\E$\emph{-tight} (or is \textit{jointly in $\E$}) if the induced morphism $\coprod_i C_i \to C$ lies in $\E$.  A sink $(m_i)_i$ is an $\M$\emph{-sink} if each $m_i \in \M$ ($i \in I$), while a small $\M$-sink $(m_i)_i$ is $\alpha$\emph{-filtered} if for every $J \subseteq I$ with $|J| < \alpha$, there is some $i \in I$ such that every $m_j$ ($j \in J$) factors through $m_i$. 

Given a $\V$-functor $F : \C \to \D$ between cocomplete proper $\V$-factegories, we say that $F$ \emph{preserves the} $\E$\emph{-tightness of} $\alpha$\emph{-filtered} $\M$\emph{-sinks} if $F$ sends every $\E$-tight $\alpha$-filtered $\M$-sink to an $\E$-tight sink.  Since $\C$ and $\D$ are proper, we can show that $F$ is $\alpha$-bounded iff $F$ preserves the $\E$-tightness of $\alpha$-filtered $\M$-sinks. To see this, first observe that if $m = \left(m_A : DA \to C\right)_{A \in \A}$ is a cocone on a small diagram $D : \A \to \C_0$, then $\overline{m} : \colim \ D \to C$ lies in $\E$ iff the cocone $m$ is an $\E$-tight sink. Indeed, the induced morphism $\left[m_A\right]_A:\coprod_A DA \to C$ factors through the canonical morphism $\coprod_A DA \to \colim D$, which is a regular epimorphism (by the construction of conical colimits from coproducts and coequalizers) and hence is an $\E$-morphism by properness. By composition and cancellation properties of factorization systems, it then follows that $\left[m_A\right]_A:\coprod_A DA \to C$ lies in $\E$ iff $\overline{m}:\colim D \to C$ lies in $\E$, as desired. Given this observation, the stated equivalence for $F : \C \to \D$ is now almost immediate. 

In a cocomplete proper $\V$-factegory $\C$ one can also define the \emph{union} of an $\M$-sink $(m_i : C_i \to C)_{i \in I}$ to be an $\M$-morphism $m : \bigcup_i C_i \to C$ through which the given $\M$-sink $(m_i)_i$ factors via an $\E$-tight sink $\left(e_i : C_i \to \bigcup_i C_i\right)_i$ (i.e. $m \circ e_i = m_i$ for each $i \in I$). The union of $(m_i)_i$ can be obtained as the $\M$-component of the $(\E, \M)$-factorization $\coprod_{i \in I} C_i \xrightarrow{e} \bigcup_{i \in I} C_i \xrightarrow{m} C$ of $[m_i]_i:\coprod_{i \in I} C_i \to C$. 

Given cocomplete proper $\V$-factegories $\C$ and $\D$ and a $\V$-functor $F : \C \to \D$ that preserves the right class (i.e. $m \in \M$ implies $Fm \in \M$), we say that $F$ \emph{preserves} ($\alpha$\emph{-filtered}) $\M$\emph{-unions} if whenever $(m_i)_i$ is an ($\alpha$-filtered) $\M$-sink in $\C$ with union $m$, then the $\M$-morphism $Fm$ is a union of the $\M$-sink $(Fm_i)_i$. It is remarked in \cite[4.31]{locbd} that if $F$ preserves the right class, then $F$ preserves $\alpha$-filtered $\M$-unions iff $F$ preserves the $\E$-tightness of $\alpha$-filtered $\M$-sinks, iff (by the previous paragraph) $F$ is $\alpha$-bounded in the sense of \ref{boundedVfunctor}. In particular, since $\C(C, -) : \C \to \V$ ($C \in \ob\C$) preserves the right class by \ref{compatible}, it follows that $C$ is $\alpha$-bounded iff $\C(C, -) : \C \to \V$ preserves $\alpha$-filtered $\M$-unions. This fact will be used in \ref{locallybounded} below. \qed
\end{rmk_sub}

\noindent Our motivation for considering bounded $\V$-functors is to make use of the following important result proved by Kelly in \cite[18.1]{Kellytrans}:

\begin{theo_sub}[Kelly \cite{Kellytrans}]
\label{Kellytheorem}
If $H : \C \to \C$ is a bounded endofunctor on a cocomplete factegory $\C$, then $U^H : H\Alg \to \C$ has a left adjoint. \qed
\end{theo_sub}

\noindent We now enrich \ref{Kellytheorem} as follows:

\begin{theo_sub}
\label{enrichedKellytheorem}
If $H : \C \to \C$ is a bounded $\V$-endofunctor on a cocomplete $\V$-factegory $\C$ that is cotensored, then the $\V$-functor $U^H : H\Alg \to \C$ has a left adjoint. 
\end{theo_sub}

\begin{proof}
The hypotheses imply that $\C_0$ is a cocomplete factegory and that $H_0 : \C_0 \to \C_0$ is a bounded endofunctor. So from \ref{Kellytheorem} we deduce that $U^{H_0} : H_0\Alg \to \C_0$ has a left adjoint, i.e. that $U^H_0 : H\Alg_0 \to \C_0$ has a left adjoint. Since $\C$ is cotensored, it follows by \ref{endofunctorcreationcolimits} that $H\Alg$ is cotensored and that $U^H : H\Alg \to \C$ preserves cotensors, so that $U^H : H\Alg \to \C$ then has a left adjoint by \cite[4.85]{Kelly}. 
\end{proof}

\noindent Using the notion of enriched $\alpha$-bounded object (\ref{boundedVfunctor}), we now introduce a new property of subcategories of arities that is central to this paper:

\begin{defn_sub}\label{bdd_sub_ar}
Let $j : \J \hookrightarrow \C$ be a subcategory of arities in a $\V$-factegory $\C$ with small conical colimits. Then $\J$ is $\alpha$\textbf{-bounded} if $\J$ is small and every $J \in \ob\J$ is \mbox{$\alpha$-bounded}, while $\J$ is \textbf{bounded} if there is a regular cardinal $\alpha$ for which $\J$ is $\alpha$-bounded. \qed
\end{defn_sub}

\begin{rmk_sub}
\label{boundedrmk}
In the situation of \ref{bdd_sub_ar}, assuming that $\J$ is small, we find that $\J$ is bounded iff each $J \in \ob\J$ is bounded. Indeed, if for each $J \in \ob\J$ there is a regular cardinal $\alpha_J$ such that $J$ is $\alpha_J$-bounded, then (because $\J$ is small) we can find a regular cardinal $\alpha$ larger than every $\alpha_J$, so that every $J \in \ob\J$ will be $\alpha$-bounded. \qed
\end{rmk_sub}

\begin{egg_sub}
\label{boundedexamples}
We now show that most of the (eleutheric) subcategories of arities from \ref{eleuthericexamples} are bounded: 
\begin{enumerate}[leftmargin=0pt,labelindent=0pt,itemindent=*,label=(\arabic*)]

\item Suppose that $\V$ is locally $\alpha$-presentable as a closed category and that $\C$ is a locally $\alpha$-presentable $\V$-category with associated subcategory of arities $j : \C_\alpha \hookrightarrow \C$ consisting of $\alpha$-presentable objects. Regarding $\V$ as a closed factegory by taking $(\E, \M)$ to be $(\Iso, \All)$, we may regard $\C$ as a $\V$-factegory equipped also with $(\Iso, \All)$.  For every $J \in \ob\C_\alpha$, the $\V$-functor $\C(J, -) : \C \to \V$ preserves small conical $\alpha$-filtered colimits and hence is $\alpha$-bounded.  So the subcategory of arities $j : \C_\alpha \hookrightarrow \C$ is $\alpha$-bounded.

In fact, \emph{every} small subcategory of arities $j : \J \hookrightarrow \C$ in a locally presentable $\V$-category $\C$ (over a locally presentable closed category $\V$) is bounded by \ref{boundedrmk}, because for each $C \in \ob\C$, there is some regular cardinal $\alpha$ for which $C$ is $\alpha$-presentable by \cite[7.4]{Kellystr}.

\item More generally, we shall show in \ref{generallocallyboundedprop} below that if $\C$ is a \textit{locally bounded $\V$-category} \cite{locbd} over a \textit{locally bounded} closed category $\V$ \cite[\S 6.1]{Kelly}, then \emph{every} small subcategory of arities in $\C$ is bounded, with respect to the proper factorization system carried by $\C$.

\item If $\V$ is a $\pi$-category in the sense of \cite{BorceuxDay}, then the eleutheric subcategory of arities $j : \SF(\V) \hookrightarrow \V$ is $\aleph_0$-bounded. Indeed, by the definition of $\pi$-category (see \cite[2.1.1]{BorceuxDay}), we know for each $V \in \ob\V$ that the functor $V \times (-) : \V \to \V$ preserves small filtered colimits. It then follows by a slight variation of the proof of \cite[3.8]{Kellystr} that $(-)^n : \V \to \V$ preserves small filtered colimits for each $n \in \N$. So if we take $(\E, \M) = (\Iso, \All)$, then for each $n \in \N$ it follows that $\V(n \cdot I, -) \cong \V(I, -)^n \cong (-)^n : \V \to \V$ preserves small filtered colimits and thus is $\aleph_0$-bounded.  

\item If we consider the subcategory of arities $j : \{I\} \hookrightarrow \V$ on the unit object in $\V$, then we can take $(\E, \M) = (\Iso, \All)$ and $\alpha := \aleph_0$, because $\V(I, -) \cong 1_\V : \V \to \V$ will certainly be $\aleph_0$-bounded. 

\item Let $\A$ be a small $\V$-category, and consider the subcategory of arities $\y_\A : \A^\op \hookrightarrow [\A, \V]$. If we equip both $\V$ and $[\A, \V]$ with the trivial enriched factorization system $(\Iso, \All)$, then $\y_\A : \A^\op \to [\A, \V]$ will be an $\aleph_0$-bounded subcategory of arities. Indeed, if $A \in \ob\A$, then the $\V$-functor $[\A, \V](\y_\A A, -) : [\A, \V] \to \V$ will be $\aleph_0$-bounded, i.e. will preserve small conical filtered colimits, because this $\V$-functor is isomorphic (by the enriched Yoneda lemma) to the cocontinuous evaluation $\V$-functor $\mathsf{Ev}_A : [\A, \V] \to \V$.      
  
\item Let $\Phi$ be a locally small class of small weights satisfying Axiom A of \cite{LR}, and let $\C = \Phi\Mod(\scrT)$ be the $\V$-category of models of a $\Phi$-theory $\scrT$.  Then, by \ref{loc_bdd_phitheory_example} below, if $\V$ is a locally bounded closed category that is $\E$-cowellpowered, then $\y_\Phi : \scrT^\op \hookrightarrow \C$ is a bounded subcategory of arities, with respect to an associated proper factorization system on $\C$.

\item Again letting $\C = \Phi\Mod(\scrT)$ as in (6), but for an arbitrary $\V$ as in \S \ref{background}, if there is a regular cardinal $\alpha$ for which every small conical $\alpha$-filtered weight is $\Phi$-flat, then $\y_\Phi:\scrT^\op \hookrightarrow \C$ will be an $\alpha$-bounded subcategory of arities with respect to the $(\Iso, \All)$ factorization systems on $\V$ and on $\C$.  Indeed, since $\y_\Phi$ presents $\C$ as a free cocompletion of $\scrT^\op$ under small $\Phi$-flat colimits, it follows by \cite[4.2(iv)]{KS} that each $\C(\y_\Phi T,-):\C \rightarrow \V$ ($T \in \ob\scrT$) preserves small $\Phi$-flat colimits and hence small conical $\alpha$-filtered colimits, so $\y_\Phi T$ is $\alpha$-bounded. 

\item As a special case of (7), suppose that $\Phi_\bbD$ is the class of small weights determined by a sound doctrine $\bbD$ as in \ref{eleuthericexamples}, where $\V$ is locally $\bbD$-presentable as a $\tensor$-category. Since $\bbD$ is a small class of small categories, we can find a regular cardinal $\alpha$ such that every $\D \in \bbD$ is $\alpha$-small. Then because $\Set$ is locally $\alpha$-presentable, small $\alpha$-filtered colimits commute in $\Set$ with $\alpha$-small limits and hence with $\D$-limits for every $\D \in \bbD$, so that small $\alpha$-filtered colimits are $\bbD$-filtered. Since the weights for small conical $\bbD$-filtered colimits are $\Phi_\bbD$-flat (see \cite[\S 5.4]{LR}), it then follows from (7) that $\y_\Phi : \scrT^\op \hookrightarrow \Phi_{\bbD}\Mod(\scrT)$ is an $\alpha$-bounded subcategory of arities for the $(\Iso,\All)$ factorization systems. \qed   
\end{enumerate}
\end{egg_sub}

\begin{egg_sub}[\textbf{Locally bounded \texorpdfstring{$\V$}{V}-categories}]
\label{locallybounded}
We now discuss in detail the general class of examples \ref{boundedexamples}(2), in locally bounded $\V$-categories $\C$ over locally bounded closed categories $\V$, and we show that any small subcategory of arities in such a $\V$-category $\C$ is bounded, with the useful consequence that all of our main results in this paper will hold for any small and eleutheric subcategory of arities in such a $\V$-category.

We begin by recalling the relevant definitions from \cite{locbd}. Let $\V$ be a cocomplete proper closed factegory (e.g. $\V = \Set$). If $\C$ is a cocomplete proper $\V$-factegory and $\G \subseteq \ob\C$ is small, then $\G$ is an \emph{enriched} $(\E, \M)$\emph{-generator} if for each $C \in \ob\C$ the canonical morphism $\coprod_{G \in \G} \C(G, C) \tensor G \to C$ lies in $\E$, where $\C(G, C) \tensor G$ is the $\V$-enriched tensor. Recall from \ref{boundedvsunions} that (in this context) an object $C \in \ob\C$ is $\alpha$-bounded iff $\C(C, -) : \C \to \V$ preserves $\alpha$-filtered $\M$-unions. A $\V$-category $\C$ is then \emph{locally} $\alpha$\emph{-bounded} if $\C$ is a cocomplete proper $\V$-factegory with an enriched $(\E, \M)$-generator consisting of $\alpha$-bounded objects.  In this case, it follows that $\C$ is also complete, by \cite[4.27]{locbd}.

A symmetric monoidal closed category $\V$ is then \emph{locally} $\alpha$\emph{-bounded (as a closed category)} if $\V_0$ is a locally $\alpha$-bounded ($\Set$-)category, the associated proper factorization system $(\E, \M)$ is $\V$-enriched, the unit object $I$ is $\alpha$-bounded, and $G \tensor G'$ is $\alpha$-bounded for all $G, G' \in \G$ (the $(\E, \M)$-generator of $\V_0$).  By the above, it then follows that $\V_0$ is also complete.

We refer the reader to \cite[\S 5.3]{locbd} for many examples of locally bounded closed categories, which include any locally presentable closed category, any cartesian closed topological category over $\Set$, any cocomplete locally cartesian closed category with a generator and arbitrary cointersections of epimorphisms, and various categories of models of symmetric monoidal limit theories in locally bounded and $\E$-cowellpowered closed categories.  We now have the following result:
\end{egg_sub}
 
\begin{prop_sub}
\label{generallocallyboundedprop}
Let $\C$ be a locally bounded $\V$-category over a locally bounded closed category $\V$. If $j : \J \hookrightarrow \C$ is a small subcategory of arities, then $\J$ is bounded.
\end{prop_sub} 

\begin{proof}
If $J \in \ob\J$, then by \cite[6.4]{locbd} there is a regular cardinal $\alpha$ such that $J$ is $\alpha$-bounded, so we deduce the result by \ref{boundedrmk} and the smallness of $\J$. 
\end{proof}

\begin{egg_sub}\label{loc_bdd_phitheory_example}
In the situation of \ref{boundedexamples}(6), if $\V$ is a locally bounded closed category that is $\E$-cowellpowered, then $\Phi\Mod(\scrT)$ carries the structure of a locally bounded $\V$-category by \cite[11.9]{locbd}, so the small subcategory of arities $\y_\Phi:\scrT^\op \hookrightarrow \Phi\Mod(\scrT)$ is bounded, by \ref{generallocallyboundedprop}. \qed
\end{egg_sub}

\subsection{Existence of free \texorpdfstring{$\J$}{J}-ary \texorpdfstring{$\V$}{V}-monads on \texorpdfstring{$\J$}{J}-ary \texorpdfstring{$\V$}{V}-endofunctors}
\label{existenceJarysection}

In this section, we prove that every $\J$-ary $\V$-endofunctor $H$ has an algebraically free $\V$-monad $\T_H$ that is $\J$-ary, provided that the subcategory of arities $j : \J \hookrightarrow \C$ is bounded and that $\C$ is a cocomplete $\V$-factegory and is cotensored.  We begin with two lemmas:

\begin{lem_sub}
\label{preservesEtightnessallcocones}
Let $F : \C \to \D$ be a $\V$-functor between cocomplete $\V$-factegories. If $F$ preserves small conical colimits and the left class (i.e. $e \in \E$ implies $Fe \in \E$), then $F$ preserves the $\E$-tightness of all small cocones. Explicitly, if $D : \A \to \C_0$ is a small diagram and $m = \left(m_A : DA \to C\right)_{A \in \A}$ is a cocone on $D$ whose induced morphism $\overline{m} : \colim D \to C$ lies in $\E$, then the induced morphism $\overline{Fm} : \colim FD \to FC$ lies in $\E$.
\end{lem_sub}

\begin{proof}
Since $F$ preserves small conical colimits, it suffices to show that the morphism $F\overline{m} : F\left(\colim D\right) \to FC$ lies in $\E$, which is true because $\overline{m} \in \E$ and $F$ preserves the left class.  
\end{proof}

\noindent If $\B$ is a small $\V$-category, then since $\V$ is a closed factegory (\ref{blanket_assumption}) the presheaf $\V$-category $[\B^\op, \V]$ is a $\V$-factegory \cite[4.7]{locbd}, whose enriched factorization system is defined \emph{pointwise}.

\begin{lem_sub}
\label{compatibilityproperty}
Let $D : \B \to \C$ be a $\V$-functor from a small $\V$-category $\B$ to a cocomplete $\V$-factegory $\C$. Then the colimit $\V$-functor $(-) \ast D : \left[\B^\op, \V\right] \to \C$ is a left adjoint that preserves the left class.   
\end{lem_sub}

\begin{proof}
The right adjoint of $(-) \ast D$ is the $\V$-functor $\C(D-, ?) : \C \to [\B^\op, \V]$ defined by $C \mapsto \C(D-, C)$, and this right adjoint preserves the right class because $(\E_\C, \M_\C)$ is compatible with $(\E, \M)$ and the right class of $[\B^\op, \V]$ is defined pointwise. It follows that the left adjoint $(-) \ast D$ preserves the left class \cite[3.4]{locbd}.
\end{proof}

\noindent The following result will be central for proving the main theorems of this paper:

\begin{prop_sub}
\label{Jarybounded}
Let $j : \J \hookrightarrow \C$ be an $\alpha$-bounded subcategory of arities in a cocomplete $\V$-factegory $\C$. Then every $\V$-endofunctor $H : \C \to \C$ that is a left Kan extension along $j$ is $\alpha$-bounded. In particular, every $\J$-ary $\V$-endofunctor $H : \C \to \C$ is $\alpha$-bounded.   
\end{prop_sub}

\begin{proof}
The last assertion follows from the previous one because every $\J$-ary $\V$-endofunctor is a left Kan extension along $j$ by \ref{non_eleu_jary_is_lan}. So let $H = \Lan_j D : \C \to \C$ for some $D:\J \rightarrow \C$. Then $H$ is isomorphic to the composite $\V$-functor 
\begin{equation}\label{eq:jary_endo_as_composite}\C \xrightarrow{\y_j} \left[\J^\op, \V\right] \xrightarrow{(-) \ast D} \C\end{equation}
where $\y_j$ is the restricted Yoneda $\V$-functor defined by $\y_j(C) = \C(j-, C)$ $(C \in \C)$. Since $\J$ is $\alpha$-bounded, it follows that $\y_j$ is $\alpha$-bounded.  Also, $(-) \ast D:\left[\J^\op, \V\right] \to \C$ is a left adjoint that preserves the left class, by \ref{compatibilityproperty}, so $(-) \ast D$ preserves the $\E$-tightness of all small cocones, by \ref{preservesEtightnessallcocones}, and hence the composite \eqref{eq:jary_endo_as_composite} preserves the $\E$-tightness of small $\alpha$-filtered $\M$-cocones, i.e. is $\alpha$-bounded. 
\end{proof}

\noindent The preceding result will enable us to invoke \ref{enrichedKellytheorem} to deduce the existence of the algebraically free $\V$-monad $\T_H$ on a $\J$-ary endofunctor $H$ under certain hypotheses, while the following will enable us to deduce that $\T_H$ is $\J$-ary:

\begin{lem_sub}
\label{freealgebramonadisJarycor}
Let $j : \J \hookrightarrow \C$ be a subcategory of arities in a $\Phi_{\underJ}$-cocomplete $\V$-category $\C$, let $H : \C \to \C$ be $\J$-ary, and suppose that the forgetful $\V$-functor $U^H : H\Alg \to \C$ has a left adjoint $F^H$.  Then the induced $\V$-monad $\T$ on $\C$ is $\J$-ary. 
\end{lem_sub}

\begin{proof}
Since $H$ preserves $\Phi_{\underJ}$-colimits, we deduce by  \ref{endofunctorcreationcolimits} that $U^H$ creates $\Phi_{\underJ}$-colimits, so $U^H$ preserves $\Phi_{\underJ}$-colimits since $\C$ has $\Phi_{\underJ}$-colimits, and the result follows.
\end{proof}

We now prove our main result of this section:

\begin{theo_sub}
\label{mainalgfreetheorem}
Let $j : \J \hookrightarrow \C$ be a bounded subcategory of arities in a cocomplete $\V$-factegory $\C$ that is cotensored. Every $\J$-ary $\V$-endofunctor $H$ on $\C$ has an algebraically free $\V$-monad $\T_H$ that is $\J$-ary, so that $U^H : H\Alg \to \C$ is strictly $\J$-monadic.  
\end{theo_sub}

\begin{proof}
Since $H$ is $\J$-ary and hence bounded by \ref{Jarybounded}, we deduce from \ref{enrichedKellytheorem} that $U^H : H\Alg \to \C$ has a left adjoint, and the resulting free $H$-algebra $\V$-monad $\T_H$ is $\J$-ary by \ref{freealgebramonadisJarycor}. We also know by \ref{algfreemonadonendofunctorcor} that $\T_H$ is an algebraically free $\V$-monad on $H$, and the result follows.    
\end{proof}

\noindent From \ref{algfreemonadonendofunctorisfree}, \ref{mainalgfreetheorem}, and \cite[Fact 3.1]{Porstmonoids} we now deduce: 

\begin{cor_sub}
\label{Whasleftadjoint}
Let $j : \J \hookrightarrow \C$ be a bounded subcategory of arities in a cocomplete $\V$-factegory $\C$ that is cotensored. Then the forgetful functor $\W : \Mnd_{\underJ}(\C) \to \End_{\underJ}(\C)$ is monadic. \qed     
\end{cor_sub}

\section{Free \texorpdfstring{$\J$}{J}-ary \texorpdfstring{$\V$}{V}-monads on \texorpdfstring{$\J$}{J}-signatures}
\label{freemonadsonsignatures}

In this section, we first define the notion of \emph{$\Sigma$-algebra} for a $\J$\emph{-signature} $\Sigma$  on a subcategory of arities $j : \J \hookrightarrow \C$, and then we show under certain hypotheses that every $\J$-signature has a free $\J$-ary $\V$-monad $\T_\Sigma$ whose $\V$-category of Eilenberg-Moore algebras is isomorphic to the $\V$-category of $\Sigma$-algebras.

\begin{defn}
\label{signatures}
Let $j : \J \hookrightarrow \C$ be a subcategory of arities in a $\V$-category $\C$. A \mbox{\textbf{$\J$-signature}} \textbf{(in $\C$)} is an ordinary functor $\Sigma : \ob\J \to \C_0$, where $\ob\J$ is the discrete category on the objects of $\J$. Thus, a $\J$-signature in $\C$ is just an $\ob\J$-indexed family of objects of $\C$. The category of $\J$-signatures is then $\JSig(\C) := \mathsf{CAT}\left(\ob\J, \C_0\right)$. \qed  
\end{defn}

\noindent In the locally presentable setting, the notion of signature for a subcategory of arities appears in \cite[Definition 37]{BourkeGarner} and (in a special case) in \cite[\S 5]{KellyPower}.

\begin{defn}
\label{signaturealgebras}
Let $\Sigma$ be a $\J$-signature for a subcategory of arities $j : \J \hookrightarrow \C$ in a tensored $\V$-category $\C$. A $\Sigma$\textbf{-algebra} (in $\C$) is a pair $(A, a)$ with $A \in \ob\C$ and $a = \left(a_J : \C(J, A) \tensor \Sigma J \to A\right)_{J \in \J}$ an $\ob\J$-indexed family of morphisms in $\C$. Given $\Sigma$-algebras $(A,a)$ and $(B,b)$, a \textbf{$\Sigma$-homomorphism} $f : (A, a) \to (B, b)$ is a morphism $f : A \to B$ in $\C$ such that the following diagram commutes for every $J \in \ob\J$:
\[\begin{tikzcd}
	\C(J, A) \tensor \Sigma J && \C(J, B) \tensor \Sigma J \\
	A && B
	\arrow["a_J"', from=1-1, to=2-1]
	\arrow["{\C(J, f) \tensor \Sigma J}", from=1-1, to=1-3]
	\arrow["b_J", from=1-3, to=2-3]
	\arrow["f", from=2-1, to=2-3]
\end{tikzcd}\]
We let $\Sigma\Alg_0$ be the ordinary category of $\Sigma$-algebras and $\Sigma$-homomorphisms. \qed 
\end{defn}

\begin{rmk}
\label{signaturealgebraremark}
If $\C$ is cotensored, then one may also define $\Sigma$-algebras in $\C$ in the following way, which is equivalent to the above if $\C$ is also tensored. A $\Sigma$\emph{-algebra} is a pair $(A, a)$ with $A \in \ob\C$ and $a = \left(a_J : \Sigma J \to [\C(J, A), A]\right)_{J \in \J}$ an $\ob\J$-indexed family of morphisms in $\C$, while a morphism $f : (A, a) \to (B, b)$ of such $\Sigma$-algebras in $\C$ is a morphism $f : A \to B$ in $\C$ such that the following diagram commutes for every $J \in \ob\J$:
\[\begin{tikzcd}
	{\Sigma J} && {[\C(J, A), A]} \\
	\\
	{[\C(J, B), B]} && {[\C(J, A), B]}
	\arrow["{a_J}", from=1-1, to=1-3]
	\arrow["{b_J}"', from=1-1, to=3-1]
	\arrow["{\left[\C(J, f), 1\right]}"', from=3-1, to=3-3]
	\arrow["{\left[1, f\right]}", from=1-3, to=3-3]
\end{tikzcd}\]
This definition of $\Sigma$-algebra is perhaps closer to the traditional notion from universal algebra. In that context, where $\C = \V = \Set$ and $\J = \FinCard$, a $\Sigma$-algebra (under this second definition) consists of a set $A$ and for each finite cardinal $n$, a function $a_n : \Sigma n \to \Set\left(A^n, A\right)$ that interprets each operation symbol $\omega \in \Sigma n$ of arity $n$ as an operation $\omega^A:A^n \to A$. Under Definition \ref{signaturealgebras}, we instead have $a_n : A^n \times \Sigma n \to A$, which sends each pair $(\vec{x}, \omega)$ to the value $\omega^A(\vec{x})$ of the operation $\omega^A$ at the input $\vec{x}$.  
\qed 
\end{rmk}

\begin{para}\label{vcat_sig_algs}
If $\C$ is tensored and $\V$ has \textit{pairwise equalizers} of $\ob\J$-indexed families of parallel pairs in the sense of \cite[2.1]{commutants} (which is true if $\J$ is small or $\V$ has wide intersections of strong monomorphisms), then the category $\Sigma\Alg_0$ underlies a $\V$-category $\Sigma\Alg$ defined as follows. For all $\Sigma$-algebras $(A, a)$ and $(B, b)$ in $\C$, we define the hom-object $\Sigma\Alg\left((A, a), (B, b)\right) \in \ob\V$ to be the pairwise equalizer of the following $\ob\J$-indexed family of parallel pairs: 
\[ \C(A, B) \xrightarrow{\C(a_J, 1)} \C(\C(J, A) \tensor \Sigma J, B)\;, \]
\[ \C(A, B) \xrightarrow{(\C(J, -) \tensor \Sigma J)_{AB}} \C(\C(J, A) \tensor \Sigma J, \C(J, B) \tensor \Sigma J) \xrightarrow{\C(1, b_J)} \C(\C(J, A) \tensor \Sigma J, B)\;. \] 
Thus we obtain a $\V$-category $\Sigma\Alg$ that is equipped with an evident faithful $\V$-functor $U^\Sigma : \Sigma\Alg \to \C$, and we call $\Sigma\Alg$ the \textbf{$\V$-category of $\Sigma$-algebras} (in $\C$). \qed
\end{para}
 
\noindent We now wish to show (under certain hypotheses) that every $\J$-signature $\Sigma$ generates a free $\J$-ary $\V$-endofunctor $H_\Sigma : \C \to \C$. 

\begin{para}[\textbf{The free $\V$-endofunctor on a $\J$-signature}]
\label{free_jary}
Let $j : \J \hookrightarrow \C$ be a small subcategory of arities in a cocomplete $\V$-category $\C$.  We write simply $\ob\J$ for the free $\V$-category on the discrete category $\ob\J$, which is just the discrete $\V$-category with object set $\ob\J$.  Hence $\JSig(\C) = \mathsf{CAT}\left(\ob\J, \C_0\right)$ may be identified with $\V\CAT\left(\ob\J, \C\right)$.  Writing $j':\ob\J \rightarrow \C$ for the canonical $\V$-functor, we obtain a `forgetful' functor
$$\V\CAT(j',\C)\;:\;\End(\C) \longrightarrow \JSig(\C)$$
that is given by restriction along $j'$ and sends each $\V$-endofunctor $H:\C \rightarrow \C$ to its \textbf{underlying $\J$-signature $(HJ)_{J \in \ob\J}$}.  This forgetful functor has a left adjoint
$$\Lan_{j'}\;:\;\JSig(\C) \longrightarrow \End(\C)$$
that is given by left Kan extension along $j'$ and so sends each $\J$-signature $\Sigma$ to the \textbf{\mbox{polynomial} $\V$-endofunctor}
$$H_\Sigma := \Lan_{j'}\Sigma\;:\;\C \longrightarrow\C\;,$$
given by
$$H_\Sigma C = \coprod_{J \in \J} \C(J, C) \tensor \Sigma J$$
$\V$-naturally in $C \in \C$.
\qed
\end{para}

\begin{prop}\label{hsigma_lan_jary_bdd}
Let $\Sigma$ be a $\J$-signature for a small subcategory of arities $j:\J \hookrightarrow \C$ in a cocomplete $\V$-category $\C$.  Then (1) $H_\Sigma:\C \rightarrow \C$ is a left Kan extension along $j$; (2) if $\J$ is eleutheric then $H_\Sigma$ is $\J$-ary; while (3) if $\C$ is a cocomplete $\V$-factegory and $\J$ is bounded then $H_\Sigma$ is bounded.
\end{prop}
\begin{proof}
Writing $i:\ob\J \rightarrow \J$ for the canonical $\V$-functor, we have that $j' = j \circ i$, so $H_\Sigma \cong \Lan_j\Lan_i\Sigma$ and hence (1) holds, and (2) now follows by \ref{charns_jary}, while (3) follows by \ref{Jarybounded}.
\end{proof}

\noindent In the situation of \ref{hsigma_lan_jary_bdd}, if $\J$ is eleutheric then $H_\Sigma$ is therefore a free \textit{$\J$-ary} $\V$-endofunctor on $\Sigma$, with respect to the evident forgetful functor $\End_{\underJ}(\C) \to \JSig(\C)$.  The latter functor is in fact monadic in this case:

\begin{prop}
\label{Vismonadic}
Let $j : \J \hookrightarrow \C$ be a small and eleutheric subcategory of arities in a cocomplete $\V$-category $\C$. Then the forgetful functor $\End_{\underJ}(\C) \to \JSig(\C)$ is monadic.
\end{prop}

\begin{proof}
In view of the equivalence $\V\CAT\left(\J, \C\right) \simeq \End_{\underJ}(\C)$ (see \ref{eleuthericequivalence}), it suffices to show that the functor $i^* = \V\CAT(i,\C):\V\CAT\left(\J, \C\right) \to \V\CAT\left(\ob\J, \C\right)$ is monadic, where $i:\ob\J \rightarrow \J$ is the canonical $\V$-functor. But $i^*$ has a left adjoint $\Lan_{i}$ and is conservative since $i$ is identity-on-objects, while $\V\CAT\left(\J, \C\right)$ has pointwise coequalizers preserved by $i^*$, so $i^*$ is monadic by Beck's monadicity theorem. 
\end{proof}

\noindent The proof of the following result is a straightforward generalization and enrichment of \cite[5.15]{AdamekMilius}.

\begin{prop}
\label{signaturefreeendofunctoralgebras}
Let $j : \J \hookrightarrow \C$ be a small subcategory of arities in a cocomplete $\V$-category $\C$, and let $\Sigma$ be a $\J$-signature.  Then $H_\Sigma\Alg \cong \Sigma\Alg$ in $\V\CAT/\C$. \qed 
\end{prop} 

\noindent  Throughout the rest of the paper, we write
$$\U : \Mnd_{\underJ}(\C) \longrightarrow \JSig(\C)$$
to denote the composite of the forgetful functors $\Mnd_{\underJ}(\C) \to \End_{\underJ}(\C) \to \JSig(\C)$. 

\begin{theo}
\label{freemonadalgebrasaresignaturealgebras}
Let $j : \J \hookrightarrow \C$ be a bounded and eleutheric subcategory of arities in a cocomplete $\V$-factegory $\C$ that is cotensored. Every $\J$-signature $\Sigma$ has a free $\J$-ary $\V$-monad $\T_\Sigma$ on $\C$ with $\T_\Sigma\Alg \cong \Sigma\Alg$ in $\V\CAT/\C$, so that the forgetful $\V$-functor $U^\Sigma : \Sigma\Alg \to \C$ is strictly $\J$-monadic. In particular, $\U : \Mnd_{\underJ}(\C) \to \JSig(\C)$ has a left adjoint.
\end{theo}
\begin{proof}
This follows immediately from \ref{algfreemonadonendofunctorisfree}, \ref{mainalgfreetheorem}, \ref{Vismonadic}, and \ref{signaturefreeendofunctoralgebras}.
\end{proof}

\noindent A weaker version of \ref{freemonadalgebrasaresignaturealgebras} holds even if $\J$ is \emph{not} eleutheric:   

\begin{theo}
\label{freemonadvariant}
Let $j : \J \hookrightarrow \C$ be a bounded subcategory of arities in a cocomplete $\V$-factegory $\C$ that is cotensored. Every $\J$-signature $\Sigma$ has a free $\V$-monad $\T_\Sigma$ on $\C$ with $\T_\Sigma\Alg \cong \Sigma\Alg$ in $\V\CAT/\C$, so that the forgetful $\V$-functor $U^\Sigma : \Sigma\Alg \to \C$ is strictly monadic. In particular, the forgetful functor $\Mnd(\C) \to \JSig(\C)$ has a left adjoint. 
\end{theo} 

\begin{proof}
By \ref{free_jary}, $H_\Sigma$ is a free $\V$-endofunctor on $\Sigma$.  Also, $H_\Sigma$ is bounded by \ref{hsigma_lan_jary_bdd}, so by \ref{enrichedKellytheorem} we deduce that $U^{H_\Sigma} : H_\Sigma\Alg \to \C$ has a left adjoint.  Hence, the induced $\V$-monad $\T_\Sigma$ on $\C$ is an algebraically free $\V$-monad on $H_\Sigma$ by \ref{algfreemonadonendofunctorcor} and so is a free $\V$-monad on $H_\Sigma$ by \ref{algfreemonadonendofunctorisfree}. The result now follows, using \ref{signaturefreeendofunctoralgebras}.
\end{proof}

\noindent Contrasting \ref{freemonadalgebrasaresignaturealgebras} and \ref{freemonadvariant}, the additional assumption of eleuthericity in \ref{freemonadalgebrasaresignaturealgebras} allows us to deduce that the free $\V$-endofunctor $H_\Sigma$ on a $\J$-signature $\Sigma$ is $\J$\emph{-ary}, which in turn allows us to deduce that the free $\V$-monad $\T_\Sigma$ is $\J$\emph{-ary}.  The latter property of $\T_\Sigma$ will be quite important in developing a theory of presentations and algebraic colimits for $\J$-ary monads in the remainder of the paper, particularly because it is only since $\T_\Sigma$ is $\J$-ary that we shall be able to deduce that $\T_\Sigma$ is bounded, which in turn is what enables us to form the quotients of $\T_\Sigma$ presented by systems of $\J$-ary equations over $\T_\Sigma$ (\ref{colimitcategoryofalgebras}, \ref{everypresentationpresentsJarymonad}).

\begin{egg}
\label{freemonadalgebrasexample}
Let us recall the familiar special case of Theorem \ref{freemonadalgebrasaresignaturealgebras} in the classical setting of universal algebra, with $\C = \V = \Set$ and $\J = \FinCard$.  Every ($\FinCard$-)signature $\Sigma$ has a free finitary endofunctor $H_\Sigma : \Set \to \Set$ and a free finitary monad $\T_\Sigma = \left(T_\Sigma, \eta^\Sigma, \mu^\Sigma\right)$ on $\Set$. Here, $H_\Sigma$ is the polynomial endofunctor induced by $\Sigma$, given by $H_\Sigma X = \coprod_n X^n \times \Sigma n$.  $H_\Sigma$-algebras may be identified with the traditional $\Sigma$-algebras of \ref{signaturealgebraremark} and, in turn, with Eilenberg-Moore algebras for the free $\Sigma$-algebra monad $\T_\Sigma$, which admits the following explicit description. For any set $X$, the set $T_\Sigma X$ is the set of \emph{terms} constructed from variables in $X$ and operation symbols in $\Sigma$. The function $\eta_X^\Sigma : X \to T_\Sigma X$ then sends each $x \in X$ to itself (\emph{qua} variable), while the function $\mu_X^\Sigma : T_\Sigma(T_\Sigma X) \to T_\Sigma X$ acts by substitution. \qed      
\end{egg}

\begin{rmk}
\label{freemonaddescription}
Under the assumptions of \ref{freemonadalgebrasaresignaturealgebras}, we now provide a more concrete description of the (algebraically) free $\J$-ary $\V$-monad on a $\J$-ary $\V$-endofunctor, and in particular of the free $\J$-ary $\V$-monad on a $\J$-signature. Recall that if $H : \C \to \C$ is $\J$-ary, then the (algebraically) free $\J$-ary $\V$-monad on $H$ is the free $H$-algebra $\V$-monad $\T_H$ on $\C$. In view of \ref{h_algebras}, the free $H$-algebra $T_H C$ on an object $C$ of $\C$ is equally the free $H_0$-algebra on $C$ for the underlying ordinary functor $H_0:\C_0 \rightarrow \C_0$, so we can obtain an explicit description of $T_H C$ in the case where $\C$ is $\E$-cowellpowered by consulting Kelly's description of free algebras for ordinary endofunctors in \cite[\S 20]{Kellytrans}.  Letting $\Ord$ denote the preordered class of ordinals, one defines a transfinite sequence $C_{(-)} : \Ord \to \C_0$ by setting $C_0 := C$, setting $C_{\beta + 1} := C + HC_\beta$ for each ordinal $\beta$, setting $C_\alpha := \colim_{\beta < \alpha} C_\beta$ for each limit ordinal $\alpha$, and defining $C_{(-)}$ suitably on inequalities $\alpha \leq \beta$.  The $\V$-functor $H : \C \to \C$ is $\J$-ary and hence bounded, by \ref{Jarybounded}, and therefore $H_0:\C_0 \rightarrow \C_0$ is bounded, so since $\C_0$ is $\E$-cowellpowered, it follows by \cite[15.6]{Kellytrans} that this sequence converges, i.e. there is an ordinal $\alpha$ with $C_\alpha \to C_{\alpha + 1}$ an isomorphism (see \cite[5.2]{Kellytrans}). By \cite[20.4]{Kellytrans}, it then follows that the underlying object of the free $H_0$-algebra on $C$ is $C_\alpha$, so that $T_HC = C_\alpha$. 

In particular, if $\Sigma$ is a $\J$-signature, so that $\Sigma$ has a free $\J$-ary $\V$-endofunctor $H_\Sigma$ on $\C$ and free $\J$-ary $\V$-monad $\T_\Sigma$ on $\C$, then $\T_\Sigma$ admits the following explicit description if $\C$ is $\E$-cowellpowered. With $H = H_\Sigma$ in the preceding paragraph and $C \in \ob\C$, the transfinite sequence $C_{(-)} : \Ord \to \C_0$ has the following form (recall the explicit description of $H_\Sigma$ from \ref{free_jary}): we have $C_0 := C$, we have $C_{\beta + 1} := C + H_\Sigma C_\beta = C + \coprod_{J \in \J} \C\left(J, C_\beta\right) \tensor \Sigma J$ for each ordinal $\beta$, and we have $C_\alpha := \colim_{\beta < \alpha} C_\beta$ for each limit ordinal $\alpha$. Since $H_\Sigma$ is $\J$-ary and hence bounded, this transfinite sequence converges at some ordinal $\alpha$, so that $T_\Sigma C = C_\alpha$. One may think of each $C_\beta$ as \emph{the} $\C$\emph{-object of $\Sigma$-terms of depth} $\leq \beta$ \emph{with variables from} $C$, and of $T_\Sigma C$ as \emph{the} $\C$\emph{-object of all $\Sigma$-terms with variables from} $C$. For example, in the classical situation of $j : \FinCard \hookrightarrow \Set$, where $T_\Sigma X$ is the set of all $\Sigma$-terms with variables from the set $X$, then an element of $X_{\beta + 1}$ for an ordinal $\beta$ is either a variable from $X$, or an element of $\Set\left(n, X_\beta\right) \times \Sigma n = X_\beta^n \times \Sigma n$ for some finite cardinal $n$, consisting of an $n$-ary operation symbol $\omega \in \Sigma n$ and an $n$-tuple of terms $t_1,...,t_n \in X_\beta$ of depth $\leq \beta$; the resulting term of depth $\leq \beta + 1$ is usual written as $\omega(t_1,...,t_n)$. \qed       
\end{rmk}

\subsection{Free strongly finitary \texorpdfstring{$\V$}{V}-monads on cartesian closed topological categories}\label{free_sf_vmnds_top_ccc}

We now pause to consider a class of examples in which $\V$ is a cartesian closed topological category over $\Set$, meaning that $\V$ is cartesian closed and the functor $V := \V_0(1, -) : \V_0 \to \Set$ is topological (see e.g.~\cite[21.1]{AHS}).  The functor $V$, which we write also as $|\text{$-$}|$, is then faithful \cite[21.3]{AHS}, so that morphisms in $\V_0$ may be regarded as certain functions between the underlying sets of objects of $\V_0$.  Also, $\V_0$ is complete and cocomplete, with limits and colimits formed as in $\Set$ (e.g. by \cite[21.16]{AHS}), so that $|\text{$-$}|$ strictly preserves limits and colimits.  

The subcategory of arities $\SF(\V) \hookrightarrow \V$ is eleutheric (\ref{eleuthericexamples}) and is bounded when $\V$ is equipped with its $(\Iso,\All)$ factorization system (\ref{boundedexamples}).  Strongly finitary $\V$-monads are the $\SF(\V)$-ary $\V$-monads for this subcategory of arities, and in this subsection we prove that every free $\SF(\V)$-ary $\V$-monad $\T_\Sigma$ on an $\SF(\V)$-signature $\Sigma$ is a \textit{lifting} of a free finitary monad $\T_{|\Sigma|}$ on $\Set$.  Here we write $|\Sigma|$ to denote the \textbf{underlying $\FinCard$-signature} of $\Sigma$, defined by $|\Sigma| (n) := |\Sigma n|$ for each $n \in \N$.  With this notation, a $\Sigma$-algebra $(A,a)$ is equivalently given by an object $A$ of $\V$ and a $|\Sigma|$-algebra structure on $|A|$ whose associated maps $A^n \times \Sigma(n) \rightarrow A$ are $\V$-morphisms.  In particular, each $\Sigma$-algebra $(A, a)$ has an \textbf{underlying $|\Sigma|$-algebra} $\left(|A|, |a|\right)$.  We obtain a commutative square
\begin{equation}\label{eq:tlsq}
\begin{tikzcd}
	\Sigma\Alg_0 & \V_0 \\
	{|\Sigma|\Alg} & \Set
	\arrow["U^\Sigma_0", from=1-1, to=1-2]
	\arrow["V^\Sigma"', from=1-1, to=2-1]
	\arrow["V\:=\:|\text{$-$}|", from=1-2, to=2-2]
	\arrow["U^{|\Sigma|}"', from=2-1, to=2-2]
\end{tikzcd}
\end{equation}
in which the functor $V^\Sigma$ sends each $\Sigma$-algebra to its underlying $|\Sigma|$-algebra.  We now show that this square satisfies the hypotheses of \textit{Wyler's taut lift theorem} (see \cite[21.28]{AHS}). 

\begin{prop_sub}\label{wtop}
Given an $\SF(\V)$-signature $\Sigma$ in a cartesian closed topological category $\V$ over $\Set$, the functor $V^\Sigma$ in \eqref{eq:tlsq} is topological, and $U^\Sigma_0$ sends $V^\Sigma$-initial sources to $V$-initial sources.
\end{prop_sub}
\begin{proof}
We shall omit from our notation all applications of the forgetful functors in \eqref{eq:tlsq}.  Given a $|\Sigma|$-algebra $A$ together with $\Sigma$-algebras $B_\lambda$ and $|\Sigma|$-homomorphisms $f_\lambda:A \rightarrow B_\lambda$ $(\lambda \in \Lambda)$, we can equip $A$ with the structure of a $\Sigma$-algebra such that the maps $f_\lambda$ constitute a $V^\Sigma$-initial source, as follows.  Since $V$ is topological, we may equip the set $A$ with the structure of an object of $\V$ such that the maps $f_\lambda$ constitute a $V$-initial source in $\V_0$.  Using the $V$-initiality of the $f_\lambda$ and the fact that the $f_\lambda$ are $|\Sigma|$-homomorphisms, we can readily show that the maps $a_n:A^n \times |\Sigma|(n) \rightarrow A$ carried by the $|\Sigma|$-algebra $A$ are morphisms $a_n:A^n \times \Sigma(n) \rightarrow A$ in $\V_0$.  Thus we may regard $A$ as a $\Sigma$-algebra.  Since the $|\Sigma|$-homomorphisms $f_\lambda$ are also morphisms in $\V_0$, it follows that these maps $f_\lambda$ are also morphisms in $\Sigma\Alg$.  The $V$-initiality of $(f_\lambda:A \rightarrow B_\lambda)_{\lambda \in \Lambda}$ now entails also its $V^\Sigma$-initiality.
\end{proof}

\begin{para_sub}\label{para_lifting}
Adapting Beck's \cite{Beck} terminology to the setting of a cartesian closed topological category $\V$ over $\Set$ (cf. also \cite[\S 5]{ADV}), we say that an (ordinary) monad $\T = (T,\eta,\mu)$ on $\V_0$ is a \textit{strict lifting} of a monad $\mathbb{S} = (S,e,m)$ on $\Set$ if $VT = SV$, $V\eta = eV$, and $V\mu = mV$.  It then follows that $S = VTD$, $e = V\eta D$, and $m = V \mu D$, where $D$ is the left adjoint section of $V$ \cite[21.12]{AHS}.  $\T$ is a \textit{non-strict lifting} of $\mathbb{S}$ if there is an isomorphism $\varphi:VT \xrightarrow{\sim} SV$ such that $\varphi \cdot V\eta = eV$ and $mV \cdot S\varphi \cdot \varphi T = \varphi \cdot V\mu$.  We say that a $\V$-monad $\T$ on $\V$ is a \textit{strict} (resp. \textit{non-strict}) \textit{lifting} of a monad $\mathbb{S}$ on $\Set$ if its underlying ordinary monad $\T_0$ is a strict (resp. non-strict) lifting of $\mathbb{S}$.\qed
\end{para_sub}

\begin{prop_sub}
\label{topologicalliftingprop}
Let $\V$ be a cartesian closed topological category over $\Set$, and let $\Sigma$ be an $\SF(\V)$-signature. The free $\SF(\V)$-ary $\V$-monad $\T_\Sigma$ is a non-strict lifting of the free finitary monad $\T_{|\Sigma|}$ on $\Set$.  Moreover, we may construct $\T_\Sigma$ in such a way that it is a strict lifting of $\T_{|\Sigma|}$.
\end{prop_sub} 
\begin{proof}  
By \ref{wtop}, the square \eqref{eq:tlsq} satisfies the hypothesis of Wyler's taut lift theorem \cite[21.28]{AHS}, which therefore entails the result. 
\end{proof} 

\noindent In the situation of \ref{topologicalliftingprop}, if $X$ is an object of $\V$ then the underlying set of
$T_\Sigma X$ is therefore the set $T_{|\Sigma|}|X|$ of terms over $|\Sigma|$ with variables in the set $|X|$ underlying $X$. 

\section{Monadicity of \texorpdfstring{$\J$}{J}-ary \texorpdfstring{$\V$}{V}-monads over \texorpdfstring{$\J$}{J}-signatures} 
\label{monadicity}

Let $j:\J \hookrightarrow \C$ be a subcategory of arities satisfying the hypotheses of \ref{freemonadalgebrasaresignaturealgebras}, where we have shown that the forgetful functor $\U : \Mnd_{\underJ}(\C) \to \JSig(\C)$ has a left adjoint. We now show that $\U$ is \emph{monadic}. In the special case where $\C$ is a locally $\alpha$-presentable $\V$-category over a locally $\alpha$-presentable closed category $\V$ and $\J = \C_\alpha$, this was achieved by Lack in \cite{Lackmonadicity}. 

We first define a functor $\di : \End_{\underJ}(\C) \times \Sig_{\underJ}(\C) \to \Sig_{\underJ}(\C)$ by $\di(H, \Sigma) := H\Sigma$. So $\di$ is a strict action (see \cite[1.4b]{BJK}) of the strict monoidal category $\End_{\underJ}(\C)$ on $\JSig(\C)$, and the forgetful functor $\mathcal{X}:\End_{\underJ}(\C) \to \JSig(\C)$ strictly preserves the actions, in the sense that $\di\left(H, \mathcal{X}\left(H'\right)\right) = \mathcal{X}\left(H \circ H'\right)$ naturally in $H, H' \in \End_{\underJ}(\C)$. We now recall \cite[Theorem 2]{Lackmonadicity}:

\begin{theo}[Lack \cite{Lackmonadicity}]
\label{Lacktheorem}
Let $\mathscr{E}$ be a monoidal category such that $(-) \tensor H : \mathscr{E} \to \mathscr{E}$ preserves coequalizers for each $H \in \ob\mathscr{E}$. Let $\mathscr{B}$ be a category with a functor $\di : \mathscr{E} \times \mathscr{B} \to \mathscr{B}$, and let $\mathcal{X} : \mathscr{E} \to \mathscr{B}$ be a monadic functor with isomorphisms $\di(H, \mathcal{X}(H')) \cong \mathcal{X}(H \tensor H')$ natural in $H,H' \in \mathscr{E}$. If the forgetful functor $\W : \mathsf{Mon}(\mathscr{E}) \to \mathscr{E}$ has a left adjoint, then the composite $ \mathsf{Mon}(\mathscr{E}) \xrightarrow{\W} \mathscr{E} \xrightarrow{\mathcal{X}} \mathscr{B}$ is monadic. \qed
\end{theo}

\noindent In \cite[Theorem 2]{Lackmonadicity}, the stronger assumption of $\mathscr{E}$ being a \emph{right-closed} monoidal category is made; however, upon inspection of the proof of \cite[Theorem 2]{Lackmonadicity}, it is clear that the only use of this assumption occurs in the proof of \cite[Lemma 1]{Lackmonadicity}, in which one just uses the fact that each functor $(-) \tensor X : \mathscr{E} \to \mathscr{E}$ ($X \in \ob\mathscr{E}$) preserves coequalizers. 

Taking $\mathscr{E}$ to be the strict monoidal category $\End_{\underJ}(\C)$ and taking $\mathscr{B}$ to be the category $\JSig(\C)$ of $\J$-signatures, we now want to use \ref{Lacktheorem} to show the monadicity of 
\[ \U = \left(\Mnd_{\underJ}(\C) = \mathsf{Mon}\left(\End_{\underJ}(\C)\right) \xrightarrow{\W} \End_{\underJ}(\C) \xrightarrow{\mathcal{X}} \JSig(\C)\right). \]

\begin{theo}
\label{Uismonadic}
Let $j : \J \hookrightarrow \C$ be a bounded and eleutheric subcategory of arities in a cocomplete $\V$-factegory $\C$ that is cotensored. The forgetful functor $\U : \Mnd_{\underJ}(\C) \to \JSig(\C)$ is monadic.  
\end{theo}

\begin{proof}
Since colimits commute with colimits, $\End_{\underJ}(\C)$ is closed under pointwise colimits in $\End(\C)$, so for each $H \in \End_{\underJ}(\C)$ the functor $(-) \circ H : \End_{\underJ}(\C) \to \End_{\underJ}(\C)$ preserves small colimits.  Also, $\mathcal{X} : \End_{\underJ}(\C) \to \JSig(\C)$ is monadic by \ref{Vismonadic}, and $\W : \Mnd_{\underJ}(\C) \to \End_{\underJ}(\C)$ has a left adjoint by \ref{Whasleftadjoint}. Since the required natural isomorphisms for $\di$ and $\mathcal{X}$ hold (as equalities) as discussed above, we now deduce the result from \ref{Lacktheorem}.   
\end{proof}

\section{Algebraic colimits of \texorpdfstring{$\J$}{J}-ary \texorpdfstring{$\V$}{V}-monads}
\label{algebraiccolimits}

In this section we study algebraic colimits of ($\J$-ary) $\V$-monads. We first develop some supporting material on weighted limits and colimits in limit $\V$-categories.  

\subsection{Weighted limits and colimits in limit \texorpdfstring{$\V$}{V}-categories}
\label{limitVcategories}

We first recall the description of small limits in $\V\CAT$. So let $\K$ be a small category and $\A : \K \to \V\CAT$ a functor, so that for each $k \in \ob\K$ we have a $\V$-category $\A_k$ and for each morphism $f : k \to k'$ in $\K$ we have a $\V$-functor $\A_f : \A_k \to \A_{k'}$. Since $\V$ is complete, we obtain a limit $\V$-category $\limit \A$ with the following explicit description. The objects of $\limit\A$ are $\ob\K$-indexed families $A = \left(A_k\right)_{k \in \K}$ with $A_k \in \ob\A_k$ for each $k \in \ob\K$, satisfying the compatibility condition $\A_f\left(A_k\right) = A_{k'}$ for each morphism $f : k \to k'$ in $\K$. Given two such families $A, B \in \ob\left(\limit\A\right)$, the hom-object $\limit\A(A, B) \in \ob\V$ is the limit of the functor $\A_{(-)}(A_{(-)}, B_{(-)}) : \K \to \V$ that sends $k \in \ob\K$ to $\A_k\left(A_k, B_k\right)$ and sends $f : k \to k'$ to the structural morphism $({\A_f})_{A_kB_k}:\A_k\left(A_k, B_k\right) \to \A_{k'}\left(A_{k'}, B_{k'}\right)$. We then have projection $\V$-functors $P_k : \limit \A \to \A_k$ for each $k \in \ob\K$ with $\A_f \circ P_k = P_{k'}$ for each $f : k \to k'$ in $\K$. If $D : \B \to \limit\A$ is a $\V$-functor, then we write $D_k := P_kD : \B \to \A_k$ for each $k \in \ob\K$.

\begin{para_sub}
\label{pointwiselimitcylinders}
Let $W:\B \to \V$ be a weight.  Regarding $\V\CAT$ as a $\Set'$-category for some category of (large) sets $\Set'$, there is a functor $\mathsf{Cyl}(W,-):\V\CAT \rightarrow \Set'$ that sends each $\V$-category $\C$ to the set $\mathsf{Cyl}(W,\C)$ of all triples $(D,C,\gamma)$ in which $D:\B \rightarrow \C$ is a $\V$-functor and $(C,\gamma)$ is a cylinder on the weighted diagram $(W,D)$.  Writing $\II$ for the unit $\V$-category, with $\ob\II = \{*\}$, we may regard $W$ as a $\V$-functor $W:\II^\op \otimes \B \rightarrow \V$ and so as a $\V$-profunctor.  The \textit{collage} of $W$ is the $\V$-category $\mathsf{Coll}_W$ with objects $\ob\II + \ob\B$, with homs $\mathsf{Coll}_W(X,Y)$ defined as $\II(X,Y)$ if $X,Y \in \ob\II$, as $\B(X,Y)$ if $X,Y \in \ob\B$, as $W(X,Y)$ if $X \in \ob\II$ and $Y \in \ob\B$, and as $0$ if $X \in \ob\B$ and $Y \in \ob\II$, with the evident composition and identities.  There is an evident $\V$-functor $B:\B \rightarrow \mathsf{Coll}_W$ with the property that $\mathsf{Coll}_W(*,B-) = W:\B \rightarrow \V$, so that $(*,1_W)$ is a cylinder on the weighted diagram $(W,B)$ in $\mathsf{Coll}_W$ and hence $(B,*,1_W) \in \mathsf{Cyl}(W,\mathsf{Coll}_W)$.

The functor $\mathsf{Cyl}(W,-)$ is representable, as $\mathsf{Cyl}(W,-) \cong \V\CAT(\mathsf{Coll}_W,-)$ with unit $(B,*,1_W)$.  Hence $\mathsf{Cyl}(W,-)$ preserves limits, so if $\A : \K \to \V\CAT$ is any small diagram, then $\mathsf{Cyl}(W,\limit \A) \cong \limit_k\mathsf{Cyl}(W,\A_k)$.  Hence, given any $\V$-functor $D:\B \rightarrow \limit\A$, a $(W,D)$-cylinder $(C,\gamma)$ is equivalently given by a family of $(W,D_k)$-cylinders $(C_k,\gamma_k) = (P_k C,P_k\gamma)$ $(k \in \ob\K)$ that is \textit{compatible} in the sense that $\left(\A_f C_k, \A_f \gamma_k\right) = (C_{k'}, \gamma_{k'})$ for each $f : k \to k'$ in $\K$. \qed
\end{para_sub}

\begin{prop_sub}
\label{jointlyreflectlimits}
Let $\A : \K \to \V\CAT$ be a small diagram. The projection $\V$-functors $P_k : \limit \A \to \A_k$ ($k \in \ob\K$) jointly reflect limits. Explicitly, if $(W : \B \to \V, D : \B \to \limit \A)$ is a weighted diagram in $\limit \A$ and $(L, \lambda)$ is a cylinder for $(W, D)$ such that for each $k \in \ob\K$, the associated cylinder $\left(L_k, \lambda_k\right)$ for $\left(W, D_k\right)$ is a limit cylinder, then $(L, \lambda)$ is itself a limit cylinder. 
\end{prop_sub}

\begin{proof}
Although $\B$ is not assumed small, we may form the $\V'$-category $[\B, \V]$, with the notation of \cite[\S 3.11, 3.12]{Kelly}. Now fix $X \in \ob\left(\limit\A\right)$ and consider the following diagram in $\V'$ for each $k \in \ob\K$: 
\[\begin{tikzcd}
	{\mathsf{lim}\A(X, L)} &&& {[\mathscr{B}, \mathscr{V}](W, \mathsf{lim}\A(X, D-))} \\
	\\
	{\A_k(X_k, L_k)} &&& {[\mathscr{B}, \mathscr{V}]\left(W, \A_k\left(X_k, D_k-\right)\right)}
	\arrow[from=1-1, to=1-4]
	\arrow["{\left(P_k\right)_{XL}}"', from=1-1, to=3-1]
	\arrow["{[\mathscr{B}, \mathscr{V}]\left(1, P_k\right)}", from=1-4, to=3-4]
	\arrow["\sim"', from=3-1, to=3-4]
\end{tikzcd}\]
The lower morphism is induced by $\lambda_k$ as in \cite[3.3]{Kelly} and is an isomorphism since $\left(L_k, P_k\lambda\right)$ is a limit cylinder, while the upper morphism is defined similarly. The morphisms $\left(P_k\right)_{XL}$ ($k \in \ob\K$) constitute a limit cone, as do the morphisms $[\B, \V]\left(1, P_k\right)$ ($k \in \ob\K$). Since the diagram commutes for all $k$ (because $P_k\lambda = \lambda_k$), it then follows that the upper morphism is an isomorphism, as required.
\end{proof}

\noindent Now \ref{pointwiselimitcylinders} and \ref{jointlyreflectlimits} immediately entail the following result:

\begin{cor_sub}
\label{pointwiselimits}
Let $\A : \K \to \V\CAT$ be a small diagram, and let $(W,D)$ be a weighted diagram in $\limit\A$.  Then a limit cylinder $(L,\lambda)$ for $(W,D)$ is equivalently given by a family of limit cylinders $(L_k,\lambda_k)$ for the $(W,D_k)$ $(k \in \ob\K)$ that is compatible in the sense of \ref{pointwiselimitcylinders}. \qed
\end{cor_sub}

\noindent We now wish to consider \emph{colimits} in a limit $\V$-category. First, we have the automorphism of (mere) categories $(-)^\op : \V\CAT \xrightarrow{\sim} \V\CAT$ that (therefore) preserves limits, so that for a given small diagram $\A : \K \to \V\CAT$ we have $\left(\limit \A\right)^\op = \left(\limit_k \A_k\right)^\op = \limit_k \A_k^\op$.  Hence, a dually weighted diagram $(W,D)$ in $\limit \A$ is equivalently given by a weighted diagram $(W, D^\op)$ in $\limit_k \A_k^\op$.  Therefore, by \ref{pointwiselimitcylinders}, a cylinder $(C,\gamma)$ on the dually weighted diagram $(W,D)$ in $\limit \A$ is equivalently given by a compatible family of cylinders $(C_k,\gamma_k)$ $(k \in \ob\K)$ on the weighted diagrams $(W,D^\op_k)$ in $\A_k^\op$, which is equally a compatible family of cylinders on the dually weighted diagrams $(W,D_k)$ in $\A_k$.  Thus, \ref{pointwiselimits} entails the following dual result, for colimits:

\begin{cor_sub}
\label{pointwisecolimits}
Let $\A : \K \to \V\CAT$ be a small diagram, and let $(W,D)$ be a dually weighted diagram in $\limit\A$.  Then a colimit cylinder $(L,\lambda)$ for $(W,D)$ is equivalently given by a compatible family of colimit cylinders $(L_k,\lambda_k)$ for the $(W,D_k)$ $(k \in \ob\K)$. \qed   
\end{cor_sub}

\noindent We now wish to consider small limits in the slice category $\V\CAT/\C$ for a $\V$-category $\C$. So fix a small category $\K$, and let $\K^\top$ be the category obtained from $\K$ by adjoining a terminal object $\top$. Then limits of shape $\K$ in $\V\CAT/\C$ can be defined in terms of limits of shape $\K^\top$ in $\V\CAT$, which we briefly review as follows. Let $\A : \K \to \V\CAT/\C$ be a functor, so that for each $k \in \ob\K$ we have a $\V$-functor $U_k : \A_k \to \C$, and for each morphism $f : k \to k'$ in $\K$ we have a $\V$-functor $\A_f : \A_k \to \A_{k'}$ with $U_{k'} \circ \A_f = U_k$. We then obtain an induced functor $\A^\top : \K^\top \to \V\CAT$ as follows: $\A^\top(k) := \A_k \in \V\CAT$ for each $k \in \ob\K$ and $\A^\top(\top) := \C \in \V\CAT$, $\A^\top(f) := \A_f : \A_k \to \A_{k'}$ for each morphism $f : k \to k'$ in $\K$, and $\A^\top(!_k) := U_k : \A_k \to \C$ for each $k \in \ob\K$, where $!_k : k \to \top$ is the unique arrow from $k$ to the terminal object $\top$ in $\K^\top$. We then have the limit $\V$-category $\limit\A^\top$, which we write simply as $\limit\A$, with projection $\V$-functors $P_k : \limit\A \to \A_k$ for each $k \in \ob\K$ and $U : \limit\A \to \C$, with $U_k \circ P_k = U : \limit \A \to \C$ for each $k \in \ob\K$ and $\A_f \circ P_k = P_{k'} : \limit \A \to \A_{k'}$ for each $f : k \to k'$ in $\K$. It is then a standard result about limits in slice categories that $U : \limit\A \to \C$ with the limit projections $P_k : \limit\A \to \A_k$ ($k \in \ob\K$) is the limit of $\A : \K \to \V\CAT/\C$. From \ref{pointwiselimits} we now deduce the following:

\begin{theo_sub}
\label{limitfunctorcreateslimits}
Let $\C$ be a $\V$-category, and let $\A : \K \to \V\CAT/\C$ be a small diagram with limit $U : \limit\A \to \C$. Let $\Lambda$ be a class of weighted diagrams with the property that $(W,D_k) \in \Lambda$ for every weighted diagram $(W,D) \in \Lambda$ in $\limit\A$ and every $k \in \ob\K$. Suppose that $U_k : \A_k \to \C$ creates $\Lambda$-limits for each $k \in \ob\K$. Then $U : \limit \A \to \C$ creates $\Lambda$-limits.
\end{theo_sub} 

\begin{proof}
Let $(W : \B \to \V, D : \B \to \limit \A)$ be a weighted diagram in $\Lambda$, and let $(L, \lambda)$ be a limit cylinder for $(W, UD : \B \to \C)$. We must show that there is a unique cylinder $\left(\Lbar, \lambdabar\right)$ for $(W, D)$ with $\left(U\Lbar, U\lambdabar\right) = (L, \lambda)$, and moreover that $\left(\Lbar, \lambdabar\right)$ is a limit cylinder for $(W, D)$. For each $k \in \ob\K$, since $U = U_kP_k$, the limit cylinder $(L, \lambda)$ is a limit cylinder for $\left(W, U_kD_k\right)$. Hence for each $k \in \ob\K$, since $U_k : \A_k \to \C$ creates $\Lambda$-limits and $\left(W, D_k\right) \in \Lambda$, there is a unique cylinder $(\Lbar_k, \lambdabar_k)$ for $\left(W, D_k\right)$ with $\left(U_k\Lbar_k, U_k\lambdabar_k\right) = (L, \lambda)$, and moreover $(\Lbar_k, \lambdabar_k)$ is a limit cylinder for $\left(W, D_k\right)$. It is now essentially immediate that the limit cylinders $(L, \lambda), (\Lbar_k, \lambdabar_k)$ ($k \in \ob\K$) form a compatible family in the sense of \ref{pointwiselimitcylinders} (for the induced functor $\A^\top : \K^\top \to \V\CAT$). So by \ref{pointwiselimitcylinders} and \ref{pointwiselimits}, we obtain a unique cylinder $\left(\Lbar, \lambdabar\right)$ for $(W, D)$ with the properties $\left(P_k\Lbar, P_k\lambdabar\right) = (\Lbar_k, \lambdabar_k)$ for each $k \in \ob\K$ and $\left(U\Lbar, U\lambdabar\right) = (L, \lambda)$, and moreover $\left(\Lbar, \lambdabar\right)$ is a limit cylinder for $(W, D)$. The required uniqueness of $\left(\Lbar, \lambdabar\right)$ easily follows from the uniqueness of the cylinders $(\Lbar_k, \lambdabar_k)$ ($k \in \ob\K$).       
\end{proof} 

\noindent We now obtain a dual result, for colimits:

\begin{cor_sub}
\label{limitfunctorcreatescolimits}
Let $\C$ be a $\V$-category, and let $\A : \K \to \V\CAT/\C$ be a small diagram with limit $U : \limit\A \to \C$. Let $\Lambda$ be a class of dually weighted diagrams with the property that $(W,D_k) \in \Lambda$ for every $(W,D) \in \Lambda$ in $\limit\A$ and every $k \in \ob\K$. Suppose that $U_k : \A_k \to \C$ creates $\Lambda$-colimits for each $k \in \ob\K$. Then $U : \limit \A \to \C$ creates $\Lambda$-colimits.
\end{cor_sub} 
\begin{proof}
There is an isomorphism of (mere) categories $(-)^\op:\V\CAT\slash \C \rightarrow \V\CAT \slash \C^\op$, so $U^\op:(\lim \A)^\op \rightarrow \C^\op$ is a limit of the composite diagram $(-)^\op \circ \A:\K \to \V\CAT\slash \C^\op$, and hence the result follows from \ref{limitfunctorcreateslimits}.
\end{proof}

\subsection{Algebraic colimits of \texorpdfstring{$\V$}{V}-monads in general}
\label{algebraiccolimitssubsection}

We now use the material in \ref{limitVcategories} to study algebraic colimits of $\V$-monads on a $\V$-category $\C$.  We begin with enriched generalizations of concepts from \cite[\S 26]{Kellytrans}, some of which we reformulate in terms of the semantics functor.  Let $\MM : \K \to \Mnd(\C)$ be a small diagram. For each $k \in \ob\K$ we thus have a $\V$-monad $\MM_k$ on $\C$, and for every morphism $f : k \to k'$ in $\K$ we have a morphism of $\V$-monads $\MM_f : \MM_k \to \MM_{k'}$. By composing with the semantics functor (\ref{semantics_early}), we obtain a functor $\K^\op \xrightarrow{\MM^\op} \Mnd(\C)^\op \xrightarrow{\ALG} \V\CAT/\C\;,$ whose limit we denote by
$$U^\MM : \MM\Alg \to \C\;.$$
An object of $\MM\Alg$ will be called an $\MM$\emph{-algebra} and is a pair $(A, a)$ with $A \in \ob\C$ and $a = (a_k : M_kA \to A)_{k \in \K}$ a family of $\C$-morphisms with $(A, a_k) \in \MM_k\Alg$ for each $k \in \K$, satisfying the compatibility condition $a_{k'} \circ \left(\MM_f\right)_A = a_k : M_kA \to A$ for each $f : k \to k'$ in $\K$. 

In \cite[\S 26]{Kellytrans}, Kelly defined the concept of \textit{algebraic colimit} in terms of cones and their induced morphisms, but we now reformulate its enriched generalization in terms of (co)limits preserved by the semantics functor:

\begin{defn_sub}
\label{algcolimit}
Let $\C$ be a $\V$-category and $\MM : \K \to \Mnd(\C)$ a small diagram. An \textbf{algebraic colimit of} $\MM$ is a colimit $\T = \colim \MM$ in $\Mnd(\C)$ that is sent to a limit by the semantics functor $\ALG : \Mnd(\C)^\op \to \V\CAT/\C$. An algebraic colimit of $\MM$ is therefore unique up to isomorphism if it exists, in which case we denote it by $\T_\MM$. \qed
\end{defn_sub}

\noindent The semantics functor $\ALG$ is fully faithful (\ref{semantics_early}) and therefore reflects limits, so in view of the definition of $\MM\Alg$ above we obtain the following equivalent characterization of algebraic colimits (but see \ref{algcolimitleftadjoint} for a stronger result on when they exist):

\begin{prop_sub}
\label{algcolimitequivalent}
Let $\C$ be a $\V$-category and $\MM : \K \to \Mnd(\C)$ a small diagram. Then a $\V$-monad $\T_\MM$ on $\C$ is an algebraic colimit of $\MM$ iff $\T_\MM\Alg \cong \MM\Alg$ in $\V\CAT/\C$.   Hence, $\MM$ has an algebraic colimit iff $\MM\Alg$ is a strictly monadic $\V$-category over $\C$.\qed
\end{prop_sub}  

\begin{para_sub}
\label{algcolimitalt}
Writing $\Monadic^! \hookrightarrow \V\CAT/\C$ for the full subcategory consisting of the strictly monadic $\V$-categories over $\C$, the semantics functor $\ALG$ of \ref{semantics_early} corestricts to an equivalence
\begin{equation}\label{sem_equiv}\ALG : \Mnd(\C)^\op \xrightarrow{\sim} \Monadic^!\;.\end{equation}
Hence a small diagram $\MM : \K \to \Mnd(\C)$ has an algebraic colimit if and only if the composite $\ALG \circ \MM^\op : \Mnd(\C)^\op \to \Monadic^!$ has a limit that is preserved by the inclusion $\Monadic^! \hookrightarrow \V\CAT/\C$, in which case the algebraic colimit of $\MM$ is then the $\V$-monad corresponding to this limit under the equivalence \eqref{sem_equiv}.\qed
\end{para_sub} 

\noindent The following entails that $\MM\Alg$ is strictly monadic over $\C$ as soon as $U^\MM$ has a left adjoint:
\begin{prop_sub}
\label{limitfunctorstrictlymonadic}
Let $\C$ be a $\V$-category, let $\A : \K \to \V\CAT/\C$ be a small diagram, and write $U : \limit \A \to \C$ for the limit of $\A$ in $\V\CAT/\C$. If $U_k : \A_k \to \C$ is strictly monadic for each $k \in \ob\K$ and $U$ has a left adjoint, then $U$ is strictly monadic. 
\end{prop_sub}

\begin{proof}
It suffices by \cite[II.2.1]{Dubucbook} to show that $U$ creates conical coequalizers of $U$-split pairs. But this follows from \ref{limitfunctorcreatescolimits} since each $U_k$ is strictly monadic and so creates conical coequalizers of $U_k$-split pairs by \cite[II.2.1]{Dubucbook}.  
\end{proof}

\noindent The following enrichment of \cite[26.4]{Kelly} now follows:

\begin{prop_sub}
\label{algcolimitleftadjoint}
Let $\C$ be a $\V$-category and $\MM : \K \to \Mnd(\C)$ a small diagram. Then $\MM$ has an algebraic colimit $\T_\MM$ iff the $\V$-functor $U^\MM : \MM\Alg \to \C$ has a left adjoint $F^\MM$; and then the algebraic colimit $\T_\MM$ is the $\V$-monad arising from the adjunction $F^\MM \dashv U^\MM$. 
\end{prop_sub}

\begin{proof}
By \ref{limitfunctorstrictlymonadic}, $U^\MM$ is strictly monadic iff $U^\MM$ has a left adjoint, so the result now follows from \ref{algcolimitequivalent}.
\end{proof} 

\subsection{Existence of algebraic colimits of \texorpdfstring{$\J$}{J}-ary \texorpdfstring{$\V$}{V}-monads}
\label{algebraiccolimitsJary}

Having studied algebraic colimits of general $\V$-monads, we now wish to investigate algebraic colimits of $\J$\emph{-ary} $\V$-monads for a subcategory of arities $j : \J \hookrightarrow \C$.  

\begin{defn_sub}
\label{algcolimitJary}
Let $\C$ be a $\V$-category with a subcategory of arities $j : \J \hookrightarrow \C$, and let $\MM : \K \to \Mnd_{\underJ}(\C)$ be a small diagram. A $\J$-ary $\V$-monad $\T$ on $\C$ is a \textbf{(}$\J$\textbf{-ary) algebraic colimit of} $\MM$ if $\T$ is a colimit of $\MM$ that is sent to a limit by the restricted semantics functor $\ALG : \Mnd_{\underJ}(\C)^\op \to \V\CAT/\C$. As special cases, we obtain the notions of ($\J$-ary) \textit{algebraic coequalizer}, \textit{algebraic coproduct}, etcetera. \qed
\end{defn_sub}

\begin{para_sub}\label{jary_alg_colim_rmk}
Let $\MM : \K \to \Mnd_{\underJ}(\C)$ be a small diagram.  If $\T$ is a $\J$-ary algebraic colimit of $\MM$, then $\T$ is also an algebraic colimit of the composite diagram $\K \xrightarrow{\MM} \Mnd_{\underJ}(\C) \hookrightarrow \Mnd(\C)$, because the fully faithful $\ALG : \Mnd(\C)^\op \to \V\CAT/\C$ reflects limits. Hence, by \ref{algcolimitequivalent}, $\MM$ has a $\J$-ary algebraic colimit iff $\MM\Alg$ is a strictly $\J$-monadic $\V$-category over $\C$, where we also write $\MM:\K \rightarrow \Mnd(\C)$ to denote the composite diagram. \qed
\end{para_sub}

\noindent We shall make use of the following important result proved by Kelly in \cite[27.1]{Kellytrans}. 

\begin{theo_sub}[Kelly \cite{Kellytrans}]
\label{secondKellytheorem}
Let $\C$ be a cocomplete factegory, and suppose either that $\C$ is proper or that $\C$ is $\E$-cowellpowered. If $\MM : \K \to \Mnd(\C)$ is a small diagram for which each $M_k$ ($k \in \ob\K$) is bounded, then the functor $U^\MM : \MM\Alg \to \C$ has a left adjoint. \qed
\end{theo_sub}

\noindent We now enrich \ref{secondKellytheorem} as follows:

\begin{theo_sub}
\label{secondenrichedKellytheorem}
Let $\C$ be a cocomplete $\V$-factegory that is cotensored, and suppose either that $\C$ is proper or that $\C$ is $\E$-cowellpowered. If $\MM : \K \to \Mnd(\C)$ is a small diagram for which each $M_k$ ($k \in \ob\K$) is bounded, then the $\V$-functor $U^\MM : \MM\Alg \to \C$ has a left adjoint.  
\end{theo_sub}

\begin{proof}
By composing $\MM$ with the forgetful functor $\Mnd(\C) \rightarrow \Mnd(\C_0)$ we obtain a functor that we shall write as $\MM_0:\K \rightarrow \Mnd(\C_0)$.  The hypotheses on $\C$ and $\MM$ entail that $\C_0$ and $\MM_0$ satisfy the hypotheses of \ref{secondKellytheorem}, whence $U^{\MM_0}:\MM_0\Alg \rightarrow \C_0$ has a left adjoint.  But $\MM\Alg_0 \cong \MM_0\Alg$ in $\mathsf{CAT}/\C_0$, so $U^\MM_0$ has a left adjoint.  Since each $U^{\MM_k}:\MM_k\Alg \rightarrow \C$ creates cotensors by \cite[II.4.7]{Dubucbook}, we deduce by \ref{limitfunctorcreateslimits} that $U^{\MM}$ creates cotensors, so since $\C$ is cotensored, we find that $\MM\Alg$ is cotensored and $U^\MM$ preserves cotensors. Hence $U^\MM$ has a left adjoint by \cite[4.85]{Kelly}. 
\end{proof}

\begin{assumption_sub}
For the remainder of \S \ref{algebraiccolimitsJary}, we assume that $j : \J \hookrightarrow \C$ is a small subcategory of arities in a cocomplete $\V$-factegory $\C$ that is cotensored, and we also suppose either that $\C$ is proper or that $\C$ is $\E$-cowellpowered. \qed
\end{assumption_sub} 

\begin{prop_sub}
\label{limitfunctorstrictlyJmonadic}
Let $\A : \K \to \V\CAT/\C$ be a small diagram with limit $U : \limit\A \to \C$. Suppose that each $U_k : \A_k \to \C$ ($k \in \ob\K$) is strictly $\J$-monadic and that $U$ has a left adjoint.  Then $U$ is strictly $\J$-monadic. 
\end{prop_sub} 

\begin{proof}
Since each $U_k$ ($k \in \ob\K$) is strictly monadic, it follows by \ref{limitfunctorstrictlymonadic} that $U$ is strictly monadic.  Also, each $U_k$ is strictly $\J$-monadic and hence creates $\Phi_{\underJ}$-colimits by \ref{str_jmon_prop}, so by \ref{limitfunctorcreatescolimits} we find that the strictly monadic $\V$-functor $U$ creates $\Phi_{\underJ}$-colimits and so is strictly $\J$-monadic by \ref{str_jmon_prop}.
\end{proof}  

\begin{prop_sub}
\label{freediagramalgebramonadisJary}
Let $\MM : \K \to \Mnd_{\underJ}(\C)$ be a small diagram. Then a $\J$-ary algebraic colimit $\T_\MM$ exists iff $U^\MM : \MM\Alg \to \C$ has a left adjoint.
\end{prop_sub}  

\begin{proof}
By \ref{jary_alg_colim_rmk}, a $\J$-ary algebraic colimit $\T_\MM$ exists iff $U^\MM$ is strictly $\J$-monadic, so since $\MM\Alg$ is a small limit in $\V\CAT/\C$ of strictly $\J$-monadic $\V$-categories over $\C$, the result follows from \ref{limitfunctorstrictlyJmonadic}.   
\end{proof} 

\noindent We now prove our main result of this section:

\begin{theo_sub}
\label{colimitcategoryofalgebras}
Suppose that the subcategory of arities $j : \J \hookrightarrow \C$ is bounded. Then every small diagram $\MM : \K \to \Mnd_{\underJ}(\C)$ has a $\J$-ary algebraic colimit $\T_\MM$. Hence $\Mnd_{\underJ}(\C)$ has small colimits, which are algebraic.     
\end{theo_sub}

\begin{proof}
Each $M_k$ ($k \in \ob\K$) is $\J$-ary and hence bounded by \ref{Jarybounded}, so the result follows from \ref{secondenrichedKellytheorem} and \ref{freediagramalgebramonadisJary}.
\end{proof} 

\noindent We conclude this section with the following result, which is an immediate corollary of \ref{colimitcategoryofalgebras} and the fact that the equivalence $\ALG : \Mnd(\C)^\op \xrightarrow{\sim} \Monadic^!$ restricts to an equivalence $\ALG_{\underJ} : \Mnd_{\underJ}(\C) \xrightarrow{\sim} \Monadic^!_{\underJ}$, where $\Monadic^!_{\underJ}$ is the full and replete subcategory of $\Monadic^!$ consisting of the strictly $\J$-monadic $\V$-categories over $\C$.   

\begin{cor_sub}
\label{strJmonadicclosedunderlimits}
Suppose that the subcategory of arities $j : \J \hookrightarrow \C$ is bounded. Then the full subcategory $\Monadic^!_{\underJ} \hookrightarrow \V\CAT/\C$ is closed under small limits. \qed  
\end{cor_sub}

\begin{rmk_sub}
\label{Kellytransrmk}
In \cite[27.2]{Kellytrans}, Kelly showed that if $\C$ is a cocomplete factegory that is $\E$-cowellpowered, and $\MM : \K \to \Mnd(\C)$ is a small diagram for which there is some regular cardinal $\alpha$ such that each $M_k$ ($k \in \ob\K$) preserves the $\E$-tightness of \emph{all} small $\alpha$-filtered cocones (not just $\M$\emph{-cocones}), then the algebraic colimit $\T \in \Mnd(\C)$ also preserves the $\E$-tightness of all small $\alpha$-filtered cocones. Just before the statement of \cite[27.2]{Kellytrans}, Kelly remarked that he did not see how to show that this result holds when replacing \emph{all} small $\alpha$-filtered cocones by just $\alpha$-filtered $\M$\emph{-cocones} in the hypothesis and conclusion. Our \ref{colimitcategoryofalgebras} provides an alternative to this proposition that Kelly was not able to prove (while our result is also enriched and does not require the $\E$-cowellpoweredness of $\C$): namely, if each of the $\V$-monads $\MM_k$ ($k \in \ob\K$) is $\J$-ary for an $\alpha$-bounded subcategory of arities $\J \hookrightarrow \C$, then each $M_k$ ($k \in \ob\K$) preserves the $\E$-tightness of small $\alpha$-filtered $\M$-cocones by \ref{Jarybounded}, and the resulting algebraic colimit $\T \in \Mnd(\C)$ is also $\J$-ary (by \ref{colimitcategoryofalgebras}) and hence preserves the $\E$-tightness of small $\alpha$-filtered $\M$-cocones by \ref{Jarybounded}. \qed 
\end{rmk_sub}

\section{Presentations of \texorpdfstring{$\J$}{J}-ary \texorpdfstring{$\V$}{V}-monads}
\label{presentationssection}

In this section, we use our results on algebraically free $\J$-ary $\V$-monads and algebraic colimits of $\J$-ary $\V$-monads to show that every $\J$-\textit{presentation}, consisting of a $\J$-signature $\Sigma$ and a \textit{system of $\J$-ary equations} over $\T_\Sigma$, presents a $\J$-ary $\V$-monad whose algebras are the $\Sigma$-algebras that \textit{satisfy} the given equations, and that, conversely, every $\J$-ary $\V$-monad has a canonical $\J$-presentation.

\subsection{\texorpdfstring{$\J$}{J}-presentations}
\label{Jpresentations} 

If $\T$ is a $\J$-ary $\V$-monad on $\C$, then we refer to $\U(\T) \in \JSig(\C)$ as the \emph{underlying} $\J$\emph{-signature} of $\T$. If $\Gamma$ is a $\J$-signature, then by an abuse of notation we write a morphism of $\J$-signatures $\Gamma \to \U(\T)$ as $\Gamma \to \T$.  

\begin{assumption_sub}
For the remainder of \S \ref{presentationssection}, we assume that $j : \J \hookrightarrow \C$ is a bounded and eleutheric subcategory of arities in a cocomplete $\V$-factegory $\C$ that is cotensored. We also suppose either that $\C$ is proper or that $\C$ is $\E$-cowellpowered. \qed
\end{assumption_sub}

\noindent Under these assumptions, we know by Theorem \ref{Uismonadic} that $\U : \Mnd_{\underJ}(\C) \to \JSig(\C)$ is monadic, and by Theorem \ref{colimitcategoryofalgebras} that $\Mnd_{\underJ}(\C)$ has small algebraic colimits.

\begin{defn_sub}
\label{systemequations}
Given a $\J$-ary $\V$-monad $\T$ on $\C$, a \textbf{system of $\J$-ary equations over $\T$} is a parallel pair of $\J$-signature morphisms $E = (\terml, \termr : \Gamma \rightrightarrows \T)$ for some $\J$-signature $\Gamma$ (called the \textit{signature of equations}).  A \textbf{system of $\J$-ary equations} is a pair $(\T, E)$ consisting of a $\J$-ary $\V$-monad $\T$ and a system of $\J$-ary equations $E$ over $\T$.\qed 
\end{defn_sub}

\begin{defn_sub}
\label{presentation}
A \textbf{$\J$-ary $\V$-monad presentation} (or $\J$\textbf{-presentation}) is a pair $P = (\Sigma, E)$ consisting of a $\J$-signature $\Sigma$ and a system of $\J$-ary equations $E$ over $\T_\Sigma$.  We then also call $P$ a $\J$\textbf{-presentation over $\Sigma$}. \qed 
\end{defn_sub}

\noindent We shall sometimes abuse notation and just write $P = (\terml, \termr : \Gamma \rightrightarrows \T_\Sigma)$ for a $\J$-presentation.

\begin{egg_sub}
\label{presentationexample}
In the classical setting of universal algebra, with $\C = \V = \Set$ and $\J = \FinCard$, Definition \ref{presentation} describes the usual data by which a variety of algebras is presented, as we now explain.  Given a ($\FinCard$-)signature $\Sigma$, a \textit{formal equation} $t \doteq u$ over $\Sigma$ with variables in a finite cardinal $n \in \N$ is, by definition, a pair of terms $t,u \in T_\Sigma(n)$.  A $\FinCard$-presentation $P = (\Sigma,E)$ consists of a signature $\Sigma \in \Set^\N$ and a parallel pair $E = (\terml, \termr : \Gamma \rightrightarrows \T_\Sigma)$, which we call a \textit{system of finitary equations} since it is a family of formal equations $t_{n\gamma} \doteq u_{n\gamma}$ over $\Sigma$ with variables in $n$, indexed by the finite cardinals $n \in \N$ and the elements $\gamma \in \Gamma(n)$ of an $\N$-graded set $\Gamma \in \Set^\N$.\qed
\end{egg_sub} 

\begin{para_sub}\label{tbar}
If $t:\Gamma \to \T$ is a morphism of $\J$-signatures valued in (the $\J$-signature underlying) a $\J$-ary $\V$-monad $\T$, then we write $\termlbar:\T_\Gamma \to \T$ for the morphism of $\J$-ary $\V$-monads induced by $t$, where $\T_\Gamma$ is the free $\J$-ary $\V$-monad on $\Gamma$ (\ref{freemonadalgebrasaresignaturealgebras}). \qed
\end{para_sub}

\begin{defn_sub}
\label{monadpresented}
Let $(\T,E)$ be a system of $\J$-ary equations, where $E = (\terml, \termr : \Gamma \rightrightarrows \T)$.  A \textbf{quotient} of $\T$ by $E$ is, by definition, a $\J$-ary algebraic coequalizer (\ref{algcolimitJary}) of the parallel pair $\termlbar, \termrbar : \T_\Gamma \rightrightarrows \T$
in $\Mnd_{\underJ}(\C)$.  A quotient of $\T$ by $E$ is unique up to isomorphism if it exists, in which case it is a $\J$-ary $\V$-monad that we denote by $\T/E$.  We say that a $\J$-ary $\V$-monad $\mathbb{S}$ \textbf{is presented by} $(\T,E)$ if $\mathbb{S}$ is a quotient $\T/E$, in which case we also say that $(\T,E)$ \textbf{presents} $\mathbb{S}$. \qed
\end{defn_sub}

\begin{defn_sub}
Let $P = (\Sigma,E)$ be a $\J$-presentation.  A $\J$-ary $\V$-monad $\T_P$ \textbf{is presented by} $P$ if $\T_P$ is presented by the system of $\J$-ary equations $(\T_\Sigma,E)$, i.e., if $\T_P$ is a quotient $\T_\Sigma / E$. \qed
\end{defn_sub}

\noindent By Theorem \ref{colimitcategoryofalgebras}, we obtain the following important result:

\begin{theo_sub}
\label{everypresentationpresentsJarymonad}
Every $\J$-presentation $P = (\Sigma,E)$ presents a $\J$-ary $\V$-monad $\T_P = \T_{\Sigma}/E$.  More generally, every system of $\J$-ary equations $(\T,E)$ presents a $\J$-ary $\V$-monad $\T/E$.  \qed 
\end{theo_sub} 

\noindent Next we show that every $\J$-ary $\V$-monad is presented by at least one $\J$-presentation.

\begin{para_sub}[\textbf{The canonical presentation of a $\J$-ary $\V$-monad}]
Since $\U : \Mnd_{\underJ}(\C) \to \JSig(\C)$ is monadic by Theorem \ref{Uismonadic}, we know that every $\J$-ary $\V$-monad $\T$ on $\C$ is canonically a coequalizer, as in
\begin{equation}\label{eq:can_coeq}
\begin{tikzcd}
	{\mathcal{F}\mathcal{U}\mathcal{F}\mathcal{U}\T} && {\mathcal{F}\mathcal{U}\T} & \T,
	\arrow["{\mathcal{F}\mathcal{U}\varepsilon \T}"', shift right=1, from=1-1, to=1-3]
	\arrow["{\varepsilon \T}", two heads, from=1-3, to=1-4]
	\arrow["{\varepsilon \mathcal{F}\mathcal{U}\T}", shift left=1, from=1-1, to=1-3]
\end{tikzcd}
\end{equation}
where $\mathcal{F}$ is the left adjoint of $\mathcal{U}$ and $\varepsilon$ is the counit of this adjunction.  Using the triangular equations and the naturality of the unit $\eta$ of the adjunction $\mathcal{F} \dashv \mathcal{U}$, the transposes of $\varepsilon \mathcal{F}\mathcal{U}\T$ and $\mathcal{F}\mathcal{U}\varepsilon \T$ under the adjunction $\F \dashv \U$ may be expressed as the following $\J$-signature morphisms
\begin{equation}\label{eq:can_pres1}
\begin{tikzcd}
	{\U\F\U\T} & & & {\U\F\U\T},
	\arrow["\eta \U\T \,\circ\, \U\varepsilon \T"', shift right=1, from=1-1, to=1-4]
	\arrow["1_{\U\F\U\T}", shift left=1, from=1-1, to=1-4]
\end{tikzcd}
\end{equation}
which may be regarded as a system of $\J$-ary equations, as follows.  Letting $\Sigma_\T = \U\T$ and $\Gamma_\T = \U\F\U\T$, note that $\F\U\T = \T_{\Sigma_\T}$ with the notation of \ref{freemonadalgebrasaresignaturealgebras}.  Hence the morphisms \eqref{eq:can_pres1} constitute a system of $\J$-ary equations $E_\T = (t_\T,u_\T:\Gamma_\T \rightrightarrows \T_{\Sigma_\T})$. Thus we obtain a $\J$-presentation $P_\T = (\Sigma_\T,E_\T)$, and since the coequalizer \eqref{eq:can_coeq} is algebraic by \ref{colimitcategoryofalgebras}, we deduce that $P_\T$ presents $\T$.  We call $P_\T$ the \textbf{canonical $\J$-presentation of $\T$}.  In particular, this proves the following corollary of \ref{Uismonadic} and \ref{colimitcategoryofalgebras}:
\end{para_sub}

\begin{cor_sub}
\label{presentationcor}
Every $\J$-ary $\V$-monad $\T$ on $\C$ has a $\J$-presentation. \qed
\end{cor_sub}

\subsection{Algebras for \texorpdfstring{$\J$}{J}-presentations} 
\label{Palgebras}

Given a $\J$-presentation $P = (\Sigma,E)$, which presents a $\J$-ary $\V$-monad $\T_P$ by \ref{everypresentationpresentsJarymonad}, we now show that the $\V$-category of $\T_P$-algebras is isomorphic to a full sub-$\V$-category of $\Sigma\Alg$ consisting of those $\Sigma$-algebras that \emph{satisfy} the system of $\J$-ary equations $E$ (\ref{presentationsatisfaction}).  More generally, we show that if $(\T,E)$ is any system of $\J$-ary equations over a $\J$-ary $\V$-monad $\T$, then $(\T/E)\Alg$ is isomorphic to the full sub-$\V$-category of $\T\Alg$ consisting of the $\T$-algebras that satisfy $E$ in the sense of \ref{generalizedsatisfaction} below.

\begin{defn_sub}\label{interp_hom}
Let $(A,a)$ be a $\T$-algebra for a $\J$-ary $\V$-monad $\T = (T, \eta, \mu)$ on $\C$.  For each $J \in \ob\J$, we write
$$\llb-\rrb^A_J \;:\; TJ \longrightarrow [\C(J, A), A]$$
to denote the transpose of the composite $\C(J,A) \xrightarrow{T_{JA}} \C(TJ,TA) \xrightarrow{\C(TJ,a)} \C(TJ,A)$, and we call $\llb-\rrb^A_J$ the \textbf{interpretation morphism} for the $\T$-algebra $(A,a)$ and the arity $J$.  Given a morphism $t:C \rightarrow TJ$ in $\C$, we call
$$\llb t \rrb^A_J \;:=\; \llb-\rrb^A_J \circ t\;:\;C \rightarrow [\C(J,A),A]$$
the \textbf{interpretation\footnote{Similar terminology and notation are used in \cite{termequationalsystems} in a different setting.} of $t$ in $A$}. \qed
\end{defn_sub}

\begin{defn_sub}
\label{generalizedsatisfaction}
Let $(\T,E)$ be a system of $\J$-ary equations, where $E = (\terml, \termr : \Gamma \rightrightarrows \T)$, and let $(A, a)$ be a $\T$-algebra. Then $(A, a)$ \textbf{satisfies} $E$, or is a $(\T,E)$\textbf{-algebra}, if
$$\llb t_J \rrb^A_J = \llb u_J \rrb^A_J\;:\;\Gamma J \rightarrow [\C(J,A),A]$$
for all $J \in \ob\J$.  We let $(\T,E)\Alg$ be the full sub-$\V$-category of $\T\Alg$ consisting of the $(\T,E)$-algebras. 
\qed 
\end{defn_sub} 

\noindent Given a $\J$-signature $\Sigma$, we know that $\Sigma\Alg \cong \T_\Sigma\Alg$ in $\V\CAT/\C$ (\ref{freemonadalgebrasaresignaturealgebras}).  Thus, in formulating the following definition, we may regard $\Sigma$-algebras equivalently as $\T_\Sigma$-algebras.

\begin{defn_sub}
\label{presentationsatisfaction}
Let $P = (\Sigma, E)$ be a $\J$-presentation, so that $E$ is a system of $\J$-ary equations over $\T_\Sigma$.  We say that a $\Sigma$-algebra $(A,a)$ \textbf{satisfies $E$}, or is a \textbf{$P$-algebra}, if its corresponding $\T_\Sigma$-algebra satisfies $E$.  We write $P\Alg$ to denote the full sub-$\V$-category of $\Sigma\Alg$ consisting of the $P$-algebras. \qed
\end{defn_sub}

\begin{para_sub}
\label{palg_iso_to_tsigmaealg}
Given a $\J$-presentation $P = (\Sigma,E)$, $P\Alg \cong (\T_\Sigma,E)\Alg$ in $\V\CAT \slash \C$ by \ref{freemonadalgebrasaresignaturealgebras}. \qed
\end{para_sub}

\begin{egg_sub}
\label{satisfyingpresentationexample}
Let us see how \ref{presentationsatisfaction} generalizes the standard notion from universal algebra, where $\C = \V = \Set$ and $\J = \FinCard$. If $\Sigma$ is a $\FinCard$-signature and $(A, a)$ is a $\Sigma$-algebra, so that $A$ is a set equipped with operations $\omega^A:A^n \rightarrow A$ $(n \in \N, \omega \in \Sigma n)$, then for each $n \in \N$ the interpretation morphism $\llb-\rrb^A_n : T_\Sigma(n) \to \Set\left(A^n, A\right)$ sends each term $t$ in $n$ variables to its interpretation $\llb t\rrb^A_n : A^n \to A$. Given a $\FinCard$-presentation $P = (\Sigma,E)$, recall from \ref{presentationexample} that $E$ is a family of formal equations $t_{n\gamma} \doteq u_{n\gamma}$ indexed by the elements $\gamma \in \Gamma n$ $(n \in \N)$ of an $\N$-graded set $\Gamma$.  A $P$-algebra is a $\Sigma$-algebra $(A, a)$ that satisfies each of these formal equations $t_{n\gamma} \doteq u_{n\gamma}$ $(n \in \N, \gamma \in \Gamma n)$, in the sense that $\llb t_{n\gamma} \rrb^A_n = \llb u_{n\gamma} \rrb^A_n:A^n \rightarrow A$.\qed
\end{egg_sub}

\noindent Our objective is now to show that if $(\T,E)$ is a system of $\J$-ary equations presenting the $\J$-ary $\V$-monad $\T/E$ by \ref{everypresentationpresentsJarymonad}, then $(\T/E)\Alg \cong (\T,E)\Alg$ in $\V\CAT/\C$. To facilitate this, we employ the following notation, generalizing ideas from \cite[\S 4]{KellyPower}:

\begin{para_sub}\label{anglebrackets}
We may regard objects $C, D \in \ob\C$ also as $\V$-functors $C, D : \II \to \C$ from the unit $\V$-category $\II$.  The right Kan extension of $D$ along $C$ is then the $\V$-functor $\Ran_C D = [\C(-,C),D]:\C \rightarrow \C$, whose restriction along $j$ is therefore $\Ran_C D \circ j = [\C(j-,C),D]$.  We denote the left Kan extension of $\Ran_C D \circ j$ along $j$ by
$$\la C, D\ra = \Lan_j\left(\Ran_CD \circ j\right)\;:\;\C \longrightarrow \C\;,$$
so that $\la C,D\ra$ is $\J$-ary by \ref{charns_jary}. Since $j$ is fully faithful, the restriction of $\la C,D\ra$ to $\J$ is $[\C(j-,C),D]$.  For each $C \in \ob\C$ we obtain a functor $\la C, -\ra : \C_0 \to \End_{\underJ}(\C)$.  Let us also write $\ev^j : \End_{\underJ}(\C) \times \C_0 \to \C_0$ to denote the functor obtained as a restriction of the evaluation functor $\ev:\End(\C) \times \C_0 \rightarrow \C_0$, so that for each $C \in \ob\C$ we obtain a functor $\ev^j(-,C):\End_{\underJ}(\C) \rightarrow \C_0$. \qed
\end{para_sub}

\begin{prop_sub}
\label{anglebracketsadjunction}
For each $C \in \ob\C$, we have $\ev^j(-,C) \dashv \la C, -\ra : \C_0 \to \End_{\underJ}(\C)$.
\end{prop_sub}

\begin{proof}
By \ref{eleuthericequivalence}, the restriction $\V$-functor $j^*:\End_{\underJ}(\C) \rightarrow \V\CAT(\J,\C)$ is an equivalence, with pseudo-inverse $\Lan_j'$ given by left Kan extension along $j$.  Identifying $\C_0$ with $\V\CAT(\II, \C)$, the functor $\la C, - \ra : \C_0 \to \End_{\underJ}(\C)$ is isomorphic to the composite functor
$$\C_0 \xrightarrow{\Ran_C(-)} \End(\C) \xrightarrow{(-) \circ j} \V\CAT(\J, \C) \xrightarrow{\Lan'_j} \End_{\underJ}(\C)$$
and therefore has a left adjoint $\End_{\underJ}(\C) \xrightarrow{j^*} \V\CAT(\J, \C) \xrightarrow{\Lan_j} \End(\C) \xrightarrow{\ev(-,C)} \C_0$. But this left adjoint is isomorphic to $\ev^j(-,C)$ because $\Lan_j \circ j^*$ is isomorphic to the inclusion $\End_{\underJ}(\C) \hookrightarrow \End(\C)$.
\end{proof}

\begin{para_sub}\label{transpose_lrlangle}
Given objects $C,D \in \ob\C$, a $\J$-ary $\V$-endofunctor $H$ on $\C$, and a morphism $f:HC \to D$ in $\C$, let us write $f^\sharp:H \rightarrow \la C,D\ra$ to denote the transpose of $f$ under the adjunction in \ref{anglebracketsadjunction}.  For each $J \in \ob\J$, the component $f^\sharp_J:HJ \to \la C,D\ra J = [\C(J,C),D]$ is the transpose of the composite $\C(J,C) \xrightarrow{H_{JC}} \C(HJ,HC) \xrightarrow{\C(HJ,f)} \C(HJ,D)$. \qed
\end{para_sub}

\begin{prop_sub}
\label{anglebracketsmonad}
The functor $\ev^j : \End_{\underJ}(\C) \times \C_0 \to \C_0$ is a right-closed strict action \cite[2.2]{BJK} of the strict monoidal category $\End_{\underJ}(\C)$ on the category $\C_0$. For each $A \in \ob\C$, the $\J$-ary $\V$-endofunctor $\la A, A \ra$ underlies a $\V$-monad, and morphisms $\T \to \la A, A\ra$ in $\Mnd_{\underJ}(\C)$ are in bijective correspondence with $\T$-algebra structures on $A$, naturally in $\T \in \Mnd_{\underJ}(\C)$.
\end{prop_sub}

\begin{proof}
It is clear that $\ev^j$ is a strict action, which is right-closed by \ref{anglebracketsadjunction}. We then deduce from \cite[2.2]{BJK} that for every $A \in \ob\C$, the object $\la A, A\ra$ of $\End_{\underJ}(\C)$ underlies a monoid in $\End_{\underJ}(\C)$, and that $\Mnd_{\underJ}(\C)(\T, \la A, A\ra) \cong \mathsf{Act}(\T, A)$ naturally in $\T \in \Mnd_{\underJ}(\C)$, where $\mathsf{Act}(\T, A)$ is the set of $\T$-algebra structures on $A$.
\end{proof} 

\begin{para_sub}\label{interp_mnd_mor}
Let $(A,a)$ be a $\T$-algebra for a $\J$-ary $\V$-monad $\T = (T,\eta,\mu)$.  Then, by \ref{anglebracketsmonad},the $\T$-algebra structure $a:TA \to A$ corresponds to a morphism of $\V$-monads
$$a^\sharp:\T \longrightarrow \la A,A\ra$$
obtained as the transpose of $a$ under the adjunction of \ref{anglebracketsadjunction}.  Hence, in view of \ref{interp_hom} and \ref{transpose_lrlangle}, the component of $a^\sharp$ at each object $J$ of $\J$ is the interpretation morphism $a^\sharp_J = \llb-\rrb^A_J:TJ \to \la A,A\ra J = [\C(J, A), A]$ of \ref{interp_hom}.  For this reason, we denote both $a^\sharp$ and its underlying $\J$-signature morphism as
$$\llb-\rrb^A :\T \longrightarrow \la A,A\ra\;\;\;\;\;\;\;\;\llb-\rrb^A = \U(a^\sharp):\U\T \longrightarrow \U\la A,A\ra$$
and refer to these both as the \textbf{$\J$-ary interpretation morphisms} for $(A,a)$.

Given also a morphism of $\J$-ary $\V$-monads $\tau:\mathbb{S} \rightarrow \T$, we write
\begin{equation}\label{eq:interp_mnd_mor}\llb \tau\rrb^A := \llb - \rrb^A \circ \tau\;:\;\mathbb{S} \longrightarrow \la A,A\ra.\end{equation}
Similarly, given a morphism of $\J$-signatures $t:\Gamma \rightarrow \T$, with the abuse of notation of \S \ref{Jpresentations}, we write
\begin{equation}\label{eq:interp_sig}\llb t\rrb^A := \llb - \rrb^A \circ t\;:\;\Gamma \longrightarrow \la A,A\ra.\end{equation}
\end{para_sub}

\begin{prop_sub}\label{charns_te_alg}
Let $E = (\terml, \termr : \Gamma \rightrightarrows \T)$ be a system of $\J$-ary equations over a $\J$-ary $\V$-monad $\T$, and let $(A,a)$ be a $\T$-algebra.  The following are equivalent: (1) $(A,a)$ satisfies $E$; (2) $\llb t \rrb^A = \llb u \rrb^A:\Gamma \to \la A,A\ra$ with the notation of \eqref{eq:interp_sig}; (3) $\bigl\llb \termlbar\bigr\rrb^A = \bigl\llb \termrbar\bigr\rrb^A:\T_\Gamma \to \langle A,A\rangle$ with the notations of \ref{tbar} and \eqref{eq:interp_mnd_mor}; (4) $a \circ \termlbar_A = a \circ \termrbar_A:T_\Gamma A \to A$.
\end{prop_sub}
\begin{proof}
The equivalence of (1)-(3) is immediate from \ref{interp_mnd_mor} (in view of \ref{freemonadvariant}), while the equivalence of (3) and (4) follows from the adjointness in \ref{anglebracketsadjunction}.
\end{proof}

\noindent In order to use \ref{charns_te_alg} to obtain an isomorphism $(\T,E)\Alg \cong (\T/E)\Alg$, we shall require the following:

\begin{lem_sub}
\label{equalizerlem}
Let
$\begin{tikzcd}
	{(\X, Q)} & {(\Y, R)} & {(\Z, S)}
	\arrow["M", tail, from=1-1, to=1-2]
	\arrow["F", shift left=1.1, from=1-2, to=1-3]
	\arrow["G"', shift right=1.1, from=1-2, to=1-3]
\end{tikzcd}$
be an equalizer in $\V\CAT/\C$ (with the notation of \S \ref{background}). If $S$ is faithful, then $M$ is fully faithful (in addition to being injective on objects). 
\end{lem_sub}

\begin{proof}
Letting $X, X' \in \ob\X$, we shall show that $M_{XX'}:\X(X, X') \to \Y(Y, Y')$ is an isomorphism, where $Y = MX$ and $Y' = MX'$.  Letting  $Z = FY = GY$ and $Z' = FY' = GY'$, we have $S_{ZZ'} \circ F_{YY'} = (SF)_{YY'} = R_{YY'} = (SG)_{YY'} = S_{ZZ'} \circ G_{YY'}$, so that $F_{YY'} = G_{YY'}$ since $S_{ZZ'}$ is a monomorphism (as $S$ is faithful). But since equalizers in $\V\CAT/\C$ are formed as in $\V\CAT$, $M_{XX'}$ is an equalizer of $F_{YY'},G_{YY'}$ in $\V$, so $M_{XX'}$ is an isomorphism.
\end{proof}

\begin{theo_sub}
\label{Palgebrasisomorphism}
Let $\T/E$ be the $\J$-ary $\V$-monad presented by a system of $\J$-ary equations $(\T,E)$, where $E = (\terml, \termr: \Gamma \rightrightarrows \T)$. Then $(\T/E)\Alg \cong (\T,E)\Alg$ in $\V\CAT/\C$. 
\end{theo_sub}

\begin{proof}
Since $\T/E$ is an algebraic coequalizer $\T_\Gamma \overset{\termlbar,\termrbar}{\rightrightarrows} \T \overset{q}{\twoheadrightarrow} \T/E$, we have an equalizer diagram
\[\begin{tikzcd}
	{(\T/E)\Alg} && \T\Alg && {\T_\Gamma\Alg}
	\arrow["q^*", tail, from=1-1, to=1-3]
	\arrow["{\overline{\terml}^*}", shift left=1, from=1-3, to=1-5]
	\arrow["{\overline{\termr}^*}"', shift right=2, from=1-3, to=1-5]
\end{tikzcd}\]
in $\V\CAT/\C$.  Since $U^{\T_\Gamma} : \T_\Gamma\Alg \to \C$ is faithful, it follows by \ref{equalizerlem} that $q^*$ is fully faithful, in addition to being injective on objects.  Hence, $(\T/E)\Alg$ is isomorphic to the full sub-$\V$-category of $\T\Alg$ consisting of those $\T$-algebras $(A,a)$ with $\termlbar^*(A,a) = \termrbar^*(A,a)$, i.e. with $a \circ \termlbar_A = a \circ \termrbar_A$.  But by \ref{charns_te_alg}, the latter are precisely the $(\T,E)$-algebras.
\end{proof}

\begin{theo_sub}
\label{Palgebrasisomorphismcor}
Let $P = (\Sigma, E)$ be a $\J$-presentation, and let $\T_P$ be the $\J$-ary $\V$-monad presented by $P$. Then $\T_P\Alg \cong P\Alg$ in $\V\CAT/\C$. 
\end{theo_sub}
\begin{proof}
$\T_P\Alg = (\T_\Sigma/E)\Alg \cong (\T_\Sigma,E)\Alg \cong P\Alg$ in $\V\CAT/\C$ by \ref{Palgebrasisomorphism} and \ref{palg_iso_to_tsigmaealg}.
\end{proof}

\section{Some examples of \texorpdfstring{$\J$}{J}-presentations}
\label{firstexamples}

We now discuss some examples of $\J$-presentations and their algebras, beyond the familiar finitary presentations of classical universal algebra (\ref{presentationexample}) and the examples in the finitely presentable setting discussed in \cite{Robinson}. As mentioned in the introduction, in a forthcoming paper \cite{Pres2} we shall give many more examples of $\J$-presentations after first providing therein a versatile and `user-friendly' method for constructing them.     
 
\subsection{Presentations of \texorpdfstring{$\V$}{V}-categories by generators and relations}\label{representablesexample}

We now consider what signatures, presentations, and their algebras amount to in the context of the bounded and eleutheric subcategory of arities $\y_X:X \rightarrow [X,\V] = \V^X$ of \ref{presheaf_as_em}, given by $\y_X(x) = X(x,-)$, where $X$ is a set, regarded also as a discrete $\V$-category.  By definition, a $\y_X$-signature is a $\V$-functor $\Sigma : X \to \V^X$, or equivalently, a $\V$-functor $\Sigma : X \tensor X \to \V$, where we write $\Sigma(x,y) = (\Sigma x)y$, noting that $X \otimes X$ is the discrete $\V$-category on $X \times X$.  Hence $\y_X$-signatures may be identified with \emph{$\V$-matrices from $X$ to $X$} (\ref{presheaf_as_em}) or \emph{$\V$-graphs with object set $X$} (see \cite{WolffVcat}), i.e. families of objects $\Sigma = (\Sigma(x,y))_{(x,y) \in X\times X}$ of $\V$.

The category $\Sig_{\y_X}(\V^X)$ of $\y_X$-signatures may thus be identified with the category $\V\Graph(X) = \VMat(X,X)$ of $\V$-graphs with object set $X$ and identity-on-objects $\V$-graph morphisms.  If $\Sigma : X \tensor X \to \V$ is a $\y_X$-signature, i.e. a $\V$-graph with object set $X$, then a $\Sigma$-algebra is (by definition) a $\V$-functor $A : X \to \V$ equipped with structural morphisms $\alpha_x: \V^X(\y_X(x),A)\otimes \Sigma(x,-) \rightarrow A$ in $\V^X$ $(x \in X)$. By the Yoneda lemma, a $\Sigma$-algebra is therefore given by a family of objects $A(x)$ $(x \in X)$ of $\V$ and a family of morphisms $\alpha_{xy}:\Sigma(x,y) \rightarrow \V(A(x),A(y))$ $(x,y \in X)$.  Hence a $\Sigma$-algebra is equivalently a morphism of $\V$-graphs $A : \Sigma \to \V$ from $\Sigma$ to the underlying $\V$-graph of the $\V$-category $\V$.

Recall from \ref{presheaf_as_em} that the category $\Mnd_{\y_X}(\V^X)$ of $\y_X$-ary $\V$-monads on $\V^X$ is equivalent to the category $\V\Cat(X)$ of $\V$-categories with object set $X$. The free $\y_X$-ary $\V$-monad $\T_\Sigma$ on a $\y_X$-signature $\Sigma$, which exists by \ref{freemonadalgebrasaresignaturealgebras}, may therefore be identified with the free $\V$-category $\scrT_\Sigma$ on the $\V$-graph $\Sigma$, which was first proved to exist in \cite[2.2]{WolffVcat}. The monadic adjunction between $\Mnd_{\y_X}(\V^X)$ and $\Sig_{\y_X}(\V^X)$ (see \ref{Uismonadic}) can now be identified with a monadic adjunction between $\V\Cat(X)$ and $\V\Graph(X)$, which is a restriction of the monadic adjunction between $\V\Cat$ and $\V\Graph$ that was first established in \cite[2.13]{WolffVcat}.  By \ref{presheaf_as_em} and \ref{freemonadalgebrasaresignaturealgebras}, we have isomorphisms $[\scrT_\Sigma,\V] \cong \T_\Sigma\Alg \cong \Sigma\Alg$ that enable us to identify $\V$-functors from $\scrT_\Sigma$ to $\V$, $\T_\Sigma$-algebras, and $\V$-graph morphisms from $\Sigma$ to $\V$.

By the foregoing, a $\y_X$-presentation may now be identified with a parallel pair $P = (\terml, \termr : \Gamma \rightrightarrows \scrT_\Sigma)$ of identity-on-objects $\V$-graph morphisms, for a $\V$-graph $\Sigma$ with object set $X$, while a system of $\y_X$-ary equations may be identified with a parallel pair $E = (\terml, \termr : \Gamma \rightrightarrows \scrT)$ of identity-on-objects $\V$-graph morphisms valued in the underlying $\V$-graph of a $\V$-category $\scrT$ with object set $X$; these systems of $\y_X$-ary equations were considered by Wolff in \cite[2.5]{WolffVcat} under the name \emph{pre-$\V$-congruences}. In a $\y_X$-presentation, one has for each pair $x, y \in X$ a parallel pair of $\V$-morphisms $\terml_{xy}, \termr_{xy} : \Gamma(x, y) \rightrightarrows \scrT_\Sigma(x, y)$, where $\scrT_\Sigma(x, y)$ has an explicit (combinatorial) description given in \cite[2.2]{WolffVcat}. In the $\V = \Set$ case, $\y_X$-presentations are the usual presentations of categories by generators and relations, where the generators $\sigma:x \rightarrow y$ are the edges in the given graph $\Sigma$ ($x,y \in X, \sigma \in \Sigma(x,y)$) and the relations $t_{xy}(\gamma) \doteq u_{xy}(\gamma)$ $(x,y \in X,\gamma \in \Gamma(x,y))$ are pairs of arrows $t_{xy}(\gamma), u_{xy}(\gamma): x \rightrightarrows y$ in the free category on $\Sigma$, i.e. paths in $\Sigma$, which will become equal in the quotient category $\scrT_P$. The general $\V$-based case may now be understood as an abstraction of the $\Set$-based case, so that a $\y_X$-presentation is a presentation of a $\V$-category with object set $X$ by generators and relations, the \textit{generators} being provided by the $\V$-graph $\Sigma$, and the \textit{relations} by the parallel pair $\terml, \termr$. Given a $\y_X$-presentation $P = (\terml, \termr : \Gamma \rightrightarrows \scrT_\Sigma)$, a $P$-algebra can now be identified with a $\V$-graph morphism $A : \Sigma \to \V$ with the property that the unique $\V$-functor $A^\sharp : \scrT_\Sigma \to \V$ extending $A$ coequalizes $\terml$ and $\termr$.

As a special case, if $X$ is a singleton set $\{*\}$, then $\y_X$-ary presentations may be described as presentations of monoids in the monoidal category $(\V_0,\otimes,I)$ by generators and relations.

\subsection{Strongly finitary \texorpdfstring{$\V$}{V}-monads on cartesian closed topological categories}\label{sf_vmds_top_ccc}

In this section, we suppose that $\V$ is a cartesian closed topological category over $\Set$ as in \ref{free_sf_vmnds_top_ccc} and, using the notations and results from that section, we discuss $\SF(\V)$-presentations, i.e. presentations of strongly finitary $\V$-monads, providing some specific examples in \ref{mod_rigs} and \ref{raff}.  Given an object $X$ of $\V$, we also write $X$ to denote the underlying set $|X|$ of $X$.  In particular, given objects $X,Y$ of $\V$, the elements of (the underlying set of) the internal hom $\V(X, Y)$ are the morphisms $f:X \rightarrow Y$ in $\V_0$, which are certain functions $f:X \rightarrow Y$.  The objects of $\SF(\V)$ are the finite cardinals $n \in \N$, each of which we regard also as the discrete object of $\V_0$ on $n$ \cite[21.12]{AHS}, or equally the $n$th copower of $1$ in $\V_0$.

\begin{defn_sub}
\label{underlyingpresentation}
Let $P = \left(\terml, \termr : \Gamma \rightrightarrows \T_\Sigma\right)$ be an $\SF(\V)$-presentation. By  \ref{topologicalliftingprop}, the underlying set of $T_\Sigma(n)$ for each finite cardinal $n$ is the set $T_{|\Sigma|}(n)$ of all terms with variables in $n$ over the underlying $\FinCard$-signature $|\Sigma|$.  Hence, the functions underlying $t_n,u_n:\Gamma(n) \rightrightarrows T_\Sigma(n)$ are morphisms $|t_n|,|u_n|:|\Gamma(n)| \rightrightarrows T_{|\Sigma|}(n)$ in $\Set$ $(n \in \N)$ and so constitute a $\FinCard$-presentation $|P| = \left(|\terml|, |\termr| : |\Gamma| \rightrightarrows \T_{|\Sigma|}\right)$ that we call the \textbf{underlying $\FinCard$-presentation} of $P$. \qed
\end{defn_sub}

\begin{para_sub}\label{interp_in_underlying}
Let $(A,a)$ be a $\Sigma$-algebra for an $\SF(\V)$-signature $\Sigma$.  Recalling that $T_\Sigma A$ is the free $\Sigma$-algebra on the object $A$ of $\V$, and writing $\eta^\Sigma_A:A \rightarrow T_\Sigma A$ for the unit morphism, let us write $\alpha:T_\Sigma A \rightarrow A$ for the $\T_\Sigma$-algebra structure carried by $A$, i.e. the unique $\Sigma$-homomorphism $\alpha$ such that $\alpha \circ \eta^\Sigma_A = 1_A$.  By \ref{topologicalliftingprop}, $|\eta^\Sigma_A| = \eta^{|\Sigma|}_{|A|}:|A| \rightarrow |T_\Sigma A| = T_{|\Sigma|}|A|$, so $|\alpha|$ is precisely the $\T_{|\Sigma|}$-algebra structure carried by the $|\Sigma|$-algebra $(|A|,|a|)$ underlying $(A,a)$ (\S \ref{free_sf_vmnds_top_ccc}).

For each finite cardinal $n$, the interpretation map $\llb-\rrb_n^A:T_\Sigma(n) \rightarrow \V(A^n,A)$ sends each $t \in T_\Sigma(n)$ to the map $\llb t\rrb_n^A:A^n \rightarrow A$ given by $\llb t\rrb_n^A(x) = \alpha((T_\Sigma x)(t))$ for all $x \in A^n = \V(n,A)$.  Also, since the $\T_{|\Sigma|}$-algebra structure carried by the $|\Sigma|$-algebra $(|A|,|a|)$ is precisely $|\alpha|$, its interpretation map $\llb-\rrb_n^{|A|}:T_{|\Sigma|}(n) \rightarrow \Set(|A|^n,|A|)$ is given by a similar formula, so in view of \ref{topologicalliftingprop} we find that $\llb t\rrb_n^{|A|}:|A|^n \rightarrow |A|$ is the function underlying $\llb t\rrb_n^A:A^n \rightarrow A$, for all $t \in T_{|\Sigma|}(n) = |T_{\Sigma}(n)|$. \qed
\end{para_sub}

\begin{lem_sub}
\label{presentationlem}
Let $P = (\Sigma,E)$ be an $\SF(\V)$-presentation.  Then there is a pullback square
\begin{equation}\label{palg_pb_square}
\xymatrix{
{P\Alg_0\;} \ar[d]_{V^P} \ar@{^(->}[r] & \Sigma\Alg_0 \ar[d]^{V^\Sigma}\\
{|P|\Alg\;} \ar@{^(->}[r] & |\Sigma|\Alg\\
}
\end{equation}
in $\mathsf{CAT}$.  In particular, a $P$-algebra is precisely a $\Sigma$-algebra whose underlying $|\Sigma|$-algebra is a $|P|$-algebra.  Also, the hom-object $P\Alg(A,B) = \Sigma\Alg(A,B)$ between $P$-algebras $A$ and $B$ is the subspace $|P|\Alg(|A|,|B|) \cap \V_0(A,B) \hookrightarrow \V(A,B)$ (with the $V$-initial structure).
\end{lem_sub}
\begin{proof}
Since the top and bottom sides are full subcategory inclusions, the first claim follows from the second, which follows straightforwardly from \ref{interp_in_underlying}.  The third claim follows from \ref{vcat_sig_algs}, since limits in $\V_0$ are formed by equipping the limit in $\Set$ with the $V$-initial structure \cite[21.15]{AHS}.
\end{proof}

\begin{rmk_sub}\label{eqn_discrete}
Lemma \ref{presentationlem} entails that $P\Alg$ depends only on $\Sigma$ and $|P|$, in an evident sense.  Hence, in this setting, it suffices to consider $\SF(\V)$-presentations $P = \left(\terml, \termr : \Gamma \rightrightarrows \T_\Sigma\right)$ that are \textbf{equation-discrete} in the sense that $\Gamma$ is \textbf{discrete}, i.e., each $\Gamma(n)$ is discrete $(n \in \N)$ \cite[8.1]{AHS}.  Indeed, in view of \cite[21.12]{AHS}, we can associate to each $\SF(\V)$-presentation $P = (\Sigma,E)$ an equation-discrete $\SF(\V)$-presentation $P'$ over $\Sigma$ with $P'\Alg = P\Alg$ as objects of $\V\CAT/\C$, so that $\T_{P'} \cong \T_P$ as well. \qed
\end{rmk_sub}

\noindent In \cite[5.1]{ADV}, Ad\'amek, Dost\'al, and Velebil show that every strongly finitary enriched monad on the category of posets is a lifting of a finitary monad on $\Set$.  The following provides an analogous result for cartesian closed topological categories over $\Set$:

\begin{theo_sub}
Every strongly finitary $\V$-monad $\T$ on $\V$ is a strict lifting of a finitary monad on $\Set$.  In particular, given any $\SF(\V)$-presentation $P = (\Sigma,E)$, the strongly finitary $\V$-monad $\T_P$ presented by $P$ is a non-strict lifting of the finitary monad $\T_{|P|}$ presented by $|P|$, while we may construct $\T_P$ in such a way that it is a strict lifting of $\T_{|P|}$.
\end{theo_sub}
\begin{proof}
Regarding the second claim, $V^\Sigma:\Sigma\Alg_0 \rightarrow |\Sigma|\Alg$ is topological by 
\ref{wtop}, so since \eqref{palg_pb_square} is a pullback whose lower edge is fully faithful, it follows that $V^P$ is topological and that the upper edge sends $V^P$-initial sources to $V^\Sigma$-initial sources.  Consequently, by \ref{wtop}, the rectangle obtained by pasting \eqref{palg_pb_square} onto the square \eqref{eq:tlsq} satisfies the hypotheses of Wyler's taut lift theorem \cite[21.28]{AHS}, and the second claim follows.  Now given instead an arbitrary $\SF(\V)$-ary $\V$-monad $\T = (T,\eta,\mu)$ on $\V$, there is an $\SF(\V)$-presentation $P$ with $\T \cong \T_P$ by \ref{presentationcor}, and without loss of generality $\T_P = (T_P,\eta^P,\mu^P)$ is a strict lifting of the finitary monad $\T_{|P|}$.  Hence $\T_{|P|} = (VT_P D,V\eta^P D, V\mu^P D)$ by \ref{para_lifting}, where $D$ is the left adjoint section of $V$ and we do not distinguish notationally between $\T_P$ and its underlying ordinary monad. Hence, since $\T \cong \T_P$, it follows that $\T$ is a strict lifting of the monad $(VTD,V\eta D,V\mu D)$, which is isomorphic to $\T_{|P|}$ and so is finitary.
\end{proof}
 
\noindent We now use \ref{topologicalliftingprop} and \ref{presentationlem} to provide some examples of $\SF(\V)$-presentations.

\begin{egg_sub}[\textbf{Modules for internal rigs}]\label{mod_rigs}
Let $\left(R, +^R, \cdot^R, 0^R, 1^R\right)$ be an \emph{internal rig} in $\V$ (see e.g. \cite[2.7]{functional}), so that $|R|$ is a rig (i.e. a unital semiring) in $\Set$. We now construct an $\SF(\V)$-presentation whose algebras are \emph{internal left} $R$\emph{-modules} in $\V$ (in the sense of, e.g., \cite[2.7]{functional}).  Note that, in the present case, where $\V$ is a cartesian closed topological category over $\Set$, an internal left $R$-module is equivalently an object $M$ of $\V$ equipped with $\V$-morphisms $+^M : M^2 \to M$, $0^M : 1 = M^0 \to M$, and $\bullet^M : R \times M \to M$ making $|M|$ a left $|R|$-module (meaning that $+^M$ and $0^M$ make $|M|$ a commutative monoid and $\bullet^M$ is an associative and unital action of $|R|$ on $|M|$ that preserves $+$ and $0$ in each variable separately).  For our $\SF(\V)$-signature $\Sigma$, we thus take $\Sigma(0) := 1$, $\Sigma(1) := R$, and $\Sigma(2) := 1$, and we take $\Sigma(n) := 0$ for all $n \geq 3$.  Hence the underlying finitary signature $|\Sigma|$ has one constant symbol $0$, one binary operation symbol $+$, and a unary operation symbol $r$ for each $r \in |R|$.

We now define an equation-discrete $\SF(\V)$-presentation $P$ over $\Sigma$ (\ref{eqn_discrete}) by first defining a discrete $\SF(\V)$-signature $\Gamma$, as follows. We shall define $\Gamma(0)$, $\Gamma(1)$, $\Gamma(2),\Gamma(3)$ to be discrete objects of $\V$ whose elements are certain formal equations over $|\Sigma|$ (i.e. pairs of terms, \ref{presentationexample}) with variables in the finite cardinals $0,1,2,3$, respectively, and for this purpose it will be convenient to write the elements of these finite cardinals as follows: $0 = \varnothing$, $1 = \{x\}$, $2 = \{x,y\}$, $3 = \{x,y,z\}$.  We let the elements of $\Gamma(0)$ be the formal equations $r0 \doteq 0$ with $r \in R$, with the notation of \ref{presentationexample}. We let $\Gamma(1)$ consist of the formal equations $0 + x \doteq x$, $x + 0 \doteq x$, $1^R x \doteq x$, and $0^R x \doteq 0$, together with the formal equations $r(sx) \doteq (r \cdot^R s)x$ and
$(r +^R s)x \doteq rx + sx$ associated to the various $r, s \in R$.  We let $\Gamma(2)$ consist of the formal equations $x + y \doteq y + x$ and $r(x + y)\doteq rx + ry$ with $r \in R$.  We let the unique element of $\Gamma(3)$ be the formal equation $x + (y + z) \doteq (x + y) + z$. Lastly, we let $\Gamma(n) = 0$ for all $n \geq 4$.

The resulting discrete $\SF(\V)$-signature $\Gamma$ has the property that, for each $n \in \N$, the underlying set of $\Gamma(n)$ is a subset of (that of) $T_{\Sigma}(n) \times T_{\Sigma}(n)$, so that since $\Gamma(n)$ is discrete there are morphisms $\terml_n, \termr_n : \Gamma(n) \rightrightarrows T_{\Sigma}(n)$ in $\V$ given as the projections. Thus we obtain an equation-discrete $\SF(\V)$-presentation $P = \left(\terml, \termr : \Gamma \rightrightarrows \T_\Sigma\right)$ over $\Sigma$ for which $|P|\Alg$ is precisely the category $|R|\Mod$ of left $|R|$-modules (in $\Set$), so that by \ref{presentationlem} it follows that $P\Alg$ is precisely the $\V$-category $R\Mod$ of internal left $R$-modules \cite[6.4.2]{functional}. \qed
\end{egg_sub}

\begin{egg_sub}[\textbf{Internal $R$-affine spaces and convex spaces}]\label{raff}
Given an internal rig $R$ in $\V$ (\ref{mod_rigs}), we now define an $\SF(\V)$-presentation whose algebras are \emph{internal (left)} $R$\emph{-affine spaces} in $\V$. To define these, recall from \cite[6.4.5]{functional} that there is a $\V$-category $\Mat_R$ whose objects are the natural numbers and whose hom-objects are $\Mat_R(n, m) := R^{m \times n}$ ($n, m \in \N$), with composition given by internal matrix multiplication (see \cite[6.4.4]{functional}) and identity morphisms given by internal identity matrices. Specializing \cite[8.4]{functional} to the present context where $\V$ is topological over $\Set$, we write $\Mat^\aff_R$ to denote the (non-full) sub-$\V$-category of $\Mat_R$ with the same objects but with hom-objects obtained as subspaces $\Mat^\aff_R(n, m) \hookrightarrow R^{m \times n}$ (with the $V$-initial structure) consisting of all $m\times n$-matrices $r \in R^{m \times n}$ such that $\sum_{j=1}^n r_{ij} =  1$ for each $i \in \left\{1, \ldots, m\right\}$.  These $\V$-categories $\Mat_R$ and $\Mat^\aff_R$ are $\V$-enriched algebraic theories for the subcategory of arities $\SF(\V)$ (\cite[4.2]{EAT}, \cite[6.4.6, 8.4]{functional}), and $\Mat^\aff_R$ is the \textit{affine core} of $\Mat_R$ \cite[8.3]{functional}.  In particular, each object $n$ of $\Mat^\aff_R$ is an $n$th power of the object $1$ in $\Mat^\aff_R$, where the projections $\pi_i:n \rightarrow 1$ ($1 \leq i \leq n$) are the standard basis vectors $\pi_i = b_i \in R^{1 \times n} = R^n$.

By definition \cite[8.5]{functional}, an \emph{internal (left) $R$-affine space} in $\V$ is a normal $\Mat_R^\aff$-algebra, i.e. an object $A$ of $\V$ equipped with a $\V$-functor $\Mat_R^\aff \to \V$ that, for each $n \in \N$, sends the $n$th power cone $\pi_i:n \rightarrow 1$ ($1 \leq i \leq n$) to the usual $n$th power cone $A^n \to A$ in $\V$.  Therefore, writing $\Delta(n, m) := \Mat^\aff_R(n, m)$ for all $n, m \in \N$, 
an internal $R$-affine space is equivalently given by an object $A$ of $\V$ equipped with structural morphisms $\Delta(n, m) \times A^n \to A^m$ ($n, m \in \N$), whose value at $(r, a) \in \Delta(n, m) \times A^n$ we write as $ra$, satisfying the following equations:
\begin{enumerate}
\item $r(sa) = (rs)a$ for all $n, m \in \N$, $s \in \Delta(n, m), r \in \Delta(m, 1)$, and $a \in A^n$;

\item $I_na = a$ for all $n \in \N$ and $a \in A^n$, where $I_n \in \Delta(n, n) \subseteq R^{n \times n}$ is the identity matrix;

\item $(ra)_i = r_ia$ for all $n \in \N$, $r \in \Delta(n, m), a \in A^n$, and $1 \leq i \leq m$, where $r_i \in \Delta(n, 1) \subseteq R^{1 \times n} = R^n$ is the $i$th row of $r$. 
\end{enumerate} 

We now define an $\SF(\V)$-signature $\Delta$ consisting of the subspaces $\Delta(n) := \Delta(n, 1)$ of $R^{1 \times n} = R^n$ with $n \in \N$, so that an internal $R$-affine space in $\V$ is equivalently a $\Delta$-algebra $(A, \alpha)$ that satisfies the following equations, where we write the action of the morphisms $\alpha_n : A^n \times \Delta(n) \to A$ ($n \in \N$) as $(a,r) \mapsto ra$:
\begin{enumerate}
\item $r(s_1a, \ldots, s_ma) = (rs)a$ for all $n, m \in \N, s \in \Delta(n, m), r \in \Delta(m, 1) = \Delta m$, and $a \in A^n$;

\item $b_ia = a_i$ for all $n \in \N, a \in A^n$, and $1 \leq i \leq n$, where $b_i \in \Delta(n) \subseteq R^n$ is the $i$th standard basis vector. 
\end{enumerate}
The reason for the use of the term \textit{$R$-affine space} here is that we normally write $ra$ as $\sum_{i=1}^n r_ia_i$, so that an internal (left) $R$-affine space is an object of $\V$ in which we can take \textit{(left) affine combinations} (i.e. left $R$-linear combinations whose coefficients add up to $1$).  Note that in the special case where $R = [0,\infty)$ and $\V$ is the category of sets (resp. the category of \textit{convergence spaces}) we recover the notion of \textit{convex space} \cite{Meng,Lu:CvxAffCmt} (resp. \textit{convergence convex space} \cite[8.10]{functional}).

It therefore follows that internal $R$-affine spaces are the $P$-algebras for an equation-discrete $\SF(\V)$-presentation $P = (t, u : \Gamma \rightrightarrows \T_\Delta)$ over $\Delta$ defined as follows, noting that we shall of course regard each $r \in \Delta(n) \subseteq R^n$ as an $n$-ary operation symbol in $|\Delta|$. For each finite cardinal $n$, let us write the elements of $n$ as $x_1,...,x_n$, and define $\Gamma(n)$ as the discrete object of $\V$ whose elements are the following formal equations over $|\Delta|$ with variables in $n$:
$$
r(s_1(x_1, \ldots, x_n), \ldots, s_m(x_1, \ldots, x_n)) \doteq (rs)(x_1, \ldots, x_n)\;\;\;\;\; (m \in \N, s \in \Delta(n, m), r \in \Delta(m, 1))
$$
\begin{equation}\label{eq:raff_eqns}
b_i(x_1, \ldots, x_n) \doteq x_i\;\;\;\;\;\;\;\;(1 \leq i \leq n).
\end{equation}
So the underlying set of $\Gamma(n)$ is a subset of (that of) $T_\Delta(n) \times T_\Delta(n)$, and since $\Gamma(n)$ is discrete there are morphisms $t_n, u_n : \Gamma(n) \rightrightarrows T_\Delta(n)$ in $\V$ given as the projections. We thus obtain an equation-discrete $\SF(\V)$-presentation $P = (t, u : \Gamma \rightrightarrows \T_\Delta)$ over $\Delta$ for which $|P|\Alg$ is isomorphic to the category of $|R|$-affine spaces (in $\Set$), so by \ref{presentationlem} we find that $P\Alg$ is isomorphic to the $\V$-category $R\text{-}\mathsf{Aff}$ of internal $R$-affine spaces \cite[8.5]{functional}.

Given $r \in \Delta(n) \subseteq R^n$ and terms $t_1, \ldots, t_n \in T_\Delta(m)$, where $n, m \in \N$, let us now write $\sum_{i=1}^n r_it_i$ as a notation for the term $r(t_1, \ldots, t_n) \in T_\Delta(m)$. Then for each $n \in \N$, the formal equations in \eqref{eq:raff_eqns} may be written as follows, where $I \in R^{n \times n}$ is the identity matrix:
\[ \sum_{i=1}^m r_i \left(\sum_{j=1}^n s_{ij}x_j\right) \doteq \sum_{j=1}^n \left(\sum_{i=1}^m r_is_{ij}\right)x_j\;,\;\;\;\;\;\;\;\;\;\;\sum_{j=1}^n I_{ij}x_j \doteq x_i\;. \tag*{\qed}\]

\end{egg_sub}

\section{Summary of main results}
\label{summary}

For convenience, we summarize the main results of the paper as follows:

\begin{theo}
\label{summarythm}
Let $\V$ be a closed factegory with small limits and colimits. Let $\C$ be a cocomplete $\V$-factegory that is cotensored, suppose either that $\C$ is proper or that $\C$ is \mbox{$\E$-cowellpowered}, and let $j : \J \hookrightarrow \C$ be a subcategory of arities.
\begin{enumerate}[leftmargin=*] 
\item If $\J$ is bounded, then:  
\begin{enumerate}[leftmargin=*]
\item The forgetful functor $\mathcal{W} : \Mnd_{\underJ}(\C) \to \End_{\underJ}(\C)$ is monadic, and $U^H : H\Alg \to \C$ is strictly $\J$-monadic for every $\J$-ary $\V$-endofunctor $H$ on $\C$ (see \ref{mainalgfreetheorem} and \ref{Whasleftadjoint}). 

\item The forgetful functor $\mathcal{U}' : \Mnd(\C) \to \JSig(\C)$ has a left adjoint, and $U^\Sigma : \Sigma\Alg \to \C$ is strictly monadic for every $\J$-signature $\Sigma$ (see \ref{freemonadvariant}).

\item $\Mnd_{\underJ}(\C)$ has small colimits, which are algebraic (see \ref{colimitcategoryofalgebras}).
\end{enumerate} 
\item If $\J$ is bounded and eleutheric, then:
\begin{enumerate}[leftmargin=*]
\item The forgetful functor $\mathcal{U} : \Mnd_{\underJ}(\C) \to \JSig(\C)$ is monadic, and $U^\Sigma : \Sigma\Alg \to \C$ is strictly $\J$-monadic for every $\J$-signature $\Sigma$ (see \ref{freemonadalgebrasaresignaturealgebras} and \ref{Uismonadic}).

\item Every $\J$-ary $\V$-monad on $\C$ has a (canonical) $\J$-presentation (see \ref{presentationcor}).   

\item Every $\J$-presentation $P$ presents a $\J$-ary $\V$-monad $\T_P$ with $\T_P\Alg \cong P\Alg$ in $\V\CAT/\C$ (see  \ref{Palgebrasisomorphismcor}).  More generally, every system of $\J$-ary equations $(\T,E)$ presents a $\J$-ary $\V$-monad $\T/E$ with $(\T/E)\Alg \cong (\T,E)\Alg$ in $\V\CAT/\C$ (see \ref{everypresentationpresentsJarymonad} and  \ref{Palgebrasisomorphism}). \qed
\end{enumerate}
\end{enumerate}
\end{theo}

\noindent In a locally bounded $\V$-category $\C$ over a locally bounded closed category $\V$ (\ref{locallybounded}), every small subcategory of arities in $\C$ is bounded (\ref{generallocallyboundedprop}), so Theorem \ref{summarythm} entails the following corollary, which we can apply in particular to $\C = \V$ itself by \cite[5.8]{locbd}.

\begin{cor}
\label{locallyboundedresults}
Let $\C$ be a locally bounded $\V$-category over a locally bounded closed category $\V$, and let $j : \J \hookrightarrow \C$ be a subcategory of arities.  If $\J$ is small, then statements 1(a)-(c) in \ref{summarythm} hold, while if $\J$ is small and eleutheric then 2(a)-(c) in \ref{summarythm} also hold. \qed
\end{cor}

\begin{egg}
\label{mainresultsexamples}
By \ref{eleuthericexamples} and \ref{boundedexamples}, each of the following classes of examples satisfies the assumptions of \ref{summarythm}(2), so that we can apply the main results of this paper to these examples:
\begin{enumerate}[leftmargin=0pt,labelindent=0pt,itemindent=*,label=(\arabic*)]
\item If $\V$ is locally $\alpha$-presentable as a closed category and $\C$ is a locally $\alpha$-presentable $\V$-category, then we may take $(\E, \M) = (\Iso, \All)$, $(\E_\C, \M_\C) = (\Iso, \All)$, and $\J = \C_\alpha \hookrightarrow \C$.

\item If $\V$ is a $\pi$-category  \cite{BorceuxDay} (e.g. if $\V$ is cartesian closed), then we may take $(\E, \M) = (\Iso, \All)$, $\C = \V$, and $\J = \SF(\V) \hookrightarrow \V$.

\item Equip $\V$ with $(\E, \M) = (\Iso, \All)$, let $\C = \V$, and take $\J = \{I\} \hookrightarrow \V$.

\item Given a small $\V$-category $\A$, equip $\V$ with $(\E, \M) = (\Iso, \All)$, let $\C = [\A, \V]$ with $(\E_\C, \M_\C) = (\Iso, \All)$, and consider the subcategory of arities $\y_\A : \A^\op \hookrightarrow [\A, \V]$. 

\item Let $\V$ be a locally bounded and $\E$-cowellpowered closed category, let $\Phi$ be a locally small class of small weights satisfying Axiom A of \cite{LR}, let $\C = \Phi\Mod(\scrT)$ for a $\Phi$-theory $\scrT$, with the associated proper factorization system on $\C$, and consider the subcategory of arities $\y_\Phi:\scrT^\op \hookrightarrow \C$.

\item Let $\bbD$ be a sound doctrine, suppose that $\V$ is locally $\bbD$-presentable as a $\tensor$-category, equip $\V$ with $(\E, \M) = (\Iso, \All)$, let $\scrT$ be a $\Phi_\bbD$-theory, let $\C = \Phi_\bbD\Mod(\scrT)$ with $(\E_\C, \M_\C) = (\Iso, \All)$, and consider the subcategory of arities $\y_\Phi : \scrT^\op \hookrightarrow \Phi_\bbD\Mod(\scrT)$.

\item As noted in \ref{locallyboundedresults}, \textit{any} small and eleutheric subcategory of arities in a locally bounded $\V$-category $\C$ over a locally bounded closed category $\V$ satisfies the hypotheses of \ref{summarythm}, with respect to the proper factorization systems carried by $\V$ and $\C$, so the main results of this paper are applicable quite generally in such $\V$-categories $\C$. \qed     
\end{enumerate}
\end{egg}

\noindent Note that (5) is a special case of (7), while (5) and (7) entail that, for locally bounded $\V$, the results in this paper specialize to provide a full theory of presentations and algebraic colimits of (i) $\Phi$-accessible $\V$-monads in the general setting of Lack and Rosick\'{y} \cite{LR} and \linebreak (ii) $\J$-ary $\V$-monads on $\C = \V$ for small eleutheric systems of arities $\J \hookrightarrow \V$ \cite{EAT}.  In both these classes of examples, and moreover in (2)-(5) and (7), $\V_0$ and $\C_0$ need not be locally presentable.

\bibliographystyle{amsplain}
\bibliography{mybib}

\end{document}